\documentclass[a4paper,11pt,]{article}

\usepackage[applemac]{inputenc}
\usepackage{amsmath,amsthm,amssymb,amscd,tensor,mathtools,xcolor}
\usepackage[margin=3cm]{geometry} 
\usepackage[T1]{fontenc} 
\usepackage{url}
\usepackage[backref=page]{hyperref}
\usepackage{mathrsfs}
\usepackage{tikz}
\usepackage{mathptmx}
\DeclareMathOperator{\im}{im}
\DeclareMathOperator{\re}{Re}

\DeclareMathOperator{\End}{End}

\DeclareMathOperator{\ev}{ev}

\DeclareMathOperator{\Vect}{Vect}

\DeclareMathOperator{\Rm}{Rm}
\DeclareMathOperator{\Ric}{Ric}

\DeclareMathOperator{\coker}{coker}

\DeclareMathOperator{\ind}{ind}

\DeclareMathOperator{\Diff}{Diff}

\DeclareMathOperator{\SO}{SO}

\DeclareMathOperator{\U}{U}

\DeclareMathOperator{\PSL}{PSL}

\DeclareMathOperator{\Isom}{Isom}

\newcommand{\Q}{\mathbb Q}
\newcommand{\R}{\mathbb R}
\newcommand{\C}{\mathbb C}
\newcommand{\Z}{\mathbb Z}
\newcommand{\N}{\mathbb N}

\newcommand{\J}{\mathscr J}

\newcommand{\D}{\mathscr D}

\newcommand{\F}{\mathscr F}

\newcommand{\CC}{\mathscr C}

\newcommand{\G}{\mathscr G}
\newcommand{\X}{\mathscr X}

\newcommand{\A}{\mathscr A}
\newcommand{\B}{\mathscr B}
\newcommand{\V}{\mathscr V}
\newcommand{\W}{\mathscr W}
\newcommand{\T}{\mathscr T}
\newcommand{\M}{\mathscr M}
\newcommand{\E}{\mathscr E}

\newcommand{\z}[1]{\tensor[^0]{#1}{}}
\newcommand{\e}[1]{\tensor[^e]{#1}{}}

\newcommand{\JES}{J_{\mathrm{ES}}}
\newcommand{\JAHS}{J_{\mathrm{AHS}}}

\newcommand{\diff}{\mathrm{d}}
\newcommand{\del}{\partial}
\newcommand{\delb}{\bar{\del}}

\newcommand{\dvol}{\mathrm{dvol}}

\newcommand{\so}{\mathfrak{so}}

\newcommand{\g}{\mathfrak{g}}

\renewcommand{\P}{\mathbb P}
\renewcommand{\H}{\mathbb H}
\renewcommand{\O}{\mathscr O}

\renewcommand{\tilde}{\widetilde}

\renewcommand{\leq}{\leqslant}
\renewcommand{\geq}{\geqslant}
\renewcommand{\L}{\mathscr L}

\renewcommand{\Re}{\mathop{\mathrm{Re}}}

\theoremstyle{plain}
	\newtheorem{theorem}{Theorem}
	\newtheorem{proposition}[theorem]{Proposition}
	\newtheorem{lemma}[theorem]{Lemma}
	\newtheorem{corollary}[theorem]{Corollary}

\theoremstyle{definition}
	\newtheorem{definition}[theorem]{Definition}
	\newtheorem{remark}[theorem]{Remark}

\theoremstyle{plain}
	\newtheorem*{theorem*}{Theorem}
	\newtheorem*{proposition*}{Proposition}
	\newtheorem*{lemma*}{Lemma}
	\newtheorem*{corollary*}{Corollary}
	\newtheorem*{conjecture*}{Conjecture}
\theoremstyle{definition}
	\newtheorem*{definition*}{Definition}
	\newtheorem*{remark*}{Remark}
	\newtheorem*{remarks*}{Remarks}

\newcommand*{\defeq}{\mathrel{\vcenter{\baselineskip0.5ex \lineskiplimit0pt
                     \hbox{\scriptsize.}\hbox{\scriptsize.}}}%
                     =}
                     
\usepackage[nottoc]{tocbibind} % [nottoc,numbib]

\title{Knots, minimal surfaces and $J$-holomorphic curves}
\author{Joel Fine\\ Département de mathématiques\\ Université libre de Bruxelles\\ Belgique}
\date{ }

\numberwithin{equation}{section}
\numberwithin{theorem}{section}

\begin{document}

\maketitle

\vfill

\abstract{
Let $K$ be a knot in the 3-sphere, viewed as the ideal boundary of hyperbolic 4-space $\H^4$. We prove that the number of minimal discs in $\H^4$ with ideal boundary $K$ is a knot invariant. I.e.\ the number is finite and doesn't change under isotopies of $K$. In fact this gives a family of knot invariants, indexed by an integer describing the extrinsic topology of how the disc sits in $\H^4$. These invariants can be seen as Gromov--Witten invariants counting $J$-holomorphic discs in the twistor space $Z$ of $\H^4$. Whilst Gromov--Witten theory suggests the general scheme for defining the invariants, there are substantial differences in how this must be carried out in our situation. These are due to the fact that the geometry of both $\H^4$ and $Z$ becomes singular at infinity, and so the $J$-holomorphic curve equation is degenerate, rather than elliptic, at the boundary. This means that both the Fredholm and compactness arguments involve completely new features.
}

\vfill
\vfill

\newpage

\vfill

\tableofcontents

\vfill

\newpage

\section{Introduction}\label{introduction-section}

\subsection{Overview and discussion}

Let $K\subset \del \H^4 \cong S^3$ be a knot in the ideal boundary of hyperbolic 4-space. The main result of this paper is that \emph{the number of minimal discs in $\H^4$ with ideal boundary $K$ is a knot invariant}. In other words, the number of minimal discs is finite and it doesn't change under isotopies of $K$. The phrase ``number of minimal discs'' needs to be carefully interpreted here. The knot $K$ must be generic in a certain sense and the number of minimal discs is a signed count. A more careful statement is given in \S\ref{statement} below; for the full details, see Defintion~\ref{definition-of-disc-invariants} and Theorem~\ref{invariance-of-disc-invariants}.

This article also lays the groundwork for counting minimal surfaces with more complicated topology, with the ultimate goal of defining a family of integer invariants $n_{\chi,d}(K)$ of an oriented link $K$, indexed by a pair of integers, $\chi, d$ where we count compatibly oriented minimal fillings of $K$, with Euler characteristic $\chi$ and ``twistorial self-linking number'' $d$. This second index $d$ describes the topology of how the surface sits in $\H^4$ (see \S\ref{self-linking}). There are, however,  difficulties to be overcome to rigorously define these counts, which we describe in~\S\ref{failure-properness}.

We follow the broad strategy employed in the definition of other enumerative invariants in Fredholm differential topology. Indeed these ``minimal surface link invariants'' can be seen as a kind of Gromov--Witten invariant of the twistor space $Z$, which is an infinite-volume 6-dimensional symplectic manifold fibring over $\H^4$. However, whilst the general set-up is familiar, there are substantial differences in the details compared with ``standard'' Gromov--Witten invariants. These are all ultimately due to the fact that the geometry of both $\H^4$ and $Z$ becomes singular at infinity and so the $J$-holomorphic equation for a map $u\colon \Sigma \to Z$ is \emph{degenerate}, rather than elliptic, along $\del \Sigma$. This means that both the Fredholm theory and compactness arguments involve completely new features. In particular, the proof of compactness near infinity for sequences of $J$-holomorphic curves in $Z$ relies entirely on the properties of minimal surfaces in $\H^4$. %(In fact, the whole story can be told without explicit reference to the symplectic form on $Z$ at all.) 

From  a technical perspective, this article treats all the new difficulties of counting $J$-holomorphic curves which are due purely to the singular geometry of $Z$ near infinity. We give a brief outline of the new techniques required in \S\ref{outline-proof}. 
 This is sufficient to count minimal discs. When the domains of the curves have more complicated topology, one encounters the usual problems of bubbling and nodes which always arise in defining Gromov--Witten invariants. We show that these problems can only occur in a compact region of $Z$ (see Theorems~\ref{convergence-near-infinity} and~\ref{convergence-interior}) and so one can hopefully deal with them by applying more standard methods. See the discussion of \S\ref{failure-properness} for more details.

It should be stressed at this point that the idea of counting minimal surfaces goes back at least to Tomi and Tromba in 1978 \cite{Tomi-Tromba}, although to the best of our knowledge, the majority of subsequent work in this direction has involved minimal \emph{hypersurfaces}. We give a brief discussion of how some of the known results relate to ours in \S\ref{context-ms}.   

A natural hope is to compare these counts of minimal surfaces with traditional link invariants. In this direction, the twistorial picture reveals a very intriguing potential connection between counts of minimal surfaces in $\H^4$ and the HOMFLYPT polynomial. As we briefly explain in \S\ref{context-skeins-on-branes}, comparing with the work of Ekholm and Shende on open Gromov--Witten invariants in the conifold \cite{Ekholm-Shende} leads to the speculation that \emph{the counts of minimal surfaces in $\H^4$ with ideal boundary $K$ may be closely related to the coefficients of the HOMFLYPT polynomial of $K$.}

%There is a very intriguing connection between counts of minimal surfaces in $\H^4$ and recent work of Ekholm--Shende on the open Gromov--Witten invariants of the confiold and the HOMFLYPT polynomial  \cite{ }. This leads to the speculation that the (currently still conjectural) counts $n_{\chi,d}(K)$  of minimal surfaces filling a link $K$ are closely related to the coefficients of the HOMFLYPT polynomial of $K$. We briefly describe this in~\S\ref{context-skeins-on-branes}. 

Going beyond minimal surfaces in $\H^4$, this work can be seen as the first step in a larger programme to define link invariants. Let $K \subset Y$ be a link in a compact 3-manifold. Given an asymptotically hyperbolic 4-manifold $M$ with ideal boundary $Y$, we would like to define an invariant of $K$ by counting minimal surfaces in $M$ which are asymptotic to $K$. When $M$ is negatively curved, the conformal infinity is totally umbilic, and we are considering knots in $Y$ rather than links, the techniques developed here can most likely be adapted to define an invariant counting minimal discs. For counting minimal surfaces with more complicated topology the same difficulties described in \S\ref{failure-properness} arise.

Another potential source of link invariants comes from the $J$-holomorphic curve perspective. Let $X$ be a compact 6-manifold with boundary $\del X \cong S^2 \times Y$. Suppose that there is a certain type of symplectic form on the interior of $X$, which blows up at $\del X$ transverse to the $S^2$ fibres. The singularity should be modelled on the asymptotic behaviour of the symplectic structure on the hyperbolic twistor space $Z \to \H^4$. Given a link $K \subset Y$, we would now like to define an invariant of $K$ by counting $J$-holomorphic curves in $X$ whose boundary curves in $\del X$ project down to $K$. The techniques of \S\ref{Fredholm-theory} can be modified to show that this is a Fredholm problem, provided one uses an almost complex structure $J$ which also has the correct type of singularity near $\del X$. However, the same obstacles discussed in \S\ref{failure-properness} must also be overcome in the definition of link invariants in this context. 

%The following introductory sections are quite discursive. In part this is to guide the reader through the even longer but often technical proofs which follow. Mainly, however, it is because the story interweaves two related but distinct disciplines, minimal surfaces and $J$-holomorphic curves. The discussion and context given here is hopefully helpful those readers who are not experts in both.  

%The remainder of \S\ref{introduction-section} is laid out as follows. In \S\ref{statement} we give a precise statement of our results. In \S\ref{failure-properness} we explain what problems occur when considering minimal surfaces besides discs, and how one might hope to overcome them. In \S\ref{outline-proof} we comment on the proofs, focusing on the differences between the situation in hand and that of ``standard'' Gromov--Witten theory. In \S\ref{context-ms} we compare our result with other work on minimal surfaces. In \S\ref{context-skeins-on-branes} we draw a parallel with Ekholm--Shende's recent work on open Gromov--Witten invariants in the conifold. This leads to speculation that the counts of minimal surfaces in $\H^4$ with ideal boundary $K$ are closely related to the coefficients of the HOMFLYPT polynomial of $K$. %In \S\ref{perspectives} we describe some possible extensions of the ideas described here. These include counting minimal surfaces in 4-manifolds and $J$-holomorphic curves in symplectic 6-manifolds that have carefully prescribed singular geometry at infinity.  

The rest of the article is as follows. \S\ref{twistors} gives the geometric background on twistor theory, and derives the asymptotic behaviour of $J$-holomorphic curves in the twistor space of $\H^4$. \S\ref{Fredholm-theory} lays out the Fredholm theory for the degenerate $J$-holomorphic curve equation. \S\ref{compactness-at-infinity} and \S\ref{convergence-interior-section} prove compactness results for sequences of minimal surfaces in $\H^4$ whose ideal boundaries converge. Up to here we work with minimal surfaces of genus $g$ filling links. \S\ref{the-invariants} focuses on minimal discs filling knots, putting all the pieces together in Defintion~\ref{definition-of-disc-invariants}, of the knot invariant given by counting minimal discs, and Theorem~\ref{invariance-of-disc-invariants} which proves that this count really is invariant under isotopy of the boundary knot.

Before proceeding, we say a word about our notation. We use a bar to denote a compact manifold with boundary, so $\overline{\Sigma}$ is a compact surface with $\del \Sigma \neq \emptyset$. We write the interior by omitting the bar, $\mathrm{Int}(\overline{\Sigma}) = \Sigma$. This is standard in several texts on analysis over manifolds with boundary, but conflicts with notation in complex geometry where the bar denotes the conjugate Riemann surface or conjugate complex vector bundle. To avoid this conflict, we do not use bars for conjugation, instead writing for example $T^*\Sigma^{0,1}$ for the anti-holomorphic cotangent bundle of a Riemann surface. Finally, by $\overline{\H}^4$ we mean the compactification of $\H^4$ given by adjoining the ideal boundary. Note that the conformal structure of $\H^4$ extends to $\overline{\H}^4$, which is conformally diffeomorphic to a closed 4-ball with the Euclidean metric.

\subsection{Statement of main result}\label{statement}

Let $\overline{\Sigma}$ be a compact connected surface of genus $g$, with boundary having $c>0$ components.  We write
\[
\M_{g,c} = \frac{\{ (f,j) \}}{\sim}
\] 
where
\begin{itemize}
\item $j$ is a complex structure on $\overline{\Sigma}$.
\item $f \colon (\overline{\Sigma},j) \to \overline{\H}^4$ is $C^{2,\alpha}$ and conformal. 

We allow critical points, where $\diff f = 0$; such maps are sometimes called ``weakly conformal''.
\item At the boundary, $f|_{\del \Sigma} \colon \del \Sigma \to \del \H^4$ is an embedding with image a $C^{2,\alpha}$ link. 
\item $f(\overline\Sigma) \cap \del \H^4 =f(\del \Sigma)$, with this intersection being transverse.
\item On the interior, $f|_\Sigma \colon \Sigma \to \H^4$ is harmonic.
\item $(f,j) \sim (f',j')$ if they are related by a diffeomorphism of $\overline{\Sigma}$.
\end{itemize}
The fact that $f$ is conformal and harmonic on $\Sigma$ means that $f(\Sigma)$ is a minimal surface, potentially immersed and with branch points where $\diff f=0$. (The precise definition of $\M_{g,c}$ that we use in the text has technical differences---see Definition~\ref{definition-moduli-space}---but we elide these points in the introduction.) 

Denote by $\L_c$ the space of $C^{2,\alpha}$ links in $\del \H^3 \cong S^3$ with $c$ components. Write \[\beta \colon \M_{g,c} \to \L_c\] for the map which sends the minimal surface $[f,j]$ to its ideal boundary $f(\del \Sigma)$. Our main results are as follows:
\begin{enumerate}
\item $\M_{g,c}$ is a smooth Banach manifold (Theorem~\ref{smooth-moduli-space}).
%\item The subset of those $[f,j] \in \M_{g,c}$ for which $f$ has branch points is of codimension~2 (Theorem~\ref{branch-points-codim2}).
\item $\beta \colon \M_{g,c} \to \L_c$ is Fredholm and has index~0 (Theorem~\ref{boundary-map-Fredholm}).
\item There is a distinguished trivialisation of the determinant line bundle $\det (\diff \beta)$ over the region of $\M_{g,c}$ in which $\dim (\coker \diff \beta) \leq 1$ (Proposition~\ref{orientations}).
\item Given a sequence $[f_n,j_n] \in \M_{g,c}$ for which $\beta[f_n,j_n]$ converges to a link $K \in \L_c$, a subsequence of $[f_n,j_n]$ converges to a possibly nodal minimal surface with ideal boundary equal to $K$ (Theorem~\ref{convergence-interior}). 

The potential appearance of nodes prevents $\beta$ from being proper.
\item When $g=0$ and $c=1$, so we are considering minimal discs, no nodes occur and $\beta$ is proper (Theorem~\ref{properness-for-discs}).
\end{enumerate}

From here, the minimal disc invariant is defined as follows. Let $K\in \L_1$ be a regular value of $\beta \colon \M_{0,1} \to \L_1$. The above results imply that $\beta^{-1}(K)$ is a compact oriented 0-manifold; we define $n(K)$ to be the signed count of the number of points it contains. Suppose that $K_0,K_1 \in \L_1$ are regular values of $\beta$ which are connected by a path in $\L_1$, i.e.\ they represent the same topological knot. We can also chose this path to be transverse to $\beta$; then the above results imply that its preimage in $\M_{0,1}$ is a compact oriented 1-cobordism between $\beta^{-1}(K_0)$ and $\beta^{-1}(K_1)$, proving $n(K_0) = n(K_1)$. In other words, $n(K)$ is a genuine knot invariant. 

In fact, there is a refinement of $n(K)$. There is an integer-valued topological invariant of the maps to $\H^4$ which appear here, which we call the ``self-linking number''. This is constant on the connected components of $\M_{g,c}$ and so one can try and count minimal surfaces with self-linking number equal to $d \in \Z$, giving a family of link invariants $n_{d}(K)$ parametrised by $d \in \Z$. When the surface is immersed, $d$ is defined as follows. Let $f \colon \overline{\Sigma} \to \overline{\H}^4$ be an immersed minimal surface with embedded ideal boundary $K$. Choose a nowhere vanishing section $v$ of the normal bundle to $\overline{\Sigma}$. It is a simple consequence of minimality that $\overline{\Sigma}$ meets $\del_\infty \H^4$ at right-angles and so $v|_K$ is a normal framing of $K$. We define $d$ to be the self-linking number of this framing. It turns out to depend only on $f$, and not the choice of $v$. See~\S\ref{self-linking} for the full discussion.

We close this description of the minimal disc invariant with an application: \emph{given an unknotted $C^{2,\alpha}$ circle $U \subset \del \H^4$, there exists a minimal disc in $\H^4$ with boundary $U$ and self-linking number $0$.} To prove this  consider a great circle $C \subset \del\H^4$, bounding a totally geodesic $\H^2 \subset \H^4$ the maximum principle quickly shows that this $\H^2$ is the \emph{only} minimal surface filling $C$. It follows fairly easily that $n_0(C)=1$. Since $U$ is isotopic to $C$, $n_0(U)=1$ also. If there were no minimal discs filling $U$ with $d=0$, then $U$ would vacuously be a regular value of $\beta$ and so $n_0(U)=0$, a contradiction. See Theorem~\ref{unknot-always-filled} for the details. 

We remark that the existence of a minimal disc filling $U$ also follows from the general theory of the Plateau problem, but that this approach gives no information on the self-linking number $d$. Put briefly, in the half-space model $\H^4 = \{ (x,y_1,y_2,y_3) : x>0\}$, consider the translate $U_\epsilon$ of $U$ which lies in the horosphere $x=\epsilon$. The solution of the Plateau problem in $x \geq \epsilon$ gives a minimal disc filling $U_\epsilon$. (Note there is no control over the self-linking number of these discs.) Finally take the limit as $\epsilon \to 0$ to find a disc filling $U$. This idea goes back to Anderson's seminal work on the asymptotic Plateau problem in hyperbolic space~\cite{Anderson}.

%This application shows both the utility and difficulty of the minimal disc invariant: if one can compute $n_d(K)$ for a particular choice of knot, this gives an existence result for minimal surfaces for all isotopic knots. The challenge is to compute it for any $K$ at all. We will say a few words on this extremely difficult problem in~\S\ref{perspectives}. 

\subsection{The failure of properness of \texorpdfstring{$\beta$}{beta} in general}\label{failure-properness}

We next discuss in more detail the properness of the boundary value map $\beta$, with an eye on how to count other minimal surfaces besides discs. This will also bring in to play the twistorial point of view. 

Let $[f_n,j_n] \in \M_{g,c}$ be a sequence of minimal surfaces whose ideal boundaries converge. Theorem~\ref{convergence-interior} shows that the only way in which properness can fail is if the complex structures $j_n$ themselves degenerate, i.e.\ they converge modulo diffeomorphisms to a \emph{nodal} Riemann surface with boundary. One important fact is that this limiting domain is constrained by the existence of the maps $f_n$: \emph{no nodes can occur on the boundary} (Corollary~\ref{Gamma-bounded-length}). Theorem~\ref{convergence-interior} then shows that the maps $f_n$ themselves converge to a nodal conformal harmonic map, i.e.\ a finite collection of conformal harmonic maps $h_j$ defined on Riemann surfaces $\overline{\Sigma}_j$ each with non-empty boundary (components of the limiting nodal Riemann surface). The $h_j$ obey the same conditions as those in the definition of $\M_{g,c}$, and they agree on those interior points of the $\Sigma_j$ which are identified in the nodes. 

The appearance of nodes prevents $\beta$ from being proper. There is an instructive concrete example of this behaviour, due to Manh~Tien Nguyen \cite{Nguyen}.  Let $H_0 \subset \del \H^4$ be the standard Hopf link, made of two great circles $C_1, C_2$ which are the boundaries of a pair of totally geodesic copies of $\H^2_1, \H^2_2 \subset \H^4$ of $\H^2$ which meet orthogonally in a single point. Nguyen has shown that the singular surface $\H^2_1 \cup \H^2_2$ is the \emph{only} minimal surface whose ideal boundary is $H_0$. He also constructs a family $H_t \subset \del \H^4$ of links isotopic to $H_0$, parametrised by $t \in \R^2$, each of which are filled by a minimal annulus. As $t\to 0$, $H_t$ converges to the standard Hopf link $H_0$; meanwhile the annuli degenerate and converge in the limit to the nodal surface $\H^2_1 \cup \H^2_2$. So $\beta$ fails to be proper over the family $H_t$. 

We can still hope to count minimal annuli, however. Notice that in Nguyen's example, \emph{the singularity occurs in codimension~2} (because $t\in \R^2$). We call a link $K \in \L_c$ \emph{bad} if it is the limit of a sequence $K_n$ of the form $K_n = \beta[f_n,j_n]$ where $[f_n,j_n] \in \M_{g,c}$ and the $j_n$ degenerate (modulo diffeomorphisms) to a nodal Riemann surface. The above example shows that the standard Hopf link $H_0$ is a bad link. Away from the bad links, the map $\beta$ is proper (by Theorem~\ref{convergence-interior}). Suppose for a moment that the bad locus is codimension~2 in $\L_c$, as suggested by Nguyen's example. Then a pair of good links in the same isotopy class could be joined by a path avoiding bad links; since $\beta$ is proper over the good links, the count of minimal fillings could be defined and would still be an isotopy invariant of the link.

To explain, heuristically at least, why the set of bad links should have codimension~2, we move to the $J$-holomorphic curve interpretation. We recall the background very briefly here; more details are given in \S\ref{twistors}. The twistor space $\pi \colon \overline{Z} \to \overline{\H}^4$ is the 2-sphere bundle whose fibres parametrise orthogonal almost-complex structures at points of $\overline{\H}^4$. There is an almost-complex structure $J$ on the interior $Z$ (due to Eells--Salamon \cite{Eells-Salamon}) which has the following special property: a map $u \colon (\Sigma,j) \to Z$ is $J$-holomorphic if its projection $f = \pi \circ u$ is conformal and harmonic. Eells--Salamon show that this gives a 1-1 correspondence between conformal harmonic maps to $\H^4$ and $J$-holomorphic curves in $Z$. So $\M_{g,c}$ is also a moduli space of $J$-holomorphic curves in $Z$, with a certain prescribed behaviour at infinity. 

Now suppose we attempt to define the invariant counting minimal annuli, as in the discussion of the Hopf links above. Let $[f_n,j_n]$ be a sequence of minimal annuli which converge to a pair of intersecting minimal discs. The minimal annuli in $\H^4$ lift to a sequence of $J$-holomoprhic annuli $u_n \colon (\overline{\Sigma},j_n) \to \overline{Z}$. There are now two possibilities. The first is that the $u_n$ converge to a pair of $J$-holomorphic discs, which cover the two minimal discs in $\H^4$. This is actually what happens in Nguyen's example above. The second possibility is that a bubble forms at the same time as the node (or even a whole bubble tree). In this case, the limit is a pair of $J$-holomorphic discs which are joined by a twistor $S^2$-fibre (as these are the only $J$-holomorphic spheres in $Z$). 

Bad behaviour of the first type (when no bubble appears) happens in codimension~2. This is because asking for a pair of surfaces to intersect in the 6-manifold $Z$ is a codimension~2 phenomenon. This needs a transversality result, which follows from the Fredholm theory constructed here, showing that by moving the boundary one has  ``enough'' freedom to move $J$-holomorphic curves around.  (This is a consequence of Theorem~\ref{ev-is-transverse}, which shows that the evaluation map from the moduli space of marked $J$-holomorphic curves is a submersion.) 

Bad behaviour of the second type (when a bubble appears) is more complicated. Heuristically, it should be of codimension~4. This is because we are now asking for two codimension~2 phenomena to occur simultaneously (the two different points of intersection with the fibre). However, at this point transversality fails: the twistor fibre has index~0 and so ``should'' be an isolated $J$-holomorphic sphere. However, the fibre is obstructed: it moves in a 4-dimensional family (the fibres of $Z \to \H^4$) and so the singular configurations actually fill generic links. 

These nodal $J$-holomorphic curves are hopefully not harmful to the definition of a minimal surface invariant, however, for the reason that they can't necessarily be smoothed out. This is precisely \emph{because} the twistor fibre is obstructed. If one were to try and smooth out a nodal $J$-holomorphic curve containing a twistor fibre, by gluing to make an approximate solution and then perturbing to a genuine one, there would be a 4~dimensional cokernel to overcome, due to the 4-dimensional obstruction space of the twistor fibre itself. One would only hope to be able to proceed by first moving the boundary link to make sure the gluing errors were orthogonal to the cokernel. This means we would have to first carefully position the boundary link so as to lie in a subspace of codimension~4 in the space of all links. In this way we see that, whilst the singular configurations due to the second type of bad behaviour fill generic links, the expectation is that they actually should only truly arise as limits for a codimension~4 subspace of links.  

Whilst the above argument is heuristically convincing, it is certainly a long way from a proof. To make it rigorous one would need to study the Kuranishi space describing small deformations of a nodal $J$-holomorphic curve including a twistor fibre (or, more generally, with a closed component mapping to a twistor fibre). In the standard setting of Gromov--Witten invariants these ideas are well established (going back at least to Ionel \cite{Ionel}). One key ingredient is the Fredholm theory of the Cauchy--Riemann equation. What must be checked is if the Fredholm theory relevant to our situation (given in \S\ref{Fredholm-theory}) combines well with these methods to give a suitable description of the Kuranisi space. 

Another potential problem is the appearance of ``ghost components'' in the limit: a closed component of the limiting nodal domain on which the $J$-holomorphic map is constant. If this component meets two other components each of which run out to infinity, there is no problem, since this behaviour is codimension~2 in the space of links, as a consequence of Theorem~\ref{ev-is-transverse}. If the ghost component only meets one other infinite component, the situation is again more complicated. One could try to argue, again via the Kuranishi space as in Doan--Walpuski's recent work~\cite{Doan-Walpuski}, that the corresponding node on the infinite component is a branch point. If this can be done, we can use the fact that branch points occur in codimension~4 to proceed (Theorem~\ref{branch-points-codim2}).

%An alternative approach to this problem is to perturb the almost complex structure on $Z$, so that the $J$-holomorphic spheres genuinely are isolated and then bubbles would appear in codimension~4. As long as the perturbation tends to zero quickly enough at infinity, the Fredholm theory of \S\ref{Fredholm-theory} will be unaffected. The difficulty with this is that it potentially clashes with the compactness results near infinity of \S\ref{compactness-at-infinity}. The compactness results given here use crucially that the $J$-holomoprhic curves in $Z$ project to minimal surfaces in $\H^4$, at least near infinity (see~\ref{outline-proof} below). The perturbed almost complex structure would have to be chosen carefully so that the barrier arguments could still be applied (most critically an analogue of Lemma~\ref{convex-hull}). This method also has the drawback that we would no longer be counting minimal surfaces; the resulting invariant would be one of the symplectic topology of $Z$ and not truly the Riemannian geometry of $\H^4$. 

%Finally, we remark that there are also other factors to consider, for domains more complicated than annuli. For example, when considering genus~1 surfaces with a single boundary component one could envisage a situation in which the limiting nodal domain was a made up from a disc and a closed elliptic curve. The maps to $\overline{Z}$ may then a constant on this elliptic curve, a so-called ``ghost'' component. One would need a different argument to deal with this, perhaps along the lines of the work of Doan--Walpuski \cite{Doan-Walpuski}. 

\subsection{Remarks on the proofs}\label{outline-proof}

We now comment on the proofs of our results and in particular the way they differ from standard results on $J$-holomorphic curves. There are two aspects: Fredholm theory and compactness arguments. The Fredholm theory of the $J$-holomorphic curve equation for a map $u \colon \overline{\Sigma} \to \overline{Z}$ is made complicated by the fact that the equation degenerates at $\del \Sigma$. The symbol of the linearisation \emph{vanishes} in directions normal to $u(\overline{\Sigma})$. Tangentially however the equation is elliptic in the standard sense. Mazzeo and Melrose have developed a theory of degenerate operators, called the 0-calculus, which we use in the normal directions. Meanwhile, we must apply classical elliptic theory in the tangental directions. Combining these two approaches poses certain technical problems, which is one reason for the length of \S\ref{Fredholm-theory}. A more conceptual hurdle is the absence of a general index theorem in the 0-calculus. To deal with this, we use an ad hoc method to compute the index of the lineraised Cauchy--Riemann operator in the normal directions. This hinges on the fact that $Z$ is an ``almost Calabi--Yau'' (see \S\ref{boundary-map-Fredholm-section}).

The degeneracy at infinity brings one interesting feature. In the standard setting of Grom\-ov--Witten theory, the moduli space of $J$-holomorphic curves is finite dimensional but not necessarily cut out transversely. To deal with this (and other similar issues such as transversality of the evaluation map) one uses the almost complex structure $J$ as a parameter. Letting $J$ vary as well gives an infinite dimensional moduli space which is now smooth (at least in good circumstances), and one applies Sard--Smale to show that for a generic $J$ the moduli space is actually cut out transversely after all. In our setting however, the degeneracy of the equation at infinity means that the moduli space of $J$-holomorphic curves is a smooth infinite dimensional Banach manifold straight away, with no need to vary $J$. Instead, the boundary link $K$ is our ``free parameter''. Whenever in the standard story one would expect to use the freedom to perturb $J$, in our discussion we perturb $K$ instead. 

When it comes to compactness, our situation also requires a new approach, relying entirely on the fact that our $J$-holomorphic curves project to minimal surfaces in $\H^4$. On the one hand, on the interior things are much more straightforward than the standard setting. Given a compact Riemann surface with boundary $(\overline{\Sigma},j)$ we use the corresponding complete hyperbolic metric on the interior $\Sigma$ to measure the energy density of maps from $\Sigma$. A simple application of the Bochner formula tells us that (at least for the maps we are considering) when $f \colon \Sigma \to \H^4$ is conformal and harmonic, it has bounded energy density. It follows from interior regularity for harmonic maps that $f$ is bounded in $C^k$ for all $k$ (see Proposition~\ref{interior-estimate}). 

The challenge is to control the maps near infinity. Take a sequence $[f_n,j_n] \in \M_{g,c}$ of conformal harmonic maps, for which $f_n(\del \Sigma)$ converges in $C^{2,\alpha}$. We extract a subsequence which converges near infinity. The difficulty is that, because the metric of $\H^4$ is singular at $\del \H^4$, the minimal surface equation is degenerate at infinity and so one cannot appeal directly to elliptic estimates. Instead we work in three stages. Firstly we use barriers, and the notion of ``convex hull'' (introduced by Anderson \cite{Anderson}) to get uniform $C^0$ control of the minimal surfaces near infinity (see \S\ref{convex-hull-section}). The second step is to upgrade this to $C^1$ control, writing the surfaces as graphs over a fixed cylindrical region near infinity. The argument here hinges on a deep result due to White, constraining the region in which area can blow up in a sequence of minimal surfaces \cite{White} (see \S\ref{C1-control-section}). The final step is to increase this to $C^2$ control and then $C^{2,\alpha}$-convergence. For this, we use the fact that minimal surfaces in $\H^4$ are also solutions of the fourth-order Willmore equation. The point is that the Willmore equation is conformally invariant and so is genuinely elliptic, rather than degenerate, along $\del \Sigma$. From here, together with the already established $C^1$-control, we can at last use estimates from the theory of elliptic PDE to increase the control of the surfaces, first in $C^2$ and then sufficiently to find the $C^{2,\alpha}$-convergent subsequence, at least near infinity (see~\S\ref{C2-control-section} and~\S\ref{C2alpha-control-section}).

\subsection{Context: counting minimal surfaces}\label{context-ms}

The idea of counting minimal surfaces goes back to the work of Tomi and Tromba in 1978 on the classical Plateau problem~\cite{Tomi-Tromba}. We briefly recall their result. Let $C \subset \R^3$ be a closed surface bounding a convex region. Given a closed loop $L \subset C$, Tomi and Tromba prove that there is an embedded minimal disc with boundary $L$. The long-standing solution of the Plateau problem (due to Douglas \cite{Douglas} and Rado \cite{Rado}) shows that a minimal disc filling $L$ exists, the point is to prove that one can actually find an \emph{embedded} disc. To do this, Tomi and Tromba consider the space $\M$ of all embedded minimal discs whose boundary lies in $C$ and the space $\L$ of all loops in $C$. They show that $\M$ is a smooth Banach manifold, and that the boundary map $\beta \colon \M \to \L$ is Fredholm of index zero and, crucially, proper. This means that $\beta$ has a well-defined $\Z_2$-valued degree. Since the intersection of $C$ with a plane is filled by a single embedded minimal disc (the part of the plane itself lying inside $C$), this degree is non-zero and so $\beta$ must be surjective. 

One important difference with our situation is that for the classical Plateau problem the minimal surface equation is not degenerate. Another significant difference is that these minimal surfaces are \emph{hypersurfaces}. This is why Tomi and Tromba can work with \emph{embedded} surfaces. In our case, where the minimal surfaces have codimension~2, we must allow immersions and branch points. This complicates things significantly since the surfaces are not graphical near a branch point and, moreover, the maximum principle is not available. One of the key parts in our compactness argument is to show that, given a sequence of minimal surfaces whose boundary links converge, any immersed points or branch points can't escape to infinity (this is the heart of the proof of Theorem~\ref{C1-bound-at-infinity}).

The asymptotic Plateau problem in hyperbolic space was first considered by Anderson~\cite{Anderson}. He proved that given any $(k-1)$-dimensional submanifold $Y \subset \del \H^{n}$ (where $k<n$), there is an area-minimising locally-integral $k$-current in $\H^n$ with ideal boundary $Y$. This seminal work has led to a whole branch of research. Rather than attempt to list results here, we refer the reader to the excellent survey of Coskunuzer~\cite{Coskunuzer}.

The direct hyperbolic analogue of Tomi and Tromba's work was considered by Alexakis and Mazzeo \cite{Alexakis-Mazzeo}. This paper had a huge influence on the present article. Write $\mathscr{N}_{g,c}$ for the space of properly embedded minimal surfaces in \emph{three}-dimensional hyperbolic space $\H^3$, of genus $g$, whose ideal boundary has $c$ components and which are $C^{k,\alpha}$-submanifolds of $\overline{\H}^3$ up to the boundary, for some $k \geq 2$ and $0< \alpha <1$. Write $\E_c$ for the space of $C^{k,\alpha}$ embeddings of $c$ disjoint circles in $\del \H^3 \cong S^2$. Alexakis and Mazzeo prove that $\mathscr{N}_{g,c}$ is a smooth Banach manifold and that the boundary map $b \colon \mathscr{N}_{g,c} \to \E_c$ is Fredholm of index~0. They also consider the properness of $b$, but unfortunately the proof they give is not complete. It was in discussions about this, that Mazzeo explained to the author how to use White's regularity result from~\cite{White} in the compactness agruments.

In any case, the properness of $b$ follows from the results in this article. The upshot is that given an isotopy class of $c$ disjoint circles $C = C_1\sqcup\cdots \sqcup C_c \subset S^2$,  there is a well-defined degree, $\delta_{g}(C)$, given by counting properly-embedded genus $g$ minimal fillings of $C$ in $\H^3$. Just as in the work of Tomi and Tromba, when $c=1$ one can compute this degree for the equator. This is filled by a totally geodesic copy of $\H^2$ and nothing else and so $\delta_0(C_1)=1$ for a single circle, whilst  $\delta_g(C_1)=0$ for all $g>0$. Unfortunately, for $c>1$, the Alexakis--Mazzeo degrees give no more information on the existence of minimal surfaces. In \cite{Nguyen} Nguyen has shown that when $c>1$, all the degrees vanish: $\delta_g(C_1 \sqcup \cdots \sqcup C_c) =0$. 

Nguyen's results also bear on our situation. He proves the following:

\begin{theorem}[Nguyen, \cite{Nguyen}]
Let $K_1, K_2$ be two $C^2$ links in $S^3$. It is possible to arrange them in an unlinked fashion $K_1 \sqcup K_2 \subset \del \H^4$ so that there is no connected minimal surface in $\H^4$ with ideal boundary equal to $K_1 \sqcup K_2$. 
\end{theorem}

Assuming one can eventually define link invariants by counting minimal surfaces, an important corollary of this would be that for a \emph{split} link, the count of connected minimal fillings would vanish. 

%It is important to point out that just because $\delta_g(C)=0$, it doesn't mean there are no connected minimal fillings. They may simply come in pairs with opposite signs. There is a beautiful explicit example of this behaviour, due to Mori~\cite{Mori}. He considers minimal annuli in $\H^3$ whose boundary are a pair of circles of latitude in $S^2$. When the two circles $C_1, C_2$ are close to the equator, there are two minimal annuli in $\H^3$ which fill $C_1 \sqcup C_2$, one which stays closer to infinity and the other which runs a long way into the interior. As you move the circles apart, towards opposite poles, the inner annuli moves outwards, and the outer annuli moves inwards. At a critical moment, the surfaces coalesce and after this point there is no filling at all. From the point of view of the Alexakis--Mazzeo degree, before the critical point we have two surfaces counted with opposite sign whilst at the critical point the boundary is not a regular value of $b$.

A pessimist might worry that the minimal surface link-invariants also all vanish, with the exception of $n_0(U)=1$ for the unknot. This seems very unlikely. Firstly, there is Nguyen's work on Hopf links mentioned above: the standard Hopf link has a \emph{unique} filling, albeit by a nodal minimal surface, while the nearby Hopf links $H_t$ are filled by minimal annuli which converge back to this nodal surface. This strongly suggests that the minimal annuli invariant of the Hopf link should equal~1 (assuming it can be defined in the first place!). A second reason to be hopeful that these invariants are non-trivial is given in the next section.

\subsection{Context: \texorpdfstring{$J$}{J}-holomorphic curves and the conifold}\label{context-skeins-on-branes}

There are many different approaches to link-invariants defined via $J$-holomorphic curves. We briefly describe one here which, whilst at first sight very different, is surely closely related to counting minimal surfaces in $\H^4$.

In \cite{Ooguri-Vafa} Ooguri and Vafa conjectured how to recover the coefficients of the HOMFLYPT polynomial of a link $K \subset S^3$ via open Gromov--Witten invariants. This has subsequently been proved by Ekholm and Shende \cite{Ekholm-Shende}. The idea involves two 6-manifolds. The first is the cotangent bundle $T^*S^3$ of the 3-sphere and the second is the total space of $\O(-1) \oplus \O(-1) \to \C\P^1$, which we denote by $X$ to ease the notation. The manifolds $X$ and $T^*S^3$  are related by a ``conifold transition''. Consider the cone $C = \{ xw=yz\}$ in $\C^3$. Away from the origin this is a Kähler manifold. One can check that the complement of the zero section in  $T^*S^3$ is symplectomorphic to $C \setminus \{0\}$. Meanwhile, the natural inclusion $\O(-1) \to \C^2$ on each summand gives a map $X \to \C^4$ with image $C$; this is a biholomorphism between the complement of the zero section in $X$ and $C\setminus\{0\}$.

Write $\omega_0$ for the closed 2-form on $X$ given by pulling back the Euclidean Kähler form from $\C^4$, write $\omega_{\mathrm{FS}}$ for the pull-back to $X$ of the Fubini--Study form from $\C\P^1$ and put $\omega_t = \omega_0 + t\omega_{\mathrm{FS}}$.  This is symplectic for $t>0$. At $t=0$ it degenerates on the central $\C\P^1$, and away from there $(X\setminus \C\P^1,\omega_0)$ is symplectomorphic to $T^*S^3$ minus the zero section. 

Now given an oriented link $K \subset S^3$, the conormal $L_K \subset T^*S^3$ of $K$ is a Lagrangian submanifold. It meets the zero section of $T^*S^3$, but it can be perturbed away from it by adding a small closed 1-form dual to the tangent direction of $K$. Via the above mentioned identifications with $C\setminus{0}$, it then gives a submanifold in $X$, which we continue to denote $L_K$. Since $L_K$ is Lagrangian in $T^*S^3$, we have $\omega_0|_{L_K}=0$. At the same time, $\omega_{\mathrm{FS}}|_{L_K}=0$ so that $L_K$ is a Lagrangian in $(X,\omega_t)$. The idea now is to consider the open Gromov--Witten invariants counting $J$-holomorphic curves in $X$ with boundary in $L_K$. Ekholm and Shende show how to define these invariants and prove Ooguri and Vafa's conjecture that, when correctly assembled, the counts of $J$-holomorphic curves recover the coefficients of the HOMFLYPT polynomial of $K$. 

The reason this appears closely related to the minimal surface link-invariant is because, for $t >0$, there is a path $\Omega_t$ of symplectic forms on the twistor space $Z \to \H^4$ and \emph{the resolved conifold $(X,\omega_t)$ is symplectomorphic to the twistor space $(Z, \Omega_t)$} (see \cite{Fine-Panov}). This is a good reason to be optimistic about both the existence and non-triviality of the minimal surface link-invariants. Nonetheless there are important differences between the two situations. Ekholm and Shende work with generic almost complex structures which are integrable near $K \subset L_K$. This certainly excludes the Eells--Salamon almost complex structure. More significantly, they use Lagrangian boundary conditions, rather than $J$-holomorphic curves which run out to infinity.  This means that Ekholm--Shende's result says nothing directly about the minimal surfaces we count here.

\subsection{Acknowledgements}

It is a pleasure to thank Marcelo Alves, Rafe Mazzeo, Manh Tien Nguyen and Marco Usula for many important conversations during several stages of this work. I would also like to thank Antoine Gloria, Roman Golovko, Michael Khanevsky, and Ivan Smith for their helpful comments and insights. 

Various parts of this research were supported by ERC consolidator grant ``SymplecticEinstein'' 646649, an Excellence of Science grant 30950721, and a grant from the Fonds Thelam of the Fondation Roi Baudouin. 

\section{The twistor geometry of \texorpdfstring{$\H^4$}{H4}}\label{twistors}

In this section we describe the relevant parts of the twistor geometry of $\H^4$. We derive the asymptotic behaviour of a $J$-holomorphic map into the twistor space (Proposition~\ref{asymptotic-expansion}) and collect various lemmas and definitions which will be important later.

\subsection{Twistor spaces}

We begin by recalling the definition of the twistor space of a 4-manifold, and of the Eells--Salamon almost-complex structure it carries. Twistor spaces were introduced by Penrose \cite{Penrose} and then first explored in the Riemannian setting by Atiyah, Hitchin and Singer \cite{Atiyah-Hitchin-Singer}. Eells and Salamon made the connection with minimal surfaces in \cite{Eells-Salamon}.

\begin{definition}
Let $(M,g)$ be an oriented 4-manifold and let $p \in M$. The \emph{twistor fibre at $p$} is the set $Z_p$ of all almost complex structures on $T_pM$ which are orthogonal with respect to $g$ and which induce the chosen orientation.  One can check that $\SO(T_pM)$ acts transitively on $Z_p$ with stabiliser a copy of $\U(2)$ and so $Z_p \cong \SO(4)/\U(2) \cong S^2$. Moreover, the action of $\SO(T_pM)$ on $Z_p$ means these twistor fibres fit together to give an $S^2$-bundle $\pi \colon Z \to M$, associated to the principal bundle of oriented frames of $(M,g)$. $Z$ is called the \emph{twistor space} of $(M,g)$.  
\end{definition}

\begin{remark}\label{conformal-invariance-twistor}
Notice that this definition depends only on the conformal class of $g$: the conformally equivalent metric $e^fg$ determines the same twistor space as $g$. Consider hyperbolic space $\H^4$ and its compactification $\overline{\H}^4$ obtained by adding the ideal boundary $\del \H^4 \cong S^3$. Since the conformal class of the hyperbolic metric extends to $\overline{\H}^4$, so does the twistor space, giving an $S^2$-bundle $\pi \colon \overline{Z} \to \overline{\H}^4$. 
\end{remark}

The Levi-Civita connection of $(M,g)$ determines a splitting $TZ = V \oplus H$ of the tangent bundle to $Z$. The vertical bundle $V = \ker \diff \pi \subset TZ$ consists of those vectors which are tangent to the fibres of the twistor projection $\pi$. On the complement $H$, there is an isomorphism $\diff \pi \colon H \to \pi^*TM$. 

Using this splitting one can define two almost-complex structures on $Z$. First note that the vertical tangent bundle $V \to Z$ is naturally a complex vector bundle, since the fibres of $\pi$ are oriented round 2-spheres. We denote the fibrewise linear complex structure of $V$ by $J_V$. The horizontal bundle $H$ is also naturally a complex vector bundle. To see this, given a point $z \in Z$, write $j_z$ for the corresponding almost complex structure on $T_{\pi(z)}M$. Since $\diff \pi \colon H_z \to T_{\pi(z)}M$ is an isomoprhism, we can interpret $j_z$ as a complex structure on $H_z$ and as $z$ varies these structures fit together to make $H$ into a complex vector bundle whose fibrewise linear complex structure we denote by~$J_H$. 

\begin{definition}
The \emph{Atiyah--Hitchin--Singer} almost-complex structure $\JAHS$ on $Z$ is given with respect to $TZ=V \oplus H$ by
\[
\JAHS = J_V \oplus J_H.
\]
Meanwhile, the \emph{Eells--Salamon} almost-complex structure $\JES$ on $Z$ is given with respect to $TZ = V \oplus H$ by
\[
\JES = (-J_V )\oplus J_H.
\]
\end{definition}

\begin{remark}
The splitting $TZ = V \oplus H$ is defined by the Levi--Civita connection and so is \emph{not} conformally invariant. This means that a~priori neither are $\JAHS$ or $\JES$. In fact, $\JAHS$ \emph{is} conformally invariant and so, in the case of hyperbolic space, it extends smoothly up to the boundary of $\overline{Z}$. As we will see, $\JES$ does not extend; it has singularities at the boundary, transverse to the fibres of $\pi$. We will focus entirely on $\JES$. We only mention $\JAHS$ in passing, because it is the typical choice made in twistor theory. This is because when the metric $g$ is anti-self-dual (as is the case for $\H^4$ for example) $\JAHS$ is integrable, and $(Z,\JAHS)$ is a complex manifold. In contrast, $\JES$ is \emph{never} integrable. 
\end{remark}

We close this quick review of twistor spaces by describing the natural metric on the twistor space of a Riemannian 4-manifold.

\begin{definition}\label{twistor-metric}
Given an oriented Riemannian 4-manifold $(M,g)$ there is a unique metric on the twistor space $\pi \colon Z \to M$  for which $\pi$ is a Riemannian submersion, the fibres have their natural round metrics (with scalar curvature equal to~1) and the Levi-Civita splitting $TZ = V \oplus H$ is orthogonal. We call this metric the \emph{twistor metric} on $Z$.
\end{definition}

\subsection{The Eells--Salamon correspondence}

Eells and Salamon showed how $\JES$-holomorphic curves in $Z$ correspond to minimal surfaces in $(M,g)$ \cite{Eells-Salamon}. We briefly recall this here.
 
\begin{definition}
Given an oriented 2-dimensional subspace $L \subset T_pM$, there is a unique $z \in Z_p$ for which $L$ is a $j_z$ complex line, with the correct orientation. This is because $j_z$ is completely determined on $L$, and then also on $L^{\perp}$ since $j_z$ must induce the overall orientation on $T_pM = L \oplus L^\perp$. From here we can lift oriented surfaces in $M$ up to $Z$. Let $\Sigma \subset M$ be an embedded oriented surface. The \emph{twistor lift} of $\Sigma$ is the surface $\tilde{\Sigma} \subset Z$ lying over $\Sigma$ whose point above $p \in \Sigma$ is the unique point $z \in Z_p$ for which $T_p\Sigma \subset T_pM$ is a $j_z$-complex line, with correct orientation. 

The definition extends in an obvious way to immersed oriented surfaces $\Sigma \subset M$. It also extends to branched immersions $\Sigma \to M$ from an oriented surface, the point being that the image of a branched immersion still has a well-defined tangent space at the branch point. (See~\cite{Eells-Salamon} for the details). 
\end{definition}

The utility of twistor lifts for studying $J$-holomorphic curves begins with the following observation. (This appears in \cite{Eells-Salamon}, but the proof is short and the result important, so we give the argument here.)

\begin{lemma}[Eells--Salamon \cite{Eells-Salamon}]
Let $u \colon \Sigma \to Z$ be a non-constant $J$-holomorphic map from a Riemann surface for either $\JES$ or $\JAHS$. Suppose moreover that $u(\Sigma)$ contains no fibres of $\pi \colon Z \to M$. Then $u$ is the twistor lift of $\pi \circ u$. 
\end{lemma}

\begin{proof}
Let $p \in \Sigma$ be a point at which $u$ is an immersion. To ease notation write $z=u(p)$. The image $\diff u(T_p\Sigma) \subset T_{z}Z$ is a complex line. Recall that there is a complex-linear splitting $T_z(Z) =V_z \oplus H_z$. Unless $u_*(T_p\Sigma)$ is vertical, its horizontal projection is a complex line $L \subset H_z$ and this complex line projects under $\diff \pi$ to the tangent space of $(\pi \circ u)(\Sigma)$. Moreover, for either $J=\JES$ or $J=\JAHS$, the map $\diff \pi \colon (H_z,J|_H) \to (T_{\pi(z)}M,j_z)$ is a complex linear isomorphism. It follows that $z$ is the twistor lift of $\pi(u(p))$. This shows that $u$ is the twistor lift of $\pi\circ u$ at all points where $u$ is a non-vertical immersion. This is a dense set of points (since we assume $u(\Sigma)$ contains no twistor fibres) and so the result follows for every point by continuity. 
\end{proof}

The magic of $\JES$ is the following result which, put briefly, says that twistor lifts of minimal surfaces in $(M,g)$ are precisely the $\JES$-holomorphic curves in $Z$. 

\begin{theorem}[Eells--Salamon \cite{Eells-Salamon}]\label{Eells-Salamon}
Let $f \colon \Sigma \to (M,g)$ be a branched conformal immersion into an oriented Riemannian 4-manifold. The twistor lift $u \colon \Sigma \to Z$ is $\JES$-holomorphic if and only if $f$ is also harmonic (and so $f(\Sigma)$ is a minimal surface in $M$). This gives a 1-1 correspondence between branched conformal harmonic maps to $M$ and $\JES$-holomorphic curves in $Z$ with no vertical components (i.e., no components are fibres of $\pi \colon Z \to M$).
\end{theorem}

\subsection{The twistor space of \texorpdfstring{$\H^4$}{H4}}

From now on, we will write $(Z,J)$ for the twistor space of $\H^4$ with its Eells--Salamon almost complex structure. In this section we will give an explicit formula for $J$. We use half-space coordinates $x, y_1, y_2, y_3$ on $\H^4$ with $x >0$, in which the hyperbolic metric is $g = x^{-2}(\diff x^2 + \diff y^2)$. These induce coordinates $x,y_i,z_i$ on $Z$ as follows. The tuple $(x,y_i,z_i)$ corresponds to the linear endomorphism on $T_{(x,y_i)}\H^4$ given by $j_z = z_1J_z+ z_2J_2+z_3J_3$ where the $J_i$ are defined by
\begin{alignat*}{2}
J_1(x\del_x) &= x\del_{y_1}, 
  &\quad 
    J_1(x\del_{y_2}) &= x\del_{y_3}.  \\
J_2(x\del_x) &= x\del_{y_2}, 
  &\quad 
    J_2(x\del_{y_3}) &= x\del_{y_1}.  \\
J_3(x\del_x) &= x\del_{y_3}, 
  &\quad 
    J_3(x\del_{y_1}) &= x\del_{y_2}.
\end{alignat*}
One checks that $j_z^2 = -(z_1^2+z_2^2+z_3^2)$ and so $Z$ is identified with the points $(x,y_iz_i)$ for which $z_1^2+_2^2+z_3^2=1$. We call $(x,y_i,z_i)$ \emph{half-space twistor coordinates} (despite the fact that they are not true coordinates because of the condition $ z^2_1+z_2^2+z_3^2=1$). 

To describe the Levi-Civita connection on $Z$ we first introduce the orthonormal coframe $\alpha_0 = x^{-1}\diff x$, $\alpha_i = x^{-1}\diff y_i$. Calculation then shows that the Levi-Civita connection on $Z$ is given in our coordinates by the following 1-form with values in $V$:
\[
A 
  = 
    \alpha_1 \otimes \left(z_2 \del_{z_3} - z_3 \del_{z_2}\right)
    +
    \alpha_2 \otimes \left(z_3 \del_{z_1} - z_1 \del_{z_3}\right)
    +
    \alpha_3 \otimes \left(z_1 \del_{z_2} - z_2 \del_{z_1}\right)
\]
The horizontal distribution $H \subset TZ$ is given by $\ker A$.

The $\alpha_i$ are horizontal 1-forms, i.e.\ they annihilate the vertical tangent spaces. The vertical 1-forms are given by
\begin{eqnarray}
  \beta_1 & = & \diff z_1 + (z_2\alpha_3 - z_3 \alpha_2)\nonumber\\
  \beta_2 & = &\diff z_2 + (z_3\alpha_1 - z_1 \alpha_3)\label{vertical-1-forms}\\
  \beta_3 & = & \diff z_3 + (z_1\alpha_2 - z_2 \alpha_1)\nonumber
\end{eqnarray}
These are vertical in the sense that they annihilate $H$. Note that the $\beta_i$ are independent if we ignore the constraint $z_1^2+z_2^2+z_3^2=1$, but become dependent when restricted to $Z$. At $(x,y_i,z_i) \in Z$, the linear combination $p_1\beta_1 + p_2 \beta_2 + p_3 \beta_3$ restricts to zero on $Z$ if and only if $(p_1,p_2,p_3)$ is proportional to $(z_1,z_2,z_3)$. We also have the dual frame of horizontal and vertical vectors:
\begin{eqnarray}
  h_0 &=& x \del_x \nonumber\\
  h_1 &=& x\del_{y_1} + z_2 \del_{z_3} - z_3 \del_{z_2}\quad \text{ etc.}
  \label{HV-vectors}\\
  v_i &=& \del_{z_i} \nonumber
\end{eqnarray}

We next describe the Eells--Salamon almost-complex structure $J$. It is convenient to extend $J$ to every value of $z$ and not just those for which $z_1^2+z_2^2+z_3^2=1$. We continue to denote this extension by $J$ and hope that no confusion occurs. We have the following:
\begin{eqnarray}
  J(v_1) &=& z_2 v_3 - z_3 v_2 \quad\text{ etc.}\nonumber\\
  J(h_0) &=& z_1h_1 + z_2h_2 + z_3 h_3\label{J-edge-vector-formulae}\\
  J(h_1) &=&-z_1 h_0 -z_2h_3 + z_3 h_2 \quad \text{ etc.}\nonumber
\end{eqnarray} 
The point is that at $(x,y_i,z_i)$, $J$ is equal to $j_z$ on the horizontal space and cross-product with $-z$ on the vertical space. One checks that this gives the above formula.  Note that $J$ has a 1-dimensional kernel, given by $z_1 v_1 + z_2 v_2+ z_3 v_3$; at a point of $Z$, the orthogonal complement to this kernel is precisely the tangent space to $Z$ and here $J$ restricts to the Eells--Salamon almost-complex structure.

We now give a formula for $J$ in half-space twistor coordinates, that is to say with respect to the coordinate vector fields $\del_x, \del_{y_i}, \del_{z_i}$. To ease the notation, we write the following
\begin{align}
z &= 
	\begin{pmatrix} z_1\\z_2\\z_3\end{pmatrix}\\
R(z)
	&=
		\begin{pmatrix}
		0 & -z_3 & z_2\\
		z_3 & 0 & -z_1\\
		-z_ 2 & z_1 & 0
		\end{pmatrix}
		\label{R}\\
P(z)
	&=
		\begin{pmatrix}
		z_2^2 + z_3^2 & -z_1z_2 & -z_1z_3 \\
		-z_1z_2 & z_1^2 + z_3^2 & -z_2z_3\\
		-z_1z_3 & -z_2z_3 & z_1^2 + z_2^2
		\end{pmatrix}
		\label{P}
\end{align}
So $R(z)$ is the matrix representing the cross-product with $z$ and, when $|z|=1$, $P(z) = 1 - zz^T$ is the matrix representing orthogonal projection onto the plane perpendicular to $z$.

The following formula follows quickly from~\eqref{HV-vectors} and~\eqref{J-edge-vector-formulae}.

\begin{lemma}\label{J-in-local-coordinates}
Let $(x,y_i,z_i)$ be half space twistor coordinates. With respect to the local frame $\del_x, \del_{y_i}, \del_{z_i}$, we have the following matrix representative for $J$:
\[
J
\begin{pmatrix}\del_x\\ \del_{y_i}\\ \del_{z_i}\end{pmatrix}
	=
		\begin{pmatrix}
		0 & -z^T & 0\\
		z & R(z) & 0\\
		0 & 2x^{-1}P(z) & -R(z)
		\end{pmatrix}
\begin{pmatrix}\del_x\\ \del_{y_i}\\ \del_{z_i}\end{pmatrix}
\]
\end{lemma}

The most important thing to take away from this formula is the $x^{-1}$ term. The precise way in which $J$ blows up as $x \to 0$ is what drives the entire story. 

%\begin{remark} The $x^{-1}$ factor here shows that $J$ does \emph{not} extend up to the boundary of $\overline{Z}$, in contrast to the Atiyah--Hitchin--Singer structure. In fact, $J$ does extend when considered as an almost complex structure on the \emph{edge} tangent bundle. This also follows from~\eqref{J-edge-vector-formulae}, since the vector fields~\eqref{HV-vectors} are edge vector fields, i.e.\ at the boundary they are tangent to the fibres of $\pi$. The theory of $J$-holomorphic curves in general edge almost complex manifolds seems not to have been explored. The results of this paper could be seen as the first steps in this direction.  
%\end{remark}

\subsection{Asymptotic expansion of \texorpdfstring{$J$}{J}-holomorphic curves}

In this section we determine the asymptotic behaviour of $J$-holomorphic curves in the twistor space of $\H^4$. Let $\overline{\Sigma}$ be a compact Riemann surface with boundary and $u \colon \overline{\Sigma} \to \overline{Z}$ be a smooth map. We pull back the half-space twistor coordinates $(x,y_i,z_i)$ via $u$ and consider them as functions on $\overline{\Sigma}$. The main result of  this section, Proposition~\ref{asymptotic-expansion} ,says that when $\delb u$ vanishes to sufficiently high order at $\del \Sigma$ the asymptotic behaviour of $u$ is completely determined by the boundary value of $y$, i.e. by $\pi\circ u|_{\del \Sigma}$.

%\begin{definition}\label{transverse-at-infinity}
%The map $u \colon \Sigma\to \overline{Z}$  \emph{is transverse at infinity} if $u|_{\del \Sigma}$ is an embedding and the function $x \circ u \colon \Sigma \to \R$ is a boundary defining function. I.e.,
%\begin{itemize}
%\item $(x\circ u)^{-1}(0) = \del \Sigma$.
%\item $\diff (x\circ u)$ is non-zero at all points of $\del \Sigma$.
%\end{itemize}
%\end{definition}

Since we are only interested in the asymptotic behaviour of $u$, we consider maps from the cylinder $\overline{C}=S^1 \times [0,1)$, with  coordinates $s \in S^1$ and $t \in [0,1)$. The complex structure on $\overline{C}$ is given by $j(\del_s) = \del_t$. Given a map $u \colon \overline{C} \to \overline{Z}$, we pull back the twistor coordinates $x,y_i,z_i$ and treat them as functions of $(s,t)$. We assume that $x,y$ are $C^{2,\alpha}$ whilst $z$ is $C^{1,\alpha}$ (for some choice of $0<\alpha<1$). We write the Taylor expansions of $x,y_i,z_i$ with respect to $t$ at $t=0$ as follows:
\begin{equation}
\begin{aligned}
x &= a(s)t + b(s)t^2 + O(t^{2+\alpha})\\
y &= \gamma(s) + \psi(s) t + \phi(s) t^2 + O(t^{2+\alpha})\\
z & = \zeta(s) + \eta(s) t + O(t^{1+\alpha}).
\end{aligned}
\label{taylor}
\end{equation}

\begin{proposition}\label{asymptotic-expansion}
Let $u\colon \overline{C} \to \overline{Z}$, with Taylor expansion \eqref{taylor} at $t=0$. We assume that $a(t)>0$ (so that the image of $u$ is transverse to the boundary of $\overline{Z}$). The asymptotic condition
\begin{equation}
\delb u (\del_t) = O(t^{1+\alpha})\del_x + O(t^{1+\alpha}) \del_y + O(t^{\alpha}) \del_z
\label{asymptotically-holomorphic}
\end{equation}
is equivalent to the following formulae for the coefficients of the Taylor expansion~\eqref{taylor}:
\begin{equation}
\begin{aligned}
a &= |\dot \gamma|,\\
b &=0,\\
\psi &=0,\\
\zeta&= -|\dot\gamma|^{-1}\dot \gamma,\\
\eta &= R(\zeta)(\dot\zeta),\\
\phi&= \frac{1}{2} \ddot{\gamma} + \dot{a}\zeta
\end{aligned}
\label{taylor-coefficients-determined}
\end{equation}
where $R$ is defined in~\eqref{R}, the dot denotes $\diff/\diff s$ and $|\dot \gamma|^2 = |\dot \gamma_1|^2 + |\dot \gamma_2|^2 + |\dot \gamma_3|^2$.  In particular, all the coefficients in~\eqref{taylor-coefficients-determined} are determined by~$\gamma$. 
\end{proposition}

\begin{proof}
Using Lemma~\ref{J-in-local-coordinates}, we write out the given condition on $\delb u(\del_t) = \del_t u - J(u)(\del_s u)$:
\begin{align}
\del_t x + z^T \del_s y &= O(t^{1+\alpha}),\label{x-dbar-u}\\
\del_t y - (\del_s x)z - R(z)(\del_s y) &=O(t^{1+\alpha}),\label{y-dbar-u}\\
\del_t z -2x^{-1}P(z)(\del_s y) + R(z)(\del_s z) &=O(t^{\alpha}).\label{z-dbar-u}
\end{align}
The leading order term on the left-hand side of~\eqref{z-dbar-u} is at $O(t^{-1})$. Setting this equal to zero gives
\[
P(\zeta)(\dot\gamma) = 0
\]
This means that $\zeta$ is parallel to $\dot \gamma$. The condition that $|z|^2=1$ means that $\zeta$ is unit length and so there are two possibilities: $\zeta = \pm |\dot\gamma|^{-1}\dot \gamma$. To fix the sign, consider the leading order term on the left of~\eqref{x-dbar-u}. This is at $O(1)$ and setting it equal to zero gives
\[
a + \zeta^T \dot \gamma = 0
\]
Now $a>0$ which means we must have $\zeta = - |\dot \gamma|^{-1}\dot \gamma$ and $a = |\dot \gamma|$ as claimed. 

Next we consider the leading order term on the left of~\eqref{y-dbar-u}. This is $O(1)$ and setting it equal to zero gives
\[
\psi - R(\zeta)(\dot \gamma) =0
\]
Since $\zeta$ is parallel to $\dot \gamma$, we deduce that $\psi = 0$. Turn now to the $O(t)$ term on the left of~\eqref{x-dbar-u}. This gives 
\[
2b + \eta^T\dot \gamma = 0
\]
The fact that $|z|^2=1$ means $\eta^T \zeta =0$ and, since $\zeta$ and $\dot \gamma$ are parallel, we deduce that $\eta^T\dot \gamma=0$ and so $b=0$. 

We next determine $\eta$. Setting the $O(1)$ term on the left of~\eqref{z-dbar-u} to zero gives
\[
\eta + \frac{2}{a}(\eta \zeta^T + \zeta \eta^T)\dot \gamma + R(\zeta)(\dot \zeta) =0
\]
(Here we have used the formula $P(z) = 1 - zz^T$ to expand $P(z)$ in powers of $t$.) Using the facts that $\zeta = -|\dot \gamma|^{-1}\dot \gamma$, that $\eta^T \dot \gamma =0$ and that $a =|\dot \gamma|$ we conclude that $\eta = R(\zeta)(\dot \zeta)$.

Finally, we determine $\phi$. Setting the $O(t)$ term in~\eqref{y-dbar-u} equal to zero gives
\[
2\phi -\dot{a} \zeta - R(\eta)(\dot \gamma) = 0.
\]
Note that $R(\eta)(\dot\gamma)=-a R(R(\zeta)(\dot\zeta))(\zeta)$. When vectors $u$ and $v$ are perpendicular, one checks that $R(R(u)(v))(u)=v$. Since $\zeta$ is unit length, $\zeta$ and $\dot \zeta$ are orthogonal and so $R(\eta)(\dot\gamma) = -a\dot{\zeta}$. This means that
\[
\phi = \frac{1}{2}(\dot{a}\zeta -a\dot{\zeta}) = \frac{1}{2} \ddot{\gamma} + \dot{a}\zeta
\]
as claimed.
\end{proof}

\begin{remark}
It is not possible to extend this expansion. Suppose that the components of $\delb u(\del_t)$ vanish to one higher order. If one attempts to solve for the next order coefficients in $x,y,z$ one finds that the terms coming from~\eqref{z-dbar-u} involving the $O(t^2)$-coefficient of $z$ all cancel, leaving it undetermined. This in turn makes it impossible to determine the $O(t^3)$ coefficients of $x$ and $y$. This phenomenon is well-known for minimal surfaces: for an asymptotically minimal surface, the ideal boundary uniquely determines its jet up to $O(t^2)$ but no further. This is shown in \cite{Alexakis-Mazzeo} for minimal surfaces in $\H^3$ and in \cite{Fine-Herfray} for minimal surfaces in general $n$-dimensional Poincaré--Einstein manifolds. We will see another manifestation of this when we compute the indicial roots of the linearised $J$-holomorphic equation in \S\ref{normal-operator-section}.
\end{remark}

\begin{remark}
There is a geometric interpretation to the first terms computed in Proposition~\ref{asymptotic-expansion}. The fact that $\psi = 0$ says that, under the twistor projection $\pi \colon Z \to \H^4$, the asymptotically minimal surface $(\pi \circ u)(\Sigma)$ meets the ideal boundary $\del \H^4$ at right angles. This can be proved directly using totally geodesic copies of $\H^3$ as barriers. 

To give the geometric explanation of $\zeta = - |\dot\gamma|^{-1}\dot\gamma$, we need a short digression. We work momentarily with the Euclidean metric on the closed unit ball $\overline{B} \subset \R^4$, which is conformally equivalent to the hyperbolic metric. Given a point $p \in S^3 = \del B$ on the boundary, and a unit tangent vector $v \in T_pS^3$, there is a unique compatible almost complex structure $j_v$ on $T_pB$ which sends $v$ to the inward unit normal $n$ at $p$. This is because $j_v$ must preserve the orthogonal complement $\left\langle v,n \right\rangle^\perp$ and its action here is determined by its action on $\left\langle v,n \right\rangle$ and the fact that it is orientation compatible. In this way, we set up an identification between $\del Z\to S^3$ and the $S^2$-bundle $P \to S^3$ of oriented directions in $TS^3$. The minimal surface $(\pi\circ u)(\Sigma)$ meets $S^3$ in a closed embedded curve---a knot, or link if $\del \Sigma$ has several components---and each component of this curve has two natural lifts to $P$ (depending on the orientation). Now Proposition~\ref{asymptotic-expansion} says that when $u$ is asymptotically $J$-holomorphic,  $u(\del \Sigma) \subset P \cong \del Z$ is precisely one of these lifts of $(\pi \circ u)(\del \Sigma) \subset S^3$.  
\end{remark}

\subsection{The edge geometry of twistor space}

As Lemma~\ref{J-in-local-coordinates} shows, $J$ blows up at the boundary of $\overline{Z}$. This means that given an arbitrary smooth map $u \colon \overline{\Sigma} \to \overline{Z}$ which sends $\del \Sigma$ to $\del Z$, the expression $\delb u$ doesn't make sense up to the boundary. It will be convenient to fix this by working with the edge tangent bundle of $Z$. Edge geometry gives a framework in which singularities like those appearing in $J$ can be treated as smooth up to the boundary. We briefly recall those parts of this theory that we will need and refer to \cite{Amman-Lauter-Nistor,Mazzeo} for more details. 

\begin{definition}
Let $\overline{X}$ be a compact manifold with boundary $\del X$. An \emph{edge structure} on $\overline{X}$ is a choice of a smooth submersion $\pi \colon \del X \to Y$ making the boundary into a fibre bundle over a compact manifold $Y$. The \emph{edge vector fields} on $\overline{X}$ are those vector fields which at points of $\del X$ are tangent to the fibres of $\pi$. We write $\Vect_e(\overline{X})$ for the set of edge vector fields, which is a Lie subalgebra of the algebra $\Vect(\overline{X})$ of all vector fields on $\overline{X}$. One special case of an edge structure is the \emph{$0$-structure on $\overline{X}$}; this is the edge structure in which the projection $\pi \colon \del X \to \del X$  is the identity map. The corresponding \emph{$0$-vector fields}, denoted $\Vect_0(\overline{X})$, are those that vanish on the boundary. 
\end{definition}

The cases of interest to us are the twistor space $\overline{Z}$ with edge structure $\del Z \to S^3$ given by the twistor projection, and a compact Riemann surface $\overline{\Sigma}$ with boundary, together with its $0$-structure. 

It is convenient to use coordinates which are adapted to an edge structure. The twistor coordinates from above on $\overline{Z}$ are just such a set of coordiantes. In general, given $p \in \del X$, we trivialise $\pi \colon \del X \to Y$ near $\pi(p)$. Choose coordinates $y_i$ on $Y$ near $\pi(p)$ and coordinates $z_\alpha$ on the fibre of $\pi$ through $p$.  These extend, via the local trivialisation of $\pi$, to functions which we continue to denote $y_i,z_\alpha$ defined on an open subset of $\del X$ containing $p$. To extend them further into the interior, we pick a collar neighbourhood of the boundary, i.e.\ an open set $U \subset \overline{X}$ containing $\del X$ and a diffeomorphism $U \cong \del X \times [0,\delta)$. The functions $y_i,z_\alpha$ on $\del X$ now extend into $U$ in the obvious way. Finally, projection $U \to [0,\delta)$ defines a function which we denote $x$. All together this gives a local system of coordinates $(x,y_i,z_\alpha)$ defined on an open set in $\overline{X}$ containing the boundary point $p \in \del X$.

\begin{definition}
In edge coordinates the edge vector fields are generated by $x\del_x, x\del_{y_i}$ and $\del_{z_\alpha}$, with coefficient functions which are smooth up to $x=0$. Locally then, $\Vect_e(\overline{X})$ is freely generated over $C^\infty(\overline{X})$ by $n+1$ vector fields (where $n+1 = \dim\overline{X})$.  It follows from the Serre--Swann Theorem that there is a rank-$(n+1)$ vector bundle $\e{TX} \to \overline{X}$ whose space of sections is isomorphic to $\Vect_e(\overline{X})$ as a $C^\infty(\overline{X})$-module. $\e{TX}$ is called the \emph{edge tangent bundle}. Equivalently, we can declare $\e{TX}$ to be isomorphic to $TX$ on the interior and trivial over any local edge coordinate patch $(x, y_i, z_\alpha)$ near the boundary as above. We then define the transition functions by the transformation rules obeyed by the $n+1$ vector fields $x \del_x, x \del_{y_i}, \del_{z_\alpha}$. The \emph{edge cotangent bundle} is the dual of $\e{TX}$, denoted  $\e{T^*X}$. It is locally spanned by the covectors $x^{-1}\diff x, x^{-1}\diff y_i, \diff z_\alpha$, again with coefficient functions which are smooth up to $x=0$.

It follows from the definition that the inclusion $\Vect_e(\overline{X}) \subset \Vect(\overline{X})$ is induced by a homomorphism of vector bundles $\e{T X} \to T\overline{X}$, called the \emph{anchor map}. This is an isomorphism on the interior whilst over the boundary it has image equal to $\ker \diff\pi$, i.e.\ the vectors tangent to the fibres of $\pi$.  In the case of a $0$-structure we talk of the $0$-tangent bundle $\z{T X}$.  
\end{definition}

\begin{definition}
An \emph{edge almost-complex structure} is an endomorphism $J$ of $\e{TX}$ which has $J^2 = -1$. 
\end{definition}

It follows from Lemma~\ref{J-in-local-coordinates} that the Eells--Salamon almost-complex structure $J$ on the hyperbolic twistor space $Z$ extends to $\e{TZ}$ to give an edge almost-complex structure on $\overline{Z}$. For the $0$-tangent bundle, there is a canonical isomorphism between $\End(\z{TX}) \cong \End(TX)$. This is because $\End(\z{TX}) = \z{T^*X} \otimes \z{TX}$ and $(x^{-1}\diff y_i )\otimes (x\del_{y_j}) = \diff y_i \otimes \del_{y_j}$. It follows that we can interpret the ordinary complex structure on a Riemann surface $\overline{\Sigma}$ with boundary as a complex structure on $\z{T\Sigma}$.

Of course one can consider edge singularities for any type of tensor. Another which we will use are edge metrics. 

\begin{definition}
An \emph{edge metric} is a fibrewise positive-definite inner-product $g$ on the the edge tangent bundle $\e{TX}$. In the case of the $0$~edge structure and $\z{TX}$, we talk of $0$-metrics.
\end{definition}

On the interior, $g$ is a Riemannian metric in the traditional sense, which becomes singular at $\del X$ with singularities controlled by the edge structure. The hyperbolic metric on $\H^4$ extends to a $0$-metric on $\overline{\H}^4$ whilst the twistor metric on $Z$ extends to an edge metric on $\overline{Z}$. Similarly, given a Riemann surface $\overline{\Sigma}$ with boundary, the natural complete hyperbolic metric on the interior $\Sigma$ extends to a $0$-metric on $\overline{\Sigma}$.

%In general, $g$ is a $0$-metric precisely when $x^2g$ extends to a smooth metric up to the boundary. (Such metrics are also commonly called ``conformally compact''.) It follows from Remark~\ref{conformal-invariance-twistor} that when $\overline{M}$ is a 4-manifold with $0$-metric, the twistor space extends up to the boundary $\overline{Z} \to \overline{M}$. It is then natural to ask about the singularities of the Eells--Salamon almost-complex structure in this setting. We state the following result without proof, since we will not rely on it in this article. 
%
%\begin{lemma}
%Let $\overline{M}$ be a compact oriented 4-manifold with boundary and $0$-metric $g$. The Eells--Salamon almost-complex structure on the twistor space $Z \to M$ extends to an edge almost-complex structure on $\overline{Z}$, where the edge structure is given by the twistor projection over the boundary.  
%\end{lemma}
%
%It would be interesting to study the behaviour of $J$-holomorphic curves in general edge almost-complex manifolds. Given the results of this paper, a sensible next step would be to move from the twistor space of $\H^4$ to twistor spaces of more general $0$-metrics on 4-manifolds. 

We next explain how to lift the differential of a map $u \colon \overline{\Sigma} \to \overline{Z}$ to a bundle map between the zero and edge tangent bundles. 

\begin{lemma}\label{lift-d-to-edge}
Let $u \colon \overline{\Sigma} \to \overline{Z}$ be $C^1$. We assume that $u(\overline{\Sigma}) \cap \del Z = \del \Sigma$ and that this intersection is transverse. Then there is a unique lift of $\diff u$, which we denote by $\diff_e u$, which makes the following diagram commute (here the vertical arrows are the anchor maps).
\[
\begin{CD}
  \z{T\Sigma} @>\diff_e u>> u^*\e{TZ}\\
  @VVV @VVV\\
  T\Sigma @>\diff u>> u^*TZ 
\end{CD}
\]
If $u$ is $C^{k,\alpha}$, then $\diff_e u \in C^{k-1,\alpha}(\z{T^*\Sigma} \otimes u^*\e{TZ})$
\end{lemma}  
\begin{proof}
The transversality condition on $u$ means that we can use $t = x \circ u$ as a boundary defining function on $\overline{\Sigma}$. Any section of $\z{T\Sigma}$ has the form $tv$ for $v$ a section of $T\Sigma$. We then set $\diff_eu(tv) = t \diff u(v)$, which is a section of $u^*\e{TZ}$.
\end{proof}

Now both $\z{T\Sigma}$ and $\e{TZ}$ are complex vector bundles and so it makes sense to talk about the anti-linear part of $\diff_e u$, which agrees with $\delb u$ over the interior. This gives the following.

\begin{corollary}\label{edge-dbar}
Let $u \colon \overline{\Sigma} \to \overline{Z}$ be $C^1$. We assume that $u(\overline{\Sigma}) \cap \del Z = \del \Sigma$ and that this intersection is transverse. Then $\delb u$, initially only defined over the interior, extends up to the boundary as a section of $\z{T^*\Sigma}^{0,1} \otimes u^*\e{TZ}$. If $u$ is $C^{k,\alpha}$ then $\delb u \in C^{k-1,\alpha}(\z{T^*\Sigma}^{0,1} \otimes u^*\e{TZ})$.
\end{corollary}  

Recall for a moment Proposition~\ref{asymptotic-expansion}, which shows that for a map $u \colon \overline{\Sigma} \to \overline{Z}$,  the asymptotic decay of $\delb u$ constrains the Taylor expansion of $u$ at $\del \Sigma$. We can now interpret the asymptotic decay condition~\eqref{asymptotically-holomorphic} as follows: it says precisely that the section $\delb u$ of $\z{T^*\Sigma}^{0,1} \otimes u^*\e{TZ}$ is $O(t^{1+\alpha})$ when measured with respect to the $0$-metric on $\overline{\Sigma}$ and edge metric on $\overline{Z}$. This is because $t\del_t$ has bounded length in $\z{T\Sigma}$ and $t\del_x, t\del_y, \del_z$ have bounded length in $\e{TZ}$.

The final piece of edge geometry we will need is that of edge connections. 

\begin{definition}\label{edge-connection}
Let $\overline{X}$ be a compact manifold with boundary and edge structure $\pi \colon \del X \to Y$. Let  $\overline{E} \to \overline{X}$ be a vector bundle and write $E \to X$ for the restriction of $\overline{E}$ to the interior. An \emph{edge connection} $\nabla$ in $\overline{E}$ is a connection in $E \to X$ of the form $\nabla = D + A$ where $D$ is a genuine connection in $\overline{E}$ and $A \in \Gamma(T^*X \otimes \End(E))$ extends up to the boundary as a section of $\e{T^*X}\otimes \End(\overline{E})$. 

Equivalently, a connection $\nabla$ in $E \to X$ is an edge connection in $\overline{E}$ if given any section $s \in \Gamma(\overline{E})$ and edge vector field $V \in \Vect_e(\overline{X})$, the section $\nabla_V(s)$, defined a~priori only over $X$, extends to a section of $\overline{E} \to \overline{X}$. 

When the edge structure is a 0-structure, we say $\nabla$ is a \emph{0-connection}. 
\end{definition}

One readily checks that the tensor product of edge connections is again an edge connection. Similarly an edge connection in $\overline{E}$ induces a dual edge connection in $\overline{E}^*$ in the obvious way.

Now let $h$ be an edge metric on $\overline{X}$. Over the interior, $h$ is an ordinary Riemannian metric and so determines a Levi-Civita connection on $TX$. The next lemma asserts that this extends up to the boundary as an edge connection on $\e{TX}$. For a proof see, for example, \cite{Amman-Lauter-Nistor}.

\begin{lemma}
Given an edge metric on $\overline{X}$, the Levi-Civita connection on $T X \to X$ (or  any associated tensor bundle) extends to an edge connection on $\e{T X} \to \overline{X}$. 
\end{lemma}

When taken together with Lemma~\ref{lift-d-to-edge} this gives the following:

\begin{lemma}
Let $u \colon \overline{\Sigma} \to \overline{Z}$ be $C^1$. We assume that $u(\overline{\Sigma})\cap \del Z = \del \Sigma$ and that this intersection is transverse. If $\overline{E} \to \overline{Z}$ carries an edge connection $\nabla$, then $u^*\nabla$ is a $0$-connection in $u^*\overline{E} \to \overline{Z}$. In particular, the pull-back of the Levi-Civita connection via $u$ is a $0$-connection in $u^*\e{TZ}$.
\end{lemma}

\section{Fredholm theory}
\label{Fredholm-theory}

In this section we lay out the Fredholm theory for $J$-holomorphic curves $u \colon \overline{\Sigma} \to \overline{Z}$. The main results in this section are:
\begin{itemize}
\item Theorem~\ref{smooth-moduli-space}: the moduli space of $J$-holomorphic curves is an infinite dimensional Banach manifold, proved in~\S\ref{smooth-moduli-space-section}. 
\item Theorem~\ref{ev-is-transverse}: the evaluation map on the moduli space of pointed $J$-holomorphic curves is a submersion, proved in~\S\ref{ev-is-transverse-section}.
\item Theorem~\ref{branch-points-codim2}: the $J$-holomorphic curves with branch points form a set of codimension 4 in the moduli space, whilst the $J$-holomorphic curves which project to minimal surfaces with branch points have codimension 2; this is proved in~\S\ref{branch-points-codim2-section}.
\item Theorem~\ref{boundary-map-Fredholm}: the map sending a $J$-holomoprhic curve $u$ to its boundary $\pi(u(\del \Sigma))$ in $S^3$ is Fredholm with index zero, proved in~\S\ref{boundary-map-Fredholm-section}. 
\end{itemize}

\subsection{Weighted Hölder spaces}

We begin with the Banach norms which are adapted to our setting. These are standard in the study of $0$-differential operators. They make sense in any dimension, but we will apply them exclusively to bundles over a compact surface $\overline{\Sigma}$ with boundary. We fix a choice of $0$-metric $h$ on the interior $\Sigma$ (for example a complete hyperbolic metric). This metric determines Hölder norms on functions which we denote by $\Lambda^{k,\alpha}$. Given a vector bundle $\overline{E} \to \overline{\Sigma}$ with fibrewise metric and compatible $0$-connection (Definition~\ref{edge-connection}), we can talk about the $\Lambda^{k,\alpha}$ norm of a section of $\overline{E}$, defined using the metric $h$ its Levi-Civita connection. It is a direct consequence of the definition of a $0$-connection that different choices of $0$-connection give equivalent $\Lambda^{k,\alpha}$ norms. (As always with Hölder norms on sections, one must decide how to compare values of a section at different points. See, for example, \S3 of \cite{Lee} for a careful discussion of one way to do this in our setting.)

We use the $\Lambda$ notation to distinguish this norm from the usual $C^{k,\alpha}$-norm as defined with respect to a metric which is \emph{smooth} up to the boundary of~$\overline{\Sigma}$. It is important to note that the $\Lambda^{k,\alpha}$-norm is weaker than the $C^{k,\alpha}$-norm. For example, in the half-space model of the hyperbolic plane, with $h= t^{-2}(\diff s^2 + \diff t^2)$, the $\Lambda^1$-norm of a function gives uniform bounds on $t\del_t f$ and $t\del_s f$, whereas the $C^1$-norm gives uniform bounds on $\del_t f$ and $\del_s f$ which is stronger. 

We will also use weighted versions of these Hölder spaces. Let $t \colon \overline{\Sigma} \to [0,\infty)$ be a boundary defining function, i.e.~$t^{-1}(0) = \del \Sigma$ and $\diff t$ is non-zero at all points of the boundary. Given $\delta\in \R$, we write
\[
t^\delta \Lambda^{k,\alpha} = \{ t^\delta f : f \in \Lambda^{k,\alpha}\}
\]
and we use the norm
\[
\| \phi \|_{t^\delta\Lambda^{k,\alpha}} = \| t^{-\delta}\phi\|_{\Lambda^{k,\alpha}}
\]
One checks that different choices of boundary defining functions lead to the same weighted spaces and equivalent norms. 

Despite their differences, there are relationships between the $C^{k,\alpha}$ Hölder spaces and the weighted $\Lambda^{k,\alpha}$ spaces. To describe this we need to introduce some notation. The boundary defining function $t$ determines a rescaled metric $\overline{h} = t^2h$ which is smooth up to the boundary. We write $\nabla_t$ for differentiation along the gradient flow of $t$ with respect to $\overline{h}$. Given an integer $0\leq j \leq k$, we define a continuous linear map $T_j$ by
\begin{equation}
\begin{aligned}
T_j 
	&\colon 
		C^{k,\alpha}(\overline{\Sigma}) \to 
			C^{k,\alpha}(\del \Sigma) \oplus 
			C^{k-1,\alpha}(\del \Sigma) \oplus \cdots \oplus 
			C^{k-j,\alpha}(\del \Sigma)\\
T_j(f) & = 
		\left(f|_{\del \Sigma} , \nabla_t f|_{\del \Sigma}, \ldots, \nabla_t^j f|_{\del \Sigma} \right)
\end{aligned}\label{taylor-map}
\end{equation}

\begin{definition}\label{vanishing-at-boundary}
We write $C^{k,\alpha}_j = \ker T_{j}$ for the subset of functions which vanish up to order $j$ along $\del \Sigma$. Note that $C^{k,\alpha}_j$ does not depend on the choice of boundary defining function $t$ (even though the map $T_j$ does). $C^{k,\alpha}_j$ is a a closed subspace of $C^{k,\alpha}$ and so is a Banach space with the same $C^{k,\alpha}$ norm. 
\end{definition}

The following is proved in Lemmas~3.1 and~3.7 in Lee's monograph \cite{Lee}. (Warning, our notation differs: Lee writes $C^{k,\alpha}_{(s)}$ for the subspace of functions which are $O(t^s)$. By Lee's Lemma~3.1, his space $C^{k,\alpha}_{(j+1)}$ is equal to our space $C^{k,\alpha}_j$.)

\begin{lemma}[Consequences of Lemmas~3.1 and~3.7 of \cite{Lee}]\label{normal-and-zero-holder}
For $1\leq j \leq k$, there is a continuous inclusion
\[
C^{k,\alpha}_j \hookrightarrow t^{j+\alpha}\Lambda^{k,\alpha}
\]
In the other direction, there is a continuous inclusion
\[
t^{j+\alpha}\Lambda^{k,\alpha} \hookrightarrow C^{j,\alpha}_j
\]
In particular the Banach spaces $t^{k+\alpha}\Lambda^{k,\alpha}$ and $C^{k,\alpha}_k$ are isomorphic.
\end{lemma}

A final technical remark: when we talk of ``the Hölder space $\Lambda^{k,\alpha}$'' we mean the completion of smooth functions in the $\Lambda^{k,\alpha}$-norm. This space is \emph{smaller} than the space of $k$-times differentiable functions for which the $\Lambda^{k,\alpha}$ norm is finite. A similar remark applies to the Hölder space $C^{k,\alpha}$. The reason we use the smaller Hölder spaces is that they are separable, and this will be important later on, when we come to apply the Sard--Smale theorem.

\subsection{Brief review of 0-elliptic operators} \label{0-calculus-review}

Next we turn to the class of differential operators we will use, giving a very brief overview of the results we will need from the $0$-calculus, developed by Mazzeo and Melrose \cite{Mazzeo-Melrose,Mazzeo}. To avoid introducing more notation, we focus on $0$-differential operators over a compact surface $\overline\Sigma$ with boundary although the general theory is insensitive to dimension.

Let $D$ be a differential operator acting on functions over $\overline{\Sigma}$. We choose coordinates $(s,t)$ centred on a point of the boundary $p \in \del \Sigma$, with $t \geq0$ a boundary defining function. In these coordinates we write
\begin{equation}
D(f) = \sum_{i,j} a_{i,j}(s,t)(t\del_t)^i(t\del_s)^j f
\label{D-in-coordinates}
\end{equation}
\begin{definition}
$D$ is a $0$-differential operator with smooth coefficients if $a_{i,j}$ are smooth up to $t =0$. Similarly, $D$ is a $0$-differential operator with $C^{k,\alpha}$ coefficients if $a_{i,j}$ are of class $C^{k,\alpha}$ up to $t=0$. 
\end{definition}
Notice that if we defined the coefficients of a $0$-differential operator in the ordinary way, relative to $\del_t,\del_s$ rather than $t\del_t, t\del_s$, then they would vanish on the boundary (with the exception of $a_{0,0}$).  

The definition of $0$-differential operators acting on sections of bundles is similar. One first chooses local trivialisations of the bundles; the coefficients $a_{i,j}$ are now matrix-valued functions, but otherwise the definition is unchanged. 

It is a simple matter to check that a $0$-differential operator of order $r$ is a bounded linear map between the weighted Hölder spaces, $t^\delta \Lambda^{k+r,\alpha} \to t^\delta \Lambda^{k,\alpha}$. The $0$-calculus gives conditions on both $D$ and $\delta$ under which these maps are Fredholm.

\begin{definition}
Given a $0$-differential operator $D$ of order $r$, the $0$-symbol of $D$ in the direction $(p,q)$ is
\[
\sigma_0(D)(p,q) = \sum_{i+j=r} a_{i,j}(s,t) p^i q^j
\]
When this is an isomorphism for all $(p,q) \neq 0$, we say that $D$ is $0$-elliptic.
\end{definition}

When $D$ is $0$-elliptic its symbol defined in the \emph{ordinary} sense vanishes at $t=0$. This means that $0$-ellipticity alone is not enough to ensure Fredholm properties. We need to consider another model operator which is used to locally invert $D$ near the boundary. (This is the analogue of the constant coefficient differential operators used to locally invert $D$ at interior points.)

\begin{definition}\label{0-calculus-definitions}
Let $p\in \del \Sigma$ and $D$ be a $0$-differential operator. Choose coordinates $(s,t)$ centred on $p$ with $t \geq 0$ a boundary defining function. Suppose $D$ has the local expression~\eqref{D-in-coordinates}.
\begin{enumerate}
\item
The \emph{normal operator} of $D$ at $p$ is the $0$-differential operator acting on vector-valued functions on $\H^2$ given by
\[
N_p(D) = \sum_{i,j} a_{i,j}(0,0)(\tau \del_\tau)^i (\tau \del_\sigma)^j
\]
Here $(\sigma, \tau)$ are half-space coordinates on $\H^2$ with $\tau >0$ and the $a_{i,j}$ are the coefficient matrices in the expression~\eqref{D-in-coordinates}.
\item
The \emph{indicial polynomial} of $D$ at $p$ is the matrix-valued polynomial of the variable $\lambda \in \C$ given by
\[
I_p(D)(\lambda) = \sum_{i,j} a_{i,j}(0,0) \lambda^i
\]
\item
The \emph{indicial roots} of $D$ at $p$ are the values of $\lambda \in \C$ for which $I_p(D)(\lambda)$ has kernel, i.e.\ the solutions of the polynomial $\det \left(I_p(D)(\lambda)\right) = 0$.
\item
An indicial root is called \emph{simple}, if its multiplicity as a root of  $\det \left(I_p(D)(\lambda)\right) = 0$ is equal to the dimension of the corresponding space $\ker (I_p(D)(\lambda))$.
\item
When the indicial roots are constant (i.e.\ independent of $p$) and simple, the vector spaces $\ker(I_p(D)(\lambda))$ associated to the root $\lambda$ fit together to give a vector bundle which we denote $K(D,\lambda) \to \del \Sigma$.
\end{enumerate}
\end{definition} 

The main result we need is the following. We summarise just what we need, the proofs of which can be found in \cite{Mazzeo}; the ``edge calculus'' developed there can deal with much more complicated situations.

\begin{theorem}\label{0-Fredholm}
Let $D$ be a $0$-elliptic operator of order $r$. We assume that the roots of the indicial polynomial $I_p(D)$ are constant,  i.e.\ they don't depend on $p \in \del \Sigma$. Let $a>0$ be such that there are no indicial roots with real part contained in the interval $(1/2-a,1/2+a)$. Suppose, moreover, that for  some $\delta \in (1/2-a,1/2+a)$, some $k \in \N$ and $0<\alpha<1$, and for every $p\in \del \Sigma$, the normal operator of $D$ at $p$ is invertible as a map
\[
N_p(D) \colon t^\delta \Lambda^{k+r,\alpha}(\H^2) \to t^\delta \Lambda^{k,\alpha}(\H^2)
\]
Then 
\begin{enumerate}
\item
For \emph{any} $\delta \in(1/2-a,1/2+a)$, $k$ and $\alpha$, both $D$ and its formal $L^2$-adjoint $D^*$ are Fredholm as maps
\[
D,D^* \colon t^\delta \Lambda^{k+r,\alpha}(\Sigma) \to t^\delta \Lambda^{k,\alpha}(\Sigma)
\]
\item
Their kernels are independent of $\delta, k ,\alpha$. Projection $t^\delta\Lambda^{k,\alpha}(\Sigma) \to \coker (D)$ restricts to an isomorphism on  $\ker D^*$ and similarly for  $\coker D^* \cong \ker D$.
\item
Let $\phi \in t^\delta \Lambda^{k+r,\alpha}$ with $D(\phi)=0$ and $\delta > \Re \lambda$ for all indicial roots $\lambda$ of $D$. Then $\phi = o(t^c)$ for any $c >0$. 
\item
Suppose moreover that the indicial roots are simple. Let $\lambda$ denote the first indicial root with real part larger than $\delta \in (1/2-a,1/2+a)$. If $\phi \in t^\delta \Lambda^{k+r,\alpha}$ with $D(\phi)=0$, then $\phi \in t^{\re \lambda} \Lambda^{k+r,\alpha}$ and the boundary value
\[
\phi_0 = \left(t^{-\re \lambda}\phi\right)|_{\del \Sigma}
\]
is a section of $K(D,\lambda) \to \del \Sigma$. 
\end{enumerate}
\end{theorem}

In fact, it is enough to check $N_p(D)$ is invertible on the corresponding weighted Sobolev spaces. We state the result in this form because this is how we will use it. The reason to focus on an interval of the form $(1/2-a,1/2+a)$, symmetric around $1/2$, is that the spaces with weights $\delta$ and $1-\delta$ are dual with respect to the $L^2$-inner-product. We use the fact that $\dim \Sigma =2$ here; over an $(n+1)$-manifold, the dual weights are  $\delta$ and $n-\delta$. This is the only change one needs to make in dealing with arbitrary dimension. One can also work in intervals which are not symmetric about $n/2$, although then the weights required to make $D$ and $D^*$ Fredholm may be different.

\subsection{The ambient space of maps}

With the analytic background in place, we return to our study of $J$-holomorphic curves in the twistor space $Z \to \H^4$. We will prove that the space of $J$-holomorphic curves is a Banach manifold via the implicit function theorem, realising it as a submanifold in an ambient space of maps. In this section we describe the ambient space. 

The naive first choice is just the space of all $C^{2,\alpha}$ maps, say, with \emph{no} restrictions near the boundary. This approach fails however: the $J$-holomorphic equation in this setting is not  Fredholm. This is because $J$ becomes singular at the boundary and, as we will see, this means the linearised Cauchy--Riemann operator is degenerate there: \emph{its symbol in directions normal to the curve vanishes at $t=0$}. As a consequence it is only Fredholm between the appropriately weighted $\Lambda$-Hölder spaces. This means that we need to also need to impose the asymptotic behaviour singled out by Proposition~\ref{asymptotic-expansion} to ensure that $\delb u$ decays quickly enough at the boundary to lie in the right weighted space. %This is the content of the main result of this section, Proposition~\ref{admissible-decay}. (See also Remark~\ref{CR-maps-between-weighted-holder-spaces}.)

Let $(\overline{\Sigma},j)$ be a compact Riemann surface with boundary. To begin, assume $\del \Sigma$ is connected. We fix a biholomorphism between $S^1 \times [0,\delta)$ and a neighbourhood of $\del \Sigma$ and use the cylindrical holomorphic coordinates $(s,t)$ there, in which $j(\del_s) = \del_t$. We extend $t$ to a global boundary defining function $t \colon\overline{\Sigma} \to [0,\infty)$. When $\del \Sigma$ has more than one component, we choose such a biholomorphism for each component, with non-overlapping images, and combine all the different $t$-coordinates into a single boundary defining function. 

\begin{definition}\label{admissible-maps}
A map $u \colon \overline{\Sigma} \to \overline{Z}$ is called \emph{admissible} if the pull-back of the twistor coordinates functions $x,y_i,z_i \colon \overline{\Sigma} \to \R$ satisfy the following properties:
\begin{enumerate}
\item
$(x,y_i, tz) \in C^{2,\alpha}$. Note the scaling on $z_i$. It follows that $z_i$ is $C^{2,\alpha}$ on the interior but only $C^{1,\alpha}$ up to the boundary. 
\item
The Taylor expansions of $x,y_i, tz_i$ with respect to $t$ at $t=0$ are given by
\begin{equation}
\begin{aligned}
x & = a(s) t + O(t^{2+\alpha})\\
y & = \gamma(s) + \phi(s)t^2 + O(t^{2+\alpha})\\
tz &= \zeta(s)t + \eta(s)t^2 + O(t^{2+\alpha})
\end{aligned}\label{admissible-taylor}
\end{equation}
where $\gamma \colon \del \Sigma \to \R^3$ is a $C^{2,\alpha}$ embedding and the other coefficients are determined by $\gamma$ via the formulae~\ref{taylor-coefficients-determined} of Proposition~\ref{asymptotic-expansion}.
\item
$x>0$ on the interior.
\end{enumerate}
We write $\A$ for the set of all admissible maps. We call $\gamma$ the \emph{boundary of $u$} and we write $\A_\gamma$ for the set of all admissible maps with boundary~$\gamma$.
\end{definition}

In the following result, we use the natural edge metric on $\overline{Z}$ and a $0$-metric on $\overline{\Sigma}$ (e.g.~the canonical complete hyperbolic metric) to define the metric on $\z{T\Sigma}^{0,1}\otimes u^*\e{TZ}$. Recall also the notation $C^{1,\alpha}_1$ of Definition~\ref{vanishing-at-boundary}.

\begin{proposition}\label{admissible-decay}
Let $u \colon \overline{\Sigma} \to \overline{Z}$ be a map for which the pull-back of the twistor coordinates satsify $(x,y_i,tz_i) \in C^{2,\alpha}$. Assume moreover that $u(\overline{\Sigma}) \cap \del Z = \del \Sigma$ and that this intersection is transverse. Then $u$ is admissible if and only if 
\[
\delb u \in C^{1,\alpha}_1(\z{T^*\Sigma}^{0,1}\otimes u^*\e{TZ}).
\]
\end{proposition}

\begin{proof}
The proof is essentially the same as that of Proposition~\ref{asymptotic-expansion}. By Lemma~\ref{edge-dbar} we know that $\delb u \in C^{1,\alpha}(\z{T^*\Sigma}^{0,1} \otimes u^*\e{TZ})$. The vector field $t\del_t$ (defined near $\del \Sigma$) is nowhere vanishing and of bounded length in $\z{T\Sigma}$. So it is enough to show that $u \in \A$ if and only if $\delb u(t \del_t) \in C^{1,\alpha}_1(u^*\e{TZ})$. The edge vector fields $x\del_x, x\del_{y_i}, \del_{z_i}$ give a bounded frame of $\e{TZ}$. Writing $\delb u(t\del_t)$ out with respect to this frame, it has coefficients given precisely by the left-hand side of the equations ~\eqref{x-dbar-u}, \eqref{y-dbar-u}, \eqref{z-dbar-u} in the proof of Proposition~\ref{asymptotic-expansion}. From here the proof proceeds identically. 
\end{proof}

\begin{lemma}\label{admissible-Banach-manifold}
Identifying $u \in \A$ with $(x,y_i,tz_i)$ makes the space $\A$ of admissible maps a Banach submanifold of $C^{2,\alpha}(\overline{\Sigma}, \R^7)$.
\end{lemma}
\begin{proof}
Taking the Taylor coefficients of a $C^{2,\alpha}$ function at $t=0$ defines a continuous linear surjection (as in \eqref{taylor-map} above):
\begin{align*}
T \colon C^{2,\alpha}(\overline{\Sigma}) 
	&\to 
		C^{2,\alpha}(\del \Sigma)\oplus 
		C^{1,\alpha}(\del \Sigma) \oplus 
		C^{0,\alpha}(\del \Sigma)\\
T (F) & =
		\left( F(s,0) , \del_t F(s,0), \del^2_t F(s,0) \right)
\end{align*}
Meanwhile, the map which sends $\gamma \in C^{2,\alpha}_{\mathrm{emb}}(\del \Sigma, \R^3)$ to the collection of Taylor coefficients given in~\eqref{admissible-taylor} is a smooth injection
\[
\iota \colon C^{2,\alpha}_{\mathrm{emb}}(\del \Sigma, \R^3) 
	\to 
		C^{2,\alpha}(\del \Sigma, \R^7)\oplus 
		C^{1,\alpha}(\del \Sigma, \R^7) \oplus 
		C^{0,\alpha}(\del \Sigma, \R^7)
\]
from the space of $C^{2,\alpha}$-embeddings $\del \Sigma \to \R^3$. If we write $\X$ for the image of $\iota$, it follows that $T^{-1}(\X)$ is a smooth submanifold of $C^{2,\alpha}(\overline{\Sigma}, \R^7)$.

Now $\A \subset T^{-1}(\X)$, determined by the requirements that $|z|^2=1$ (so that the map $u$ takes values in $\overline{Z}$) and $x>0$ on the interior. The first condition $|z|^2=1$ is transverse and so determines a submanifold $\mathscr{S} \subset T^{-1}(\X)$; the proof is the same as that which shows $S^2\subset \R^3$ is a submanifold. The second condition, $x>0$ on the interior, is open and so $\A$ is an open set of the submanifold $\mathscr{S}$ and so a submanifold of $C^{2,\alpha}(\overline{\Sigma}, \R^7)$ in its own right.
\end{proof}

\begin{lemma}\label{admissible-fixed-boundary-zero-holder}
Let $\gamma \colon \del \Sigma \to \R^3$ be a $C^{2,\alpha}$ embedding. Identifying $u \in \A_\gamma$ with $(x,y_i,tz_i)$ makes $\A_\gamma$ a submanifold of an affine Banach space modelled on $C^{2,\alpha}_2(\overline{\Sigma},\R^7)$, the subspace of functions which vanish at $\del \Sigma$ up to and including second order.
\end{lemma}

\begin{proof}
First note that, in the notation of the proof of Lemma~\ref{admissible-Banach-manifold},  $\A_\gamma = \A \cap T^{-1}(\iota (\gamma))$. The subset $\A_\gamma \subset T^{-1}(\iota(\gamma))$ is determined once again by $x>0$ on the interior and $|z|^2=1$ and again one checks that this is transverse. This means that $\A_\gamma$ is both a submanifold of $\A$ and also of $T^{-1}(\iota (\gamma))$. This second space is an affine space modelled on $\ker T = C^{2,\alpha}_2(\overline{\Sigma},\R^7)$. 
\end{proof}

\begin{remark}\label{tangent-space-A-gamma}
By Lemma~\ref{normal-and-zero-holder}, $C^{2,\alpha}_2(\overline{\Sigma}, \R^7) \cong t^{2+\alpha}\Lambda^{2,\alpha}(\overline{\Sigma}, \R^7)$. If we unscale the $z$ coordinates, $(x,y_i,tz_i) \mapsto (x,y_i,z_i)$, then we find that $\A_\gamma$ can also be seen as a submanifold of the following Banach space:
\[
\A_\gamma 
	\subset 
		t^{2+\alpha}\Lambda^{2,\alpha}(\overline{\Sigma},\R^4) \oplus 
		t^{1+\alpha}\Lambda^{2,\alpha}(\overline{\Sigma}, \R^3)
\]
Now recall that the edge vector fields $x\del_x, x\del_{y_i}, \del_{z_i}$ are a bounded frame for $\e{TZ}$.  When we write infinitesimal perturbations of $u \in \A_\gamma$ with respect to this frame, the different scaling of $x,y_i$ uses up a factor of $t$, putting all the directions on the same footing. So in an edge frame the coefficients are all of the \emph{same} weight, lying in $t^{1+\alpha}\Lambda^{2,\alpha}$. In other words,
\begin{equation}
T_u \A_\gamma = t^{1+\alpha} \Lambda^{2,\alpha}(u^*\e{TZ})
\label{tangent-space-A-gamma-is-weighted-edge}
\end{equation}
Meanwhile, again by Lemma~\ref{normal-and-zero-holder}, 
\[
C_1^{1,\alpha} (\z{T^*\Sigma}^{0,1}\otimes u^*\e{TZ})
\cong
t^{1+\alpha}\Lambda^{1,\alpha}(\z{T^*\Sigma}^{0,1}\otimes u^*\e{TZ})
\]
and so $\delb$ acting on admissible maps takes values in the weighted spaces to which the $0$-calculus applies. Suppose now that $u \in \A$ is $J$-holomorphic. If we vary $u$ in $\A_\gamma$ keeping the boundary fixed then, given the description~\eqref{tangent-space-A-gamma-is-weighted-edge} of $T_u\A_\gamma$, it follows from Proposition~\ref{admissible-decay} that the linearised Cauchy--Riemman operator of $u$ gives a map
\[
D_u \colon 
	t^{1+\alpha}\Lambda^{2,\alpha}(u^*\e{TZ}) 
		\to 
	t^{1+\alpha}\Lambda^{1,\alpha}(\z{T^*\Sigma}^{0,1}\otimes u^*\e{TZ})
\]
between the appropriate Hölder spaces of the same weight, exactly as is required to apply the $0$-calculus.  Ultimately, this is the technical justification for the definition of the space of admissible maps $\A$. 
\end{remark}

We now consider reparametrisations of an admissible map $u\in \A$. Given a vector field $v \in C^{2,\alpha}(T\overline{\Sigma})$ on $\overline{\Sigma}$ we ask when is $\diff u (v) \in T_u \A$? There are two things to consider here. Firstly, as always with maps of finite regularity, we must require that $u$ is \emph{more regular} than a generic point of $\A$. For example, if $u$ is only $C^{2,\alpha}$, then $\diff u(v)$ is $C^{1,\alpha}$ and so cannot possibly be in $T_u\A$. To avoid this, we assume that $u$ is $C^{3,\alpha}$. Then we must take care that $\diff u (v)$ infinitesimally preserves the constraints~\eqref{admissible-taylor} on the Taylor coefficients of $u$. This we do in the following Lemma. To state it, we first identify the tangent bundle $T\overline{\Sigma}$ with the holomorphic tangent bundle $T\overline{\Sigma}^{1,0}$. Then $\delb$ denotes the holomorphic structure on this line bundle. It takes values in sections of $T^*\overline{\Sigma}^{0,1}\otimes T\overline{\Sigma}^{1,0} \cong \End_a(T\overline{\Sigma})$ where the subscript $a$ denotes \emph{anti}-linear endomorphisms. 

\begin{lemma}\label{admissible-reparametrisations}
Let $u \in \A$ and of regularity $C^{3,\alpha}$ and let $v \in C^{2,\alpha}(T\overline{\Sigma})$. Then $\diff u(v) \in T_u \A$ if and only if $v$ is parallel to the boundary $\del \Sigma$ and $\delb v \in C^{1,\alpha}_1(\End_a(T\overline{\Sigma}))$.
\end{lemma}

\begin{proof}
This is essentially the linearisation of Proposition~\ref{admissible-decay}. That result shows the choice of Taylor coefficients~\eqref{admissible-taylor} is equivalent to $\delb u \in C^{1,\alpha}_1$. It then stands to reason that given such a map and a vector field $v$ which also satisfies $\delb v \in C^{1,\alpha}_1$, reparametrising by $v$ will preserve $\delb u \in C^{1,\alpha}_1$ and so the conditions~\eqref{admissible-taylor}. We now explain how to verify this by simple, albeit somewhat lengthy, calculation. 

Write $v = A \del_s + B \del_t$ where
\begin{align*}
A & = A_0(s) + A_1(s) t + A_2 (s) t^2 + O(t^{2+\alpha})\\
B & = B_0(s) + B_1(s) t + B_2(s) t^2 + O(t^{2+\alpha})
\end{align*}
Identifying $v$ with the section $\frac{1}{2}(A+iB)(\del_s - i \del_t)$ of the holomorphic tangent bundle $T\overline{\Sigma}^{1,0}$, one checks that $v$ is parallel to $\del \Sigma$ with $\delb v \in C^{1,\alpha}_1$ means simply that
\begin{equation}
\begin{gathered}
B_0 = 0,\quad B_1 = \dot{A}_0,\quad B_2 = 0 \\
A_1= 0,\quad A_2 = -\frac{1}{2}\ddot{A}_0
\end{gathered}
\label{taylor-admissible-vector-fields}
\end{equation}
(where the dots denote differentiation with respect to $s$).

Meanwhile, the infinitesimal changes in $x,y_i,z_i$ coming from $\diff u(v) = A \del_s u + B \del_t u$ are
\begin{align}
\delta x & = A \dot{a} t + B a + O(t^2)B + O(t^3)
	\label{delta-x}\\
\delta y & = A \dot{\gamma} + A \dot{\phi}t^2 + 2 B \phi t + O(t^2)B + O(t^3)
	\label{delta-y}\\
\delta z & = A \dot{\zeta} + A\dot{\eta} t  + B \eta + O(t^2)B + O(t^3)
	\label{delta-z}
\end{align}
where  $O(t^2)B$ means a term of the form $B$ times $O(t^2)$. Looking at~\eqref{delta-x} we see that $x = O(t)$ forces $B_0=0$. This means that $B = O(t)$ and so the $O(t^2)B$ terms above can all be absorbed into the $O(t^3)$ terms. 

Now from the formula~\eqref{delta-y} for $\delta y$, we see that the infinitesimal change of boundary value is $\delta \gamma = A_0 \dot \gamma$. Hence
\[
\delta a 
	= 
		a^{-1}\left\langle \dot{\gamma}, \delta \dot \gamma \right\rangle 
	=
		\dot{A}_0 a 
		+
		A_0 a^{-1}\left\langle \dot \gamma, \ddot \gamma \right\rangle
	=
		\dot{A}_0 a + A_0 \dot{a}
\]
Comparing this with the $O(t)$ term in $\delta x$ we see that the $O(t)$ Taylor coefficients of $x$ is preserved provided that
\[
B_1 a + A_0 \dot{a}
	= 
		\dot{A}_0 a
		+
		A_0 \dot{a}
\] 
which is equivalent to $B_1 = \dot{A}_0$. 

Next look at the $O(t)$ term in~\eqref{delta-y}, which is $A_1 \dot \gamma$. When~\eqref{admissible-taylor} is preserved, this must vanish, which forces $A_1=0$. Now the $O(t^2)$ term in~\eqref{delta-x} is $B_2 a$. Again this must vanish, which forces $B_2=0$. 

At this point we have shown that $B_0, B_2, A_1=0$ and $B_1 = \dot{A}_0$ is equivalent to preserving the conditions on all the Taylor coefficents of $x$ and on the $O(1)$ and $O(t)$ coefficients of $y$. It remains to show that also putting $A_2 = -\frac{1}{2}\ddot{A}_0$ is equivalent to preserving the condition on the $O(t^2)$ Taylor coefficient of $y$ and also the $O(1)$ and $O(t)$ terms in $z$. 

To this end, note that $\delta \gamma = A_0 \dot\gamma$ means
\[
\delta \zeta 
	= 
		\frac{ \delta a}{a^{2}} \dot\gamma - \frac{1}{a} \delta \dot \gamma
	%= 
	%	a^{-2}\left(\dot{A}_0a + A_0\dot{a}\right) \dot \gamma 
	%	- 
	%	a^{-1}\left(\dot{A}_0 \dot \gamma + A_0 \ddot \gamma \right)
	=
	 	A_0\left(\frac{\dot{a}\dot \gamma}{a^{2}} - \frac{\ddot \gamma}{a}\right)
	=
		A_0 \dot \zeta
\]
Comparing with~\eqref{delta-z} we see that the $O(1)$ Taylor coefficient in $z$ is automatically preserved. For the $O(t)$ term, recall the coefficient is $\eta = R(\zeta)(\dot\zeta)$ and so
\[
\delta \eta = R(\delta \zeta)(\dot \zeta)+ R(\zeta)(\dot \delta \zeta)
= R(\zeta)(A_0 \ddot \zeta + \dot{A}_0 \dot \zeta) 
= \dot{A}_0 \eta + A_0 R(\zeta)(\ddot{\zeta})
\]
Meanwhile, the corresponding term in~\eqref{delta-z} is
\[
A_0\dot{\eta} + B_1 \eta
=
A_0 (R(\dot{\zeta})(\dot\zeta) + R(\zeta)(\ddot \zeta)) + \dot{A}_0 \eta
=
\delta \eta
\]
and so again the $O(t)$ Taylor coefficient in $z$ is automatically preserved.

It remains to investigate the $O(t^2)$ term in the Taylor expansion of $y$, whose coefficient is $\phi = \frac{1}{2} \ddot{\gamma} + \dot{a}\zeta$. This means that
\begin{align*}
\delta \phi
	&=
		\frac{1}{2} \delta \ddot\gamma + \delta(\dot{a})\zeta + \dot{a} \delta \zeta\\
	&=
		\frac{1}{2}\left(
			\ddot{A}_0 \dot \gamma 
			+ 
			2\dot{A}_0 \ddot{\gamma} 
			+ 
			A_0 \dddot{\gamma}
		\right)
		+
		(\ddot{A}_0 a + 2\dot{A}_0 \dot{a} + A_0 \ddot{a})\zeta
		+
		A_0\dot{a}\dot{\zeta}
\end{align*}
Meanwhile, the $O(t^2)$ term in~\eqref{delta-y} is
\[
A_2 \dot{\gamma} + A_0 \dot \phi + 2B_1 \phi
=
A_2 \dot{\gamma} + A_0 \left( \frac{1}{2}\dddot{\gamma}  + \ddot{a}\zeta + \dot{a}\dot\zeta \right)+\dot{A}_0 \left(  \ddot{\gamma} + 2\dot{a}\zeta \right)
\]
Comparing with $\delta \phi$ we see that the constraint on the $O(t^2)$ Taylor coefficient of $y$ is preserved precisely when $A_2 = -\frac{1}{2}\ddot{A}_0$ as claimed (where we have used $a\zeta = - \dot{\gamma})$.
\end{proof}

\subsection{Admissible complex structures on the domain}

In what follows we will want to vary the complex structure $j$ on the domain $\overline{\Sigma}$. The definition of the space $\A$ of admissible maps depends on this choice of $j$: the definition uses a biholomorphism between $\del \Sigma \times [0,\delta)$ and a neighbourhood of $\del \Sigma \subset \overline{\Sigma}$ in order to talk about the coordinates $s$ and $t$ and, in particular, the Taylor coefficients in the $t$~direction. It will be convenient to restrict ourselves to complex structures on $\overline{\Sigma}$ which all give rise to the same space of admissible maps. We call these complex structures \emph{admissible} (Definition~\ref{def-admissible-j} and Lemma~\ref{admissible-j-give-same-maps}). We show that up to diffeomorphisms, we do not lose any generality: any complex structure on $\overline{\Sigma}$ can be made admissible by pulling back by  a diffeomorphism (Lemma~\ref{every-j-is-admissible-mod-diffeos}).  So the space of all admissible complex structures on $\overline{\Sigma}$ covers the whole moduli space. In Lemma~\ref{admissible-Teichmuller-slices} we give an ``admissible Teichmuller slice'' which is a slice for the Teichmuller action made entirely of admissible complex structures. 

To begin, we fix a reference complex structure $j_0$ on $\overline{\Sigma}$. We also fix a biholomorphism of $\del \Sigma \times [0,\delta)$ with a neighbourhood of $\del \Sigma \subset (\overline{\Sigma},j_0)$. The coordinates $s \in \del \Sigma$ and $t \in [0,\delta)$, with $j_0(\del_s) = \del_t$, then determine our space of admissible maps $\A$ as in Definition~\ref{admissible-maps}.

\begin{definition}\label{def-admissible-j}
A $C^{1,\alpha}$ complex structure $j$ on $\overline{\Sigma}$ is called \emph{admissible} if $j - j_0 \in C^{1,\alpha}_1$. We write $\J$ for the set of all admissible complex structures on $\overline{\Sigma}$, equipped with the $C^{1,\alpha}$ topology.
\end{definition}

Note that since $j$ and $j_0$ are equal on the boundary, they induce the same orientation on the whole of $\overline{\Sigma}$. 

\begin{lemma}
$\J$ is a closed submanifold of the space of all $C^{1,\alpha}$ complex structures on $\overline{\Sigma}$. 
\end{lemma}
\begin{proof}
The subspace $C^{1,\alpha}_1(\End(T\overline{\Sigma})) \subset C^{1,\alpha}(\End(T\overline{\Sigma}))$ is closed. The proof that $\J \subset j_0 + C^{1,\alpha}_1(\End(T\overline{\Sigma}))$ is a submanifold is now identical to that for the space of all complex structures with no restrictions. 
\end{proof}

\begin{lemma}\label{admissible-j-give-same-maps}
Let $j$ be an admissible complex structure on $\overline{\Sigma}$ and $u \colon \overline{\Sigma} \to \overline{Z}$ be a map for which the pull-back of the twistor coordinates satisfy $(x,y_i,tz_i) \in C^{2,\alpha}$. Assume moreover that $u(\overline{\Sigma}) \cap \del Z = \del \Sigma$ and that this intersection is transverse. Then $u$ is admissible (with respect to the reference $j_0$) if and only if 
\[
\delb_j u \in C^{1,\alpha}_1 (\z{T^*\Sigma}^{0,1}_j\otimes u^*\e{TZ}).
\]
where $\delb_j u = \frac{1}{2}\diff u + \frac{1}{2}J(u) \circ \diff u \circ j$ is defined using $j$. 
\end{lemma}

\begin{proof}
The definition of admissible complex structure implies that $\delb_j u - \delb_{j_0} u \in C^{1,\alpha}_1$. The result now follows from Proposition~\ref{admissible-decay}.
\end{proof}

\begin{lemma}\label{every-j-is-admissible-mod-diffeos}
Let $j$ be any $C^{1,\alpha}$ complex structure on $\overline{\Sigma}$, inducing the same orientation as $j_0$. There exists a $C^{2,\alpha}$ diffeomorphism $\phi$ of $\overline{\Sigma}$ for which $\phi^*j=j_0$ on a neighbourhood of the boundary. In particular, the admissible complex structures account for the whole moduli space of complex structures on $\overline{\Sigma}$.
\end{lemma}

\begin{proof}
Pick a $C^{1,\alpha}$ Riemannian metric $h$ on $\overline{\Sigma}$ in the conformal class determined by $j$. The Laplacian $\Delta$ of $h$ has $C^{0,\alpha}$ coefficients. We take a real-valued function $\tau$, defined on an neighbourhood of the boundary, which is harmonic $\Delta \tau = 0$, satisfies the Dirichlet boundary condition $\tau =0$ on $\del \Sigma$ and also has $\del_t \tau > 0$ on $\del \Sigma$, where $\del_t$ is the $t$-derivative of the coordinates $(s,t)$ in which $j_0(\del_s) = \del_t$. (Note that since we only need $\tau$ to be defined near $\del \Sigma$, we are able to fix both Dirichlet \emph{and} Neumann boundary values.) Since the coefficients of $\Delta$ are $C^{0,\alpha}$, elliptic regularity means $\tau \in C^{2,\alpha}$. Since $\del_t \tau >0$ on the boundary, it is a genuine boundary defining function. 

By construction, the 1-form $j\diff \tau$ is closed. We first observe that it has positive integral along each component of $\del \Sigma$. To see this, write $j(\diff t) = a \diff s + b \diff t$ for functions $a,b$. Since $j$ induces the same orientation as $j_0$,  $a>0$. The period of $j\diff \tau$ on a given component of $\del \Sigma$ is $\int_0^{2\pi} a \del_t \tau\, \diff s>0$. We now multiply $\tau$ by a constant near each boundary component so that $j \diff \tau$ has period $2\pi$ around each boundary component. It then integrates to give an $S^1=\R/2\pi$ valued function $\sigma$, defined near the boundary, with $\diff \sigma = j \diff \tau$. Since $\Delta \sigma \,\dvol_h= \diff j \diff \sigma = 0$, $\sigma$ is also $\Delta$-harmonic, hence $C^{2,\alpha}$. Now $\sigma + i \tau$ is a holomorphic coordinate for $j$. Identifying $s+it$ with $\sigma +i \tau$ defines a $C^{2,\alpha}$ diffeomorphism $\phi$ near the boundary for which $\phi^* j = j_0$. We then extend $\phi$ arbitrarily to a diffeomorphism on the whole of $\overline{\Sigma}$.
\end{proof}

\begin{definition}\label{def-admissible-diffeo}
A  $C^{2,\alpha}$ diffeomorphism $\phi$ of $\overline{\Sigma}$ is called \emph{admissible} if $\phi^*j_0$ is an admissible complex structure, i.e.\ $\phi^*j_0 - j_0 \in C^{1,\alpha}_1$.
\end{definition}

\begin{lemma}\label{admissible-diffeos-act}
If $\phi$ is an admissible diffeomorphism and $j$ is an admissible complex structure, then $\phi^*j$ is also admissible. It follows that the set of all admissible diffeomorphisms $\G$ is a group which acts on the space of all admissible complex structures $\J$.
\end{lemma}

\begin{proof}
First note that for any $C^{2,\alpha}$ diffeomorphism $\phi$ and any tensor $T \in C^{1,\alpha}_1$, $\phi^*T \in C^{1,\alpha}_1$. Now, for an admissible diffeomorphism $\phi$, 
\[
(\phi^{-1})^*j_0 - j_0 = (\phi^{-1})^* \left( j_0 - \phi^*j_0\right) \in C^{1,\alpha}_1
\]
It follows that $\phi^{-1}$ is also admissible. If $j$ is also admissible then
\[
\phi^*j -j_0 = \phi^*\left( j - j_0\right) +\phi^* \left( j_0 - (\phi^{-1})^*j_0\right) \in C^{1,\alpha}_1
\]
It follows that $\phi^* j$ is also admissible. Finally to check that $\G$ is closed under composition, if $\phi, \psi \in \G$, then by the above $(\phi \circ \psi)^* j_0 = \psi^*(\phi^*j_0)$ is an admissible complex structure which shows that $\phi \circ \psi \in \G$. 
\end{proof}

We close our discussion of admissible complex structures with an ``admissible'' version of Teichmüller slices. Recall that the Teichmüller space $\T(\overline{\Sigma})$ of $\overline{\Sigma}$ is the set of complex structures on $\overline{\Sigma}$, moduli the identity component of the diffeomorphism group. An infinitesimal change of the complex structure $j$ on $\overline{\Sigma}$ is an anti-linear endomorphism of $T\overline{\Sigma}$, and so there is a surjection
\begin{equation}
C^{\infty}(T^*\overline{\Sigma}^{0,1} \otimes T\overline{\Sigma}^{1,0}) \to T_{[j_0]}\T(\overline{\Sigma})
\label{Teichmuller-projection}
\end{equation}

\begin{definition}\label{Teichmüller-slice}
A \emph{Teichmüller slice} $\mathscr{S}$ is a vector subspace \[\mathscr{S} \subset C^{\infty}(T^*\overline{\Sigma}^{0,1}\otimes T\overline{\Sigma}^{1,0})\] which projects isomorphically onto $T_{[j_0]}\T(\overline{\Sigma})$ under~\eqref{Teichmuller-projection}. Equivalently, if we write $\delb$ for the holomorphic structure on the tangent bundle of $(\overline{\Sigma},j_0)$, then a Teichmüller slice $\mathscr{S}$ is a complement for the image of $\delb$ acting on sections of $T\overline{\Sigma}$ which are tangent to~$\del \Sigma$. 
\end{definition}

We need a small Lemma regarding the freedom in choice of Teichmüller slices. In principle, it follows from Lemma~\ref{every-j-is-admissible-mod-diffeos} that we can choose the elements of a Tecihmüller slice to vanish to infinite order at the boundary. We actually only need to use a slice in $C^{\infty}_1$, and we give a direct proof for simplicity.  

\begin{lemma}\label{admissible-Teichmuller-slices}
There exist Teichmüller slices $\mathscr{S} \subset C^{\infty}_1(T^*\overline{\Sigma}^{0,1} \otimes T\overline{\Sigma}^{1,0})$, i.e.\ $\mathscr{S}$ is composed entirely of sections of $T^*\overline{\Sigma}^{0,1} \otimes T\overline{\Sigma}^{1,0}$ which vanish to first order along $\del \Sigma$. 
\end{lemma}
\begin{proof}
Let $\phi \in C^\infty(T^*\overline{\Sigma}^{0,1} \otimes T\overline{\Sigma}^{1,0})$. We will find a vector field $v$, tangent to $\del \Sigma$, such that $\delb(v) + \phi \in C^{
\infty}_1$. Doing this for a basis $\phi_1, \ldots , \phi_d$ of a Teichmüller slice will give a new Teichmüller slice, spanned by $\delb(v_i) + \phi_i$, and so which lies in $C^{\infty}_1$

To do this, recall the holomorphic collar neighbourhood $S^1\times [0,\delta) \to \Sigma$ at a boundary component, with corresponding holomorphic coordinate $s+it$ where $s \in S^1$, $t \in [0,\delta)$. We treat $\phi$ as a section of $\End_a(T\overline{\Sigma})$ and with respect to $s+it$, we write
\[
\phi
=
P(s,t) (\diff s \otimes \del_s - \diff t \otimes \del_t) + Q(s,t) (\diff s \otimes \del_t + \diff t \otimes \del_s)
\]
We write a vector field in these coordiantes as $v = A(s,t)\del_s + B(s,t) \del_t$. Solving $\delb(v) + \phi\in C^{\infty}_1$ amounts to finding $A, B$ for which
\begin{equation}
\begin{gathered}
\del_t A + \del_s B + P \in C^{\infty}_1\\
\del_s A - \del_t B + Q \in C^{\infty}_1
\end{gathered}
\label{ABPQ}
\end{equation}
We do this via Taylor expansions with respect to $t$:
\begin{align*}
P & = P_0(s) + P_1(s) t + \cdots\\
Q & = Q_0(s) + Q_1(s) t + \cdots\\
A & = A_0(s) + A_1(s) t + A_2(s)t^2 + \cdots\\
B & = B_1(s)t + B_2(s)t^2 + \cdots
\end{align*}
where we have imposed $B(s,0)=0$ to ensure that $v$ is tangent to $\del \Sigma$. Then~\eqref{ABPQ} is equivalent to
\begin{align*}
A_1 + P_0 &=0\\
2A_2 + \dot{B}_1 + P_1 &=0\\
\dot{A}_0 - B_1 + Q_0 &=0\\
\dot{A}_1 - 2B_2 + Q_1&=0
\end{align*}
We solve this system by first taking $A_1=-P_0$. We then fix $B_1$ to be a constant chosen so that $Q_0-B_1$ has zero integral over $S^1$; this means we can find $A_0$ with $\dot{A}_0 = Q_0 -B_1$. Finally, we take $A_2 = -\frac{1}{2}P_1$ and $B_2 = \frac{1}{2}(\dot{P}_0 + Q_1)$.
\end{proof}

\begin{definition}\label{def-admissible-Teichmuller-slice}
A Teichmüller slice $\mathscr{S}$ is called \emph{admissible} if $\mathscr{S} \subset C^{\infty}_1$. We will always use admissible Teichmüller slices.
\end{definition}

\subsection{Asymptotics of the Cauchy--Riemann operator}
\label{normal-operator-section}

Let $u \colon \overline{\Sigma} \to \overline{Z}$ be admissible and $J$-holomorphic on the interior. In this section we compute the asymptotics of the linearised Cauchy--Riemann operator $D_u$ of $u$. The main result is that the normal operator of $D_u$, in the sense of the $0$-calculus, is invertible \emph{in the directions perpendicular to the curve} (Proposition \ref{normal-invertible} below).

The normal operator of $D_u$ turns out to be precisely the linearised Cauchy--Riemann operator of $u \colon \H^2 \to Z$ given by taking the twistor lift of a totally geodesic copy of $\H^2 \subset \H^4$. Geometrically, the reason for this is that if a minimal surface $M \subset \H^4$ passes through the point $p \in \del \H^4$ then under rescalings by hyperbolic isometries centred at $p$, $M$ converges to the totally geodesic copy of $\H^2$ which passes through $p$ and is tangent to $\del M$ at $p$.

Rather than following this geometric approach we instead give a proof by calculation, since in any case we will need a formula for the resulting normal operator. We begin by recalling a coordinate expression for the linearised Cauchy--Riemann operator. Fix a $J$-holomorphic map $u \colon (\Sigma,j) \to (Z,J)$ to an almost complex manifold and take coordinates $\xi_a$ on $Z$ and a holomorphic coordinate $s+it$ on $\Sigma$. We choose $s,t$ centred at a point of $\del \Sigma$ and chosen so that $t$ is a local boundary defining function. The linearised Cauchy--Riemann operator is the first order operator $D_u \colon \Gamma(u^*TZ) \to \Gamma(T^*\Sigma^{0,1}\otimes u^*TZ)$ defined by linearising the $J$-holomorphic equation
\[
u \mapsto \delb u :=  \diff u + J(u)\circ \diff u \circ j
\]
In coordinates, a section of $u^*TZ$ has the form $v = v^a\del_a$, where we are using the summation convention, and the shorthand $\del_a = \del_{\xi_a}$. Now $D_u(\del_t)(v)$ has the coordinate expression
\begin{equation}
D_{u}(\del_t)(v)^a = \del_t v^a - J^a_b\del_s v^b - v^c(\del_c J^a_b) \del_s u^b
\label{CR-in-coordinates}
\end{equation}
where $u^b$ and $J^a_b$ are the coordinate expressions for $u$ and $J$. 

This formula is particularly useful for us, since we have \emph{global} coordinate vector fields on the hyperbolic twistor space. More precisely, we have global coordinates $(x,y_i,z_i)$ on a 7-dimensional space $\hat{Z}$ and $Z = \{ |z|=1\} \subset \hat{Z}$. The endomorphism $J$ is defined on the whole of $T\hat{Z}$ (see the expression for $J$ in Lemma~\ref{J-in-local-coordinates}) and so~\eqref{CR-in-coordinates} defines $D_u$ acting on sections of the trivial bundle $u^*T\hat{Z}$. It is more convenient to calculate in this setting and then restrict back to sections of $u^*TZ$. 

Working in the coordinate frame $\del_x, \del_{y_i}, \del_{z_i}$ we have the following formulae, in which $\epsilon_1,\epsilon_2,\epsilon_3$ denotes the standard basis of $\R^3$ in column vectors. They come directly from~\eqref{CR-in-coordinates} together with Lemma~\ref{J-in-local-coordinates}. 
\begin{align}
D_{u}(\del_t)(\del_x)
  &=
    -(\del_x J)(\del_s u)=\begin{pmatrix}
    0 & 0 & 0\\
    0 & 0 & 0\\
    0 & 2x^{-2}P(z)& 0  
    \end{pmatrix}
    \begin{pmatrix}
    \del_s x \\
    \del_s y \\
    \del_s z  
    \end{pmatrix} \label{CR-dx}
 \\
D_{u}(\del_t)(\del_{y_i})
  &=
      - (\del_{y_i}J)(\del_s u)
    =
      0 \label{CR-dy}\\
D_{u}(\del_t)(\del_{z_i})
  &=
      - (\del_{z_i} J)(\del_s u)
    =  
    \begin{pmatrix}
      0 & \epsilon_i^T & 0 \\
      -\epsilon_i & -R(\epsilon_i) & 0\\
      0 & -2x^{-1}(\epsilon_iz^T + z\epsilon_i^T) &  R(\epsilon_i)
    \end{pmatrix}
    \begin{pmatrix}
      \del_s x \\
      \del_s y \\
      \del_s z  
    \end{pmatrix} \label{CR-dz}
\end{align}

In both \eqref{CR-dx} and \eqref{CR-dz}, the coefficient in the middle of the bottom row blows up at $t=0$. To obtain a $0$-differential operator, we will switch to the  edge tangent bundle, with frame $t\del_x, t\del_y, \del_z$. Similarly, we work with $\z{T^*\Sigma}$, by letting $D_u$ act on $t\del_t$. These introduce the extra factors of $t$ that are needed to produce a $0$-differential operator:

\begin{lemma}
$D_u$ extends to give a $0$-differential operator 
\[
D_u \colon \Gamma(u^*\e{TZ}) \to \Gamma(\z{T^*\Sigma}^{0,1}\otimes u^*\e{TZ})
\] 
\end{lemma}
\begin{proof}
The action of $D_{u}(t\del_t)$ on an arbitrary section of $u^*T\hat{Z}$ can be computed from the above expressions \eqref{CR-dx}, \eqref{CR-dx}, \eqref{CR-dz}, and the Leibniz formula
\begin{equation}
D_{u}(t\del_t)(fv) = (t\del_t f) v - (t\del_s f) Jv + f D_{u}(t\del_t)(v)
\label{Leibniz}
\end{equation}
To see that this defines a $0$-differential operator, one must check that when $v$ is one of edge coordinate vector fields $t\del_x, t\del_y, \del_z$ then $Jv$ and $D_{u}(t\del_t)(v)$ are given by linear combinations of the edge frame in which the coefficients are smooth up to $t=0$. This is easily verified from the expression for $J$ in Lemma~\ref{J-in-local-coordinates} and the Leibniz formula~\eqref{Leibniz} with $f(s,t)=t$. 
\end{proof}

As a Cauchy--Riemann operator, $D_{u}$ is elliptic in the interior. To apply the $0$-calculus we must compute the normal operator, check it is $0$-elliptic and prove that it is invertible between the appropriate spaces with weights constrained by the indicial roots. The normal operator is obtained by first writing everything in terms of $t\del_t$ and $t\del_s$, and then fixing $s$ and setting $t$ equal to zero in the coefficients. We begin by recalling the start of the expansion of $u$ near $t=0$:
\begin{align*}
x  &=|\dot \gamma(s)|t + O(t^2)\\
y & =\gamma(s) + O(t^2)\\
z & =\zeta(s) +O(t)
\end{align*}
where $\gamma \colon \del \Sigma \to \R^3$ is an embedding, and $\zeta = -|\dot \gamma|^{-1}\dot \gamma$. We can further ease the notation by rotating and rescaling our half-space coordinates $(x,y_i)$ so that at our chosen point $s$, we have $\zeta(s) = (1,0,0)$. The vectors which are tangent to $Z$ at this point $(x,y_i,1,0,0)$ are $t\del_x, t\del_{y_i}, \del_{z_2}, \del_{z_3}$. 

With this in hand we can compute:

\begin{lemma}\label{CR-edge-frame}
At $t=0$, the action of $D_{u}(t\del_t)$ on the edge coordinate fields is given by
\begin{align*}
D_{u}(t\del_t)(t\del_x)|_{t=0}
  &=
    t\del_x\\
D_{u}(t\del_t)(t\del_{y_i})|_{t=0}
  &=
    t\del_{y_i}\\
D_{u}(t\del_t)(\del_{z_2})|_{t=0}
  &=
    -t\del_{y_3} - 2\del_{z_2}\\
D_{u}(t\del_t)(\del_{z_3})|_{t=0}
  &=t \del_{y_2} - 2\del_{z_3}
\end{align*}
Here, by $D_{u}(t\del_t)(t\del_x)|_{t=0}$ etc.\ we mean first write $D_{u}(t\del_t)(t\del_x) = \phi t\del_x + \psi_i t\del_{y_i} + \theta_i \del_{z_i}$ and then evaluate the coefficients $\phi,\psi_i,\theta_i$ at $t=0$. 
\end{lemma}
\begin{proof}
For the first equation note that by Leibniz
\[
D_{u}(t\del_t)(t\del_x)
  =
    t\del_x
    +
    t^2D_{u}(\del_t)(\del_x)
\]
Looking at the above expression for $D_{u}(\del_t)(\del_x)$ we see an additional term with a coefficient of the form $2t^2x^{-2}P(z)(\del_s y)$; at $t=0$ this becomes $2 P(\zeta)(\dot\gamma) = 0$.

For the next expression, by Leibniz
\[
D_{u}(t\del_t)(t\del_{y_i})
  =
    t\del_{y_i}
    +
    t^2D_{u}(\del_t)(\del_{y_i})
  =
    t\del_{y_i}
\]
(where in fact this even holds for non-zero values of $t$; this reflects the translation invariance of the situation). 

Next we turn to $D_{u}(t\del_t)(\del_{z_2}) = t D_{u}(\del_t)(\del_{z_2})$. 
The term $-R(\epsilon_i)(\del_s y)$ in \eqref{CR-dz} contributes $-t\del_{y_3}$ here, whilst the term $-2x^{-1}(\epsilon_iz^T+z\epsilon_i^T)(\del_s y)$ contributes $-2\del_{z_2}$. The term $R(\epsilon_i)(\del_s z)$ gives no contribution, because it is multiplied by $t$ before we set $t=0$. This gives the formula for $D_{u}(t\del_t)(\del_{z_2})|_{t=0}$. The formula for $D_{u}(t\del_t)(\del_{z_3})|_{t=0}$ is similar.  
\end{proof}

To obtain an explicit formula for the normal operator, it turns out to be more convenient to use the following frame for $u^*\e{TZ}$:
\begin{equation}
e_0 = t\del_x,\quad
e_1 = t \del_{y_1},\quad
e_2  = t\del_{y_2} - \del_{z_3},\quad
e_3 = t\del_{y_3} + \del_{z_2}\quad
e_4 = -\del_{z_2}, \quad
e_5 = \del_{z_3}
\label{J-frame}
\end{equation}
The point here is that $e_0,e_1$ are tangent to the curve $u$ at $(s,0)$, whilst $e_2,e_3,e_4,e_5$ are normal to the curve; moreover, at $(s,0)$, $Je_0=e_1, Je_2=e_3, Je_4=e_5$ and so $J$ is standard in this basis. 

\begin{lemma}\label{whole-normal-operator}
With respect to the frame $e_i$, the normal operator $\D$ of $D_{u}(t\del_t)$ is given by
\[
\D 
\begin{pmatrix}
\phi_0\\ \phi_1 \\   \phi_2 \\ \phi_3 \\ \phi_4 \\ \phi_5
\end{pmatrix}
=
\begin{pmatrix}
  \tau\del_\tau +1 & -\tau\del_\sigma & 0 & 0 & 0 & 0\\
  \tau\del_\sigma  & \tau\del_\tau + 1 & 0 & 0 & 0 & 0\\
  0 & 0 & \tau\del_\tau & -\tau\del_\sigma & 0 & 1 \\
  0 & 0 & \tau\del_\sigma & \tau\del_\tau & 1 & 0 \\
  0 & 0 &  0 & 2 & \tau\del_\tau - 1& -\tau\del_\sigma \\
  0 & 0 & 2 & 0 & \tau\del_\sigma & \tau\del_\tau - 1
\end{pmatrix}
\begin{pmatrix}
\phi_0\\ \phi_1 \\   \phi_2 \\ \phi_3 \\ \phi_4 \\ \phi_5
\end{pmatrix}
\]
it follows that $\D$ and hence $D_{u}$ is $0$-elliptic.
\end{lemma}
\begin{proof}
It follows from Lemma~\ref{CR-dx} that, at $t=0$,
\begin{align*}
D_{u}(t\del_t)(e_0)  = e_0,&\quad 
D_{u}(t\del_t)(e_1)  = e_1,\\
D_{u}(t\del_t)(e_2)  = 2e_5, &\quad
D_{u}(t\del_t)(e_3)  = 2e_4,\\
D_{u}(t\del_t)(e_4)  = e_3 - e_4, &\quad
D_{u}(t\del_t)(e_5)  = e_2 - e_5.
\end{align*}
This, the Leibniz rule\eqref{Leibniz} and the fact that $J$ is standard in the frame $e_i$ give the stated formula for $\D$. For example,
\[
D_{u}(t\del_t)(\phi_4 e_4) = (t\del_t -1)\phi_4e_4 + (t\del_s \phi_4)e_5 + \phi_4 e_3.
\]
To check that $\D$ is $0$-elliptic, note that the first order terms are precisely $\tau$ times the standard Cauchy--Riemann operator on triples of complex-valued functions $(\phi_0+i\phi_1, \phi_2+i\phi_3, \phi_4+i\phi_5)$ of $z= \sigma + i \tau$.
\end{proof}

Write $T$ for the span of $e_0,e_1$ and $N$ for the span of $e_2,e_3,e_4,e_5$. Notice that $\D$ preserves both $T$ and $N$. $T$ corresponds to $\diff_eu(\z{T\Sigma}) \subset u^*\e{TZ}$, which is the tangent bundle of $u(\Sigma)$ whilst $N$ corresponds to the normal bundle of $u(\Sigma)$.

\begin{proposition}\label{normal-invertible}~
\begin{enumerate}
\item
The normal operator $\D$ has indicial roots $-1$ on $T$ and $-1,2$ on $N$. The roots are simple in the sense of Definition~\ref{0-calculus-definitions}. 
\item
In $N$ the subspace corresponding to the indicial root $-1$ is spanned by $t\del_{y_2}, t\del_{y_3}$ whilst that corresponding to the indicial root $2$ is spanned by $t\del_{y_2} - 2\del_{z_3},t \del_{y_3} - \del_{z_2}$. 
\item
When restricted to the normal bundle $N$, the normal operator $\D$ is invertible between spaces with weight $\delta \in (-1,2)$:
\[
\D \colon \tau^\delta \Lambda^{k+1,\alpha}(N) \to \tau^\delta \Lambda^{k,\alpha}(N).
\]
\end{enumerate}
\end{proposition}

\begin{remark}\label{not-really-blowing-up}
It is important to note that, whilst the indicial root $-1$ means that the corresponding sections of $N$ blow up as $t \to 0$, the actual vector fields remain \emph{finite}. This is because the sections $t\del_{y_2}, t\del_{y_3}$ which have bounded norm with respect to the natural metric on $N$, are actually $O(t)$ when considered as ordinary vector fields on $\R^3$. The upshot is that the $-1$ indicial root corresponds to genuine deformations of $u$.
\end{remark}

\begin{proof}
On $T=\left\langle e_0,e_1 \right\rangle$ the indicial polynomial of $\D$ is
\[
\det\begin{pmatrix}
\lambda + 1 & 0\\
0 & \lambda+1  
\end{pmatrix}
=
(\lambda+1)^2
\]
so $\D$ has a double root $\lambda = -1$ on $T$. Since $\dim (T) = 2$ this root is simple in the sense of Defintion~\ref{0-calculus-definitions}. On $N = \left\langle e_2,e_3,e_4,e_5 \right\rangle$, the indical polynomial of $\D$ is
\[
\det \begin{pmatrix}
  \lambda & 0 & 0 & 1\\
  0 & \lambda & 1 & 0\\
  0 & 2 & \lambda - 1 & 0\\
  2 & 0 & 0 & \lambda -1
\end{pmatrix}
= 
\det \left[ 
  \begin{pmatrix}
    \lambda & 1 \\
    2 & \lambda -1
  \end{pmatrix}
    \oplus
  \begin{pmatrix}
   \lambda & 1 \\
   2 & \lambda -1
  \end{pmatrix}
\right]
=
(\lambda+1)^2(\lambda -2)^2
\]
It follows the indicial roots of $\D$ on $N$ are $\lambda=-1,2$ again each with multiplicity two. Moreover, it is straightforward matter to check that when $\lambda=-1$ the kernel is spanned by $e_2+e_5 = t\del_{y_2}$ and $e_3+e_4 = t\del_{y_3}$ whilst the kernel when $\lambda=2$ is spanned by $e_2 - e_5 = t\del_{y_2} - 2\del_{z_3}$ and $e_3 - e_4 = t\del_{y_3} + 2 \del_{z_2}$. In each case the kernels have dimension~2 and so the indicial roots are simple.

We next turn to the invertibility of $\D$ on sections of $N$. We must show that there is a unique solution to the equation
\begin{equation}
\begin{pmatrix}
\tau\del_\tau & - \tau\del_\sigma & 0 & 1\\
\tau\del_\sigma & \tau\del_\tau & 1 & 0 \\
0 & 2 & \tau\del_\tau -1 & -\tau\del_\sigma\\
2 & 0 & \tau\del_\sigma & \tau\del_\tau -1
\end{pmatrix}
\begin{pmatrix}
  f\\g\\p\\q 
\end{pmatrix}
=
\begin{pmatrix}
  F\\G\\P\\Q
\end{pmatrix}  
\label{fgpq-eqn}
\end{equation}
when the functions $f,g ,p,q,F,G, P, Q$ lie in the appropriate weighted spaces. 

We carry out some simple manipulations which turn these equations into the Poisson equation for the hyperbolic Laplacian. The first two rows are equivalent to the equations
\begin{align}
q &= F -\tau\del_\tau f + \tau\del_\sigma g  \label{q-from-f-and-g}\\
p &= G - \tau\del_\sigma f - \tau\del_\tau g \label{p-from-f-and-g}
\end{align}
Substituting into the second rows we see that $f$ and $g$ must solve the following second order equations:
\begin{align}
2g 
+ 
(\tau\del_\tau - 1)(-\tau\del_\tau g - \tau\del_\sigma f) 
- 
\tau\del_\sigma(-\tau\del_\tau f + \tau\del_\sigma g)
  &=
    P - (\tau\del_\tau - 1)G + \tau\del_\sigma F
    \label{halfway-g}\\
2f 
+
t\del_s (-\tau\del_\sigma f - \tau\del_\tau g) 
+ 
(\tau\del_\tau -1) ( - \tau\del_\tau f + \tau\del_\sigma g)
  &=
    Q - \tau\del_\sigma G - (\tau\del_\tau -1)F
    \label{halfway-f}
\end{align}
Expanding out, we see that the terms in~\eqref{halfway-g} involving $f$ all cancel, leaving
\[
-(\tau\del_\tau - 1)\tau\del_\tau g - \tau^2 \del^2_{\sigma\sigma} g + 2g 
  =
    P - (\tau\del_\tau -1)G + \tau\del_\sigma F
\]
Now the first two terms are precisely the hyperbolic Laplacian and so this becomes
\begin{equation}
(\Delta + 2) g = P - (\tau\del_\tau-1)G +\tau\del_\sigma F
\label{g-eqn}
\end{equation}
where $\Delta = -(\tau\del_\tau - 1)\tau\del_\tau - \tau^2 \del^2_{\sigma\sigma}$ (the \emph{non-negative} Laplacian). Similarly, \eqref{halfway-f} rearranges to give
\begin{equation}
(\Delta+2) f = Q - \tau\del_\sigma G - (\tau\del_\tau-1)F
\label{f-eqn}
\end{equation}
The upshot of these manipulations is that solving~\eqref{fgpq-eqn} is equivalent to solving \eqref{g-eqn} and \eqref{f-eqn} and then defining $p$ and $q$ via \eqref{q-from-f-and-g} and \eqref{p-from-f-and-g}.

Everything then comes down to the invertibility of $\Delta +2$. This is a $0$-elliptic operator with indicial roots $-1$ and $2$ which are unsurprisingly the same as we started with. Since it is strictly positive  it has no kernel in $L^2$ and since it is self-adjoint it follows that it is invertible on $\H^2$ for weights $\delta \in (-1,2)$. It follows that $\D$ is also invertible. 

Strictly speaking we should also check the orders of derivatives line up. Suppose that $F,G,P,Q \in \tau^\delta \Lambda^{k,\alpha}$. Then the right-hand sides of \eqref{g-eqn} and \eqref{f-eqn} are in $\tau^\delta \Lambda^{k-1,\alpha}$. By elliptic regularity, $f,g$ lie in $\tau^\delta \Lambda^{k+1,\alpha}$ as required. 

From~\eqref{q-from-f-and-g} and~\eqref{p-from-f-and-g} it follows that $p,q\in \tau^\delta \Lambda^{k,\alpha}$ To prove that in fact $p, q \in \tau^\delta \Lambda^{k+1,\alpha}$ we return to the third and fourth lines of~\eqref{fgpq-eqn}. Taking $\tau\del_\tau$ of the third line, and $\tau\del_\sigma$ of the fourth gives
\begin{align*}
2\tau\del_\tau g 
+ 
\tau\del_\tau(\tau\del_\tau - 1) p 
- 
\tau^2\del^2_{\sigma \tau}q 
- 
\tau \del_\sigma q 
  &=
    \tau \del_\tau P\\
2\tau \del_\sigma f 
+ 
\tau^2\del_{\sigma\sigma}^2 p 
+ 
\tau^2\del^2_{\sigma\tau}q 
- 
\tau \del_\sigma q
  &=
    \tau \del_\sigma Q
\end{align*} 
These equations combine to give
\[
\Delta p = 2\tau \del_t g + 2\tau \del_\sigma f - 2\tau \del_\sigma q - \tau \del_\tau P - \tau \del_\sigma Q
\]
Crucially the terms which are second order in $q$ cancel. The right-hand side of this last equation is in $\tau^\delta \Lambda^{k-1,\alpha}$ and so by elliptic regularity $p \in \tau^\delta \Lambda^{k+1,\alpha}$ as required. A similar argument shows $q \in \tau^\delta \Lambda^{k+1,\alpha}$ also. 
\end{proof}

It is an important point that $\D$ is \emph{not} invertible on sections of $T$. Given a section $\phi_0 e_0 + \phi_1 e_1$ of $T$, $\D(\phi_0,\phi_1)$ can be identified with $\delb(\tau(\phi_0 + i\phi_1))$ where $\delb$ is the standard Cauchy--Riemann operator on acting on complex functions of $\sigma + i\tau$. This is because on sections of $u^*TZ$ which are tangent to $u(\Sigma)$, the Cauchy--Riemann operator $D_u$ of $u$ agrees with the natural Cauchy--Riemann operator on $T\Sigma$. From this description, one can see that $\D$ is injective: if $f \in \tau^\delta \Lambda^{k+1,\alpha}(\H^2,\C)$ for $\delta \in (-1,2)$ and if $\delb(\tau f)=0$ then $\tau f$ is a holomorphic function which vanishes on the real axis, which then forces $f=0$. The problem arises because $\D$ is not surjective. Indeed, the $L^2$-adjoint can be identified with $\tau^3 \delb (\tau^{-2}f)$ and this certainly has infinite dimensional kernel: let $F$ be any bounded holomorphic function on $\H^2$ which extends smoothly up to the boundary and then take $f = \tau^2F$; we have $f \in \tau^\delta \Lambda^{k,\alpha}$ and $\D^*f = 0$. This shows that $\D$ has infinite cokernel.
  
The upshot of this is that the $0$-calculus is not a convenient way to treat the problem in the tangential directions. This shouldn't be a surprise: the Cauchy--Riemann operator in these directions is elliptic in the classical sense and so we will use the classical theory here. 

\subsection{The normal bundle of \texorpdfstring{$u$}{u}}\label{normal-bundle-section}

As we have just explained, we will use different analytic approaches in the normal and tangential directions of $u$. With this in mind, we now explain how to split $u^*\e{TZ}$ into tangential and normal directions, over the whole of $\overline{\Sigma}$ and not just at a point on the boundary as we did in the discussion above of the normal operator. Recall the ``edge differential'' $\diff_e u \colon \z{T\Sigma} \to u^*\e{TZ}$ from Lemma~\ref{lift-d-to-edge}. Since $u$ is an embedding near $\del \Sigma$, $\diff_e u$ is an injection near $\del \Sigma$. When $u$ is globally an immersion, we can write $u^*\e{TZ} = \diff_e u(\z{T \Sigma}) \oplus N_u$ for the orthogonal decomposition, where $N_u \to \overline{\Sigma}$ is the ``edge normal bundle'' to $u$. 
%We identify sections of $u^*TZ$ with sections of $T\Sigma \oplus N$. Then the Cauchy--Riemann operator splits: given $\xi \in \Gamma(T\Sigma)$ and $\nu \in \Gamma(N_u)$,
%\[
%D_u 
%\begin{pmatrix} 
% \xi\\
% \nu
%\end{pmatrix}
%  =
%  \begin{pmatrix}
%    \delb_{T\Sigma} & A\\
%    0 & D_{N_u}
%  \end{pmatrix}
%  \begin{pmatrix} 
%   \xi\\
%   \nu
%  \end{pmatrix}
%\]
%where $D_{N_u}$ is a Cauchy--Riemann operator on sections of $N_u$ and $A \in \overline{T^*\Sigma} \otimes N_u^* \otimes T \Sigma$ is the ``second fundamental form'' of $\diff u \colon T\Sigma \hookrightarrow u^*TZ$. 

When $u$ has branch points, things are a little more complicated. It is standard that the image $u(\Sigma)$ still has a well defined tangent bundle, and we now recall how this goes. We work over the interior $\Sigma$. The linearised Cauchy--Riemann operator associated to $u$ is a \emph{real}-linear map. Taking the complex-linear part of $D_u$ defines a genuine $\delb$-operator on $u^*TZ$, making it into a holomorphic vector bundle over $\Sigma$. Moreover, since $\delb u=0$, one checks that $\diff u$ is a \emph{holomorphic} section of $T^*\Sigma^{1,0} \otimes u^*TZ$. This means that in a holomorphic local trivialisation for $u^*TZ$, and picking a local holomorphic coordinate $z$ on $\Sigma$ centred at a branch point of $u$, we can write $\diff u = z^k F(z) \otimes \diff z$ where $F$ is a $\C^3$-valued holomorphic function with $F(0) \neq 0$ and $k \in \N$ is the order of vanishing of $\diff u$ at $z=0$. For $z \neq 0$, $F(z)$ spans the tangent space to the image at $u(z)$ and we now see that it extends over the critical point to define a complex line subbundle $E \subset u^*TZ$ which plays the role of the tangent bundle to the image $u(\Sigma)$.

In the presence of branch points, the bundle $E$ is \emph{not} isomorphic to $T\Sigma$. All the zeros $p_1, \ldots , p_r$ of $\diff u$ together form the \emph{branching divisor} $B = \sum k_i p_i$, where $k_i \geq 1$ denotes the order of vanishing of $\diff u$ at $p_i$. The divisor $B$ determines a holomorphic line bundle $L_B \to \overline{\Sigma}$. We also fix once and for all a choice of holomoprhic section $\chi_B$ of $L_B$ whose zero divisor is precisely $B$. 

Now let $v$ be a section of $\z{T\Sigma} \otimes L_B$. Then $v \otimes \chi^{-1}_B$ is a $0$-vector field with poles at $B$. Since the poles are of at most the same order as the zeros of $\diff_eu$, the section $\diff_e u(v\otimes \chi^{-1}_B)$ extends over $B$. In this way $\diff_e u$ gives an injection $\z{T\Sigma} \otimes L_B \hookrightarrow u^*\e{TZ}$ with image equal to $E$, the tangent bundle to the image of $u$. We define the edge normal bundle $N_u \to \overline{\Sigma}$ to be the orthogonal complement of $E$. 

Whether there are branch points or not, the edge normal bundle carries a Cauchy--Riemann operator, given by composing orthogonal projection $p \colon u^*\e{TZ} \to N_u$ with~$D_u$.
\begin{equation}
D_{N_u} 
	\colon \Gamma(N_u) 
		\hookrightarrow \Gamma(u^*\e{TZ}) 
		\stackrel{D_u}{\longrightarrow} \Gamma(\z{T^*\Sigma}^{0,1} \otimes u^*\e{TZ})
		\stackrel{p}{\longrightarrow} \Gamma(\z{T^*\Sigma}^{0,1} \otimes N_u).
\label{normal-CR-defined}
\end{equation}
\begin{proposition}\label{Fredholm-normal-directions}
Let $u \colon \overline{\Sigma} \to \overline{Z}$ be admissible and holomorphic on the interior. Write $N_u \to \overline{\Sigma}$ for the normal bundle of $u$ as defined above. Let $\delta \in (-1,2)$. Then
\begin{enumerate}
\item
The Cauchy--Riemann operator of $N_u$ is a Fredholm map
\[
D_{N_u} \colon t^\delta \Lambda^{k+1,\alpha}(N_u) \to t^\delta\Lambda^{k,\alpha}(\z{T^*\Sigma}^{0,1} \otimes N_u).
\]
\item
The $L^2$-adjoint $D^*_{N_u}$ is a Fredholm map
\[
D^*_{N_u} \colon t^\delta \Lambda^{k+1,\alpha}(\z{T^*\Sigma}^{0,1} \otimes N_u) \to t^\delta\Lambda^{k,\alpha}(N_u).
\]
\item
There are natural identifications $\ker D_{N_u}^* \cong \coker D_{N_u}$ and $\ker D_{N_u} \cong \coker D_{N_u}^*$.
\end{enumerate}
\end{proposition}
\begin{proof}
This follows from the $0$-calculus, as outlined in \S\ref{0-calculus-review}. The vector field $t\del_t$ (defined near the boundary of $\Sigma$) has bounded length with respect to the metric on $\Sigma$. This means that to prove the claimed Fredholm property it suffices to consider the  behaviour of $D_{N_u}(t\del_t)$ (again defined only near the boundary) which has normal  operator given by the action of $\D$ on $N$. This  was shown to be an isomorphism for the relevant weights in Proposition~\ref{normal-invertible}. Notice that the indicial roots $-1,2$ of $D_{N_u}$ are symmetric about $1/2$. The result now follows from Theorem~\ref{0-Fredholm}
\end{proof}

%\begin{definition}\label{regular-curve}
%An $J$-holomorphic map $u \in \A$ is called \emph{non-degenerate} when the Cauchy--Riemann operator $D_{N_u}$ is surjective as a map between weighted spaces of weight $\delta \in (-1,2)$:
%\[
%D_{N_u} \colon t^\delta \Lambda^{2,\alpha}(N_u) \to \Lambda^{1,\alpha}(\overline{T^*\Sigma}\otimes N_u)
%\]
%When $D_{N_u}$ has non-zero cokernel we say that $u$ is \emph{degenerate}.
%\end{definition}
%
%Given a link $L \subset S^3$, we write $\M_L$ for the moduli space of $J$-holomorphic curves for which $\pi \circ u (\del\Sigma) = L$. As we will see later (in the proof of Proposition~\ref{index-dbeta}) non-degenerate curves give smooth points of~$\M_L$. Meanwhile, degenerate curves are still smooth points of the global moduli space (where no boundary condition is imposed). 

\subsection{The tangent bundle of \texorpdfstring{$u$}{u}}

In this section we show how the tangential part of the linearised equations are surjective, once we include the freedom to move the complex structure on $\overline{\Sigma}$. This is Theorem~\ref{tangential-surjective} below. 

We will apply Lemma~\ref{admissible-reparametrisations} and so we work with an admissible $J$-holomorphic map $u$, which is $C^{3,\alpha}$ up to the boundary. (Of course, elliptic regularity ensures $u$ is $C^{\infty}$ on the interior.) As above we denote the branching locus by $B = \sum k_i p_i$ and put $k=\sum k_i$ for the degree of $B$. Given a section $v$ of $T\overline{\Sigma}^{1,0}\otimes L_B$, we write $\hat{v} = v \otimes \chi_B^{-1}$ for the corresponding vector field with poles at $B$. (Recall we have fixed a holomorphic section $\chi_B$ of $L_B \to \overline{\Sigma}$ with zero divisor equal to $B$.)

\begin{definition}\label{admissible-vector-fields}
An \emph{admissible tangent vector field} is a section $v \in C^{2,\alpha}(T\overline{\Sigma}^{1,0}\otimes L_B)$ with the following properties 
\begin{enumerate}
\item
$\delb v \in C^{1,\alpha}_1$, where $\delb$ is the natural holomorphic structure on $T\overline{\Sigma} \otimes L_B$.
\item
$\hat{v} = v\otimes \chi_B^{-1}$ is parallel to $\del \Sigma$ (when interpreted as a section of the real tangent bundle $T\overline{\Sigma} \cong T\overline{\Sigma}^{1,0}$). 
\end{enumerate}
We write $\V_B$ for the set of all admissible tangent vector fields. $\V_B$ is a closed subspace of $C^{2,\alpha}$ and so is itself a Banach space.  
\end{definition}
By Lemma~\ref{admissible-reparametrisations}, when $v$ is an admissible tangent vector field, $\diff u (\hat{v}) \in T_u \A$ is an admissible deformation of $u$. Since $u$ is $J$-holomoprhic, $\diff u$ intertwines the $\delb$-operator with the Cauchy--Riemann operator $D_u$:
\begin{equation}
D_u (\diff u(\hat{v})) = \delb v
\label{intertwines}
\end{equation}
where on the right-hand side, we treat
$
\delb v \in C^{1,\alpha}(T^*\overline{\Sigma}^{0,1} \otimes T\overline{\Sigma}^{1,0} \otimes L_B)
$
as a section of
\begin{equation}
\begin{aligned}
T^*\overline{\Sigma}^{0,1} \otimes u^*\e{TZ} 
	&\cong 
		(\z{T^*\Sigma}^{0,1} \otimes \z{T\Sigma} \otimes L_B) 
		\oplus 
		(\z{T^*\Sigma}^{0,1}\otimes N_u)\\
	& \cong
		(T^*\Sigma^{0,1} \otimes T\Sigma^{1,0} \otimes L_B)
		\oplus
		(\z{T^*\Sigma}^{0,1} \otimes N_u)
\end{aligned}
\label{range-splitting}
\end{equation}
Here in the second line, we have used the natural isomorphism
\[
\z{T^*\Sigma} \otimes \z{T\Sigma} \cong T^*\Sigma \otimes T\Sigma.
\] 
(This is due to the fact that the $t$-factor in $\z{T\Sigma}$  cancels the $t^{-1}$-factor in $\z{T^*\Sigma}$.)

So surjectivity of $D_u$ in the tangential directions comes down to understanding the cokernel of the map
\begin{equation}
\delb \colon \V_B \to C^{1,\alpha}_1(T^*\overline{\Sigma} ^{0,1}\otimes T\overline{\Sigma}^{1,0} \otimes L_B)
\label{admissible-tangential-dbar}
\end{equation}
The main result of this section is the following (recall Definition~\ref{def-admissible-Teichmuller-slice} for the definition of an admissible Teichmüller slice).
\begin{theorem}\label{tangential-surjective}~
\begin{enumerate}
\item
The map $\delb \colon \V_B \to C^{1,\alpha}_1(T^*\overline{\Sigma}^{0,1} \otimes T\Sigma^{1,0} \otimes L_B)$ is Fredholm, with index $3\chi(\overline{\Sigma}) + 2k$, where $k$ is the degree of the branching divisor $B$.
\item 
Given an admissible Teichmüller slice $\mathscr{S} \subset C^{\infty}_1$, the map $\mathscr{S} \to \coker(\delb|_{\V_B})$ given by $s \mapsto  [s\otimes \chi_B]$ is surjective.
\end{enumerate}
\end{theorem}
\begin{proof}
We prove this first in the case that there are no branch points, $B =\emptyset$. We will work on the double: $S = \overline{\Sigma} \#_{\del \Sigma} \overline{\Sigma}'$ where $\overline{\Sigma}'$ is the same smooth surface as $\overline{\Sigma}$, but with the \emph{opposite} complex structure, $-j_0$. We write $\hat{j} = j_0\#(-j_0)$ for the complex structure on $S$ induced from $(\overline{\Sigma},j_0)$. The Teichmüller space $\T(S)$ of $S$ is of twice the dimension of $\T(\overline{\Sigma})$. $S$ comes with an anti-holomorphic involution $\sigma \colon S \to S$, given by flipping $\overline{\Sigma}$ and $\overline{\Sigma}'$. This induces an involution on $\T(S)$ whose fixed locus is the set of complex structures of the form $j\#(-j)$ for $j \in \T(\overline{\Sigma})$. In this way, the fixed points of $\sigma$ acting on $\T(S)$ are an embedded copy $\T(\overline{\Sigma}) \subset \T(S)$.

We need the infinitesimal version of this picture, namely for the interaction between $\sigma$ and the $\delb$ operator on the holomorphic tangent bundle of $(S, \hat{j})$:
\[
\delb \colon C^{2,\alpha}(TS^{1,0}) \to C^{1,\alpha}(T^*S^{0,1} \otimes TS^{1,0})
\]
Recall that $\coker(\delb) \cong T_{[\hat{j}]}\T(S)$. Now the involution $\sigma$ lifts to an antilinear bundle isomorphism of $TS$. This in turn induces a bundle endomorphism of $\End(TS)$ which preserves the sub-bundle $T^*S^{0,1} \otimes TS^{1,0}$ of antlinear endomorphisms. We continue to denote these actions on $TS$ and $T^*S^{0,1} \otimes TS^{1,0}$ by $\sigma$. If $\Gamma$ is a space of sections of one of these bundles, we write $\Gamma^+$ for those sections $\xi$ with $\sigma(\xi) = \xi$ and $\Gamma^-$ for those sections $\eta$ for which $\sigma(\eta) = -\eta$. An infinitesimal deformation $\phi$ of the complex structure $\hat{j}=j_0\#(-j_0)$ is induced by a deformation of $j$ precisely when $\sigma(\phi) = -\phi$. Moreover, $\sigma \circ \delb = - \delb \circ \sigma$, because $\sigma$ is antiholomorphic. It follows that $\delb$ restricts to the following map on $\sigma$-invariant vector fields, which we denote by $\delb^{\sigma}$ to indicate the change in range and domain:
\[
\delb^{\sigma} \colon C^{2,\alpha}(TS^{1,0})^+ \to C^{1,\alpha}(T^*S^{0,1} \otimes TS^{1,0})^-
\]
It is a simple matter to check that $\ker(\delb^{\sigma}) = \ker (\delb)^+$ and $\coker(\delb^{\sigma})= \coker(\delb)^-$. The inclusion $\coker(\delb^{\sigma}) \subset \coker(\delb)$ corresponds to $T_{[j_0]}\T(\overline{\Sigma}) \subset T_{[\hat{j}]}\T(S)$. In particular, $\delb^{\sigma}$ is Fredholm of index $3 \chi(\overline{\Sigma})$. 

We  now relate this to tensors on $\overline{\Sigma}$. A section $v \in C^{2,\alpha}(S, TS)^+$ is completely determined by its values on $\overline{\Sigma}$. The holomorphic coordinate $z= s+it$ near $\del \Sigma$ extends to $S$ in which the involution is given by $\sigma(s+it) = s-it$. Write $v = A(s,t)\del_s + B(s,t) \del_t$. Then $\sigma (v)=v$ means that $A(s,t) = A(s,-t)$ is even in $t$ and $B(s,t) = - B(s,-t)$ is odd. In particular, $b(s,0)=0$ and so $v$ is tangent to $\del \Sigma$. Conversely given $v \in  C^{2,\alpha}(\overline{\Sigma}, T\overline{\Sigma})$ which is tangent to $\del \Sigma$, it extends in a unique way to a \emph{continuous} $\sigma$-invariant section of $TS \to S$. The extension is $C^{2,\alpha}$ over the boundary precisely when the coefficients $a,b$ have even and odd Taylor expansions in $t$, up to $O(t^2)$; i.e.\ $A(s,t) = A_0(s) + A_2(s)t^2 + O(t^{2+\alpha})$ and $B(s,t) = B_1(s)t + O(t^{2+\alpha})$. Significantly for us, \emph{all admissible tangent vector fields satsify these conditions}, as one sees from~\eqref{taylor-admissible-vector-fields}. In this way, have an inclusion
\[
\V \hookrightarrow C^{2,\alpha}(S, TS^{1,0})^+
\]

There is a similar discussion for $\sigma$-anti-invariant  sections of $T^*S^{0,1}\otimes TS^{1,0}$. Let $\phi$ be a section of this bundle, regarded as an anti-linear endomorphism of the real tangent bundle $TS$. Write it locally in our holomorphic chart as
\[
\phi = P(s,t) (\diff s \otimes \del_s - \diff t \otimes \del_t) + Q(s,t) (\diff s \otimes \del_t + \diff t \otimes \del_s)
\]
Then $\sigma(\phi) = - \phi$ means that  $P(s,t)=-P(s,-t)$ is odd in $t$ and $Q(s,t) = Q(s,-t)$ is even. Given a section $\phi \in C^{1,\alpha}(\overline{\Sigma}, T^*\overline{\Sigma}^{0,1} \otimes T\overline{\Sigma}^{1,0})$, with $P(s,0)=0$ it extends in a unique way to a \emph{continuous} $\sigma$-anti-invariant section over all of $S$. This extension is $C^{1,\alpha}$ over the boundary precisely when $P,Q$ have even and odd Taylor expansions in $t$ respectively, up to $O(t)$; i.e.\ $P(s,t) = P_1(s)t + O(t^{1+\alpha})$ and $Q(s,t) = Q_0(s) + O(t^{1+\alpha})$. Significantly for us, \emph{if $\phi$ vanishes to first order along $\del \Sigma$, then it satisfies these conditions}. To ease the notation, write $\W = C^{1,\alpha}_1(T^*\overline{\Sigma}^{0,1} \otimes T\overline{\Sigma}^{1,0})$ for the space of sections which vanish to first order at $\del \Sigma$. We have an inclusion 
\[
\W \hookrightarrow C^{1,\alpha}(T^*S^{0,1} \otimes TS^{1,0})^-.
\]

By construction, $\delb^{\sigma}$ restricts to a map $\delb|_\V \colon \V \to \W$, which is precisely the $\delb$-operator in the statement of the theorem. Now $\ker (\delb|_\V) = \ker (\delb^{\sigma})$. Meanwhile, there is a natural map $\coker(\delb|_\V) \to \coker(\delb^{\sigma})$. This is injective, because $(\delb^{\sigma})^{-1}(\W) = \V$. To see that it is surjective, given any $\phi \in C^{1,\alpha}(S, T^*S^{0,1}\otimes TS^{1,0})^-$, we will write down $v \in C^{2,\alpha}(S,TS^{1,0})^+$ with the property that $\delb(v) + \phi \in \W$. The construction of $v$ is identical to that which appears in proof of Lemma~\ref{admissible-Teichmuller-slices}. Note that since $\phi$ is $\sigma$-anti-invariant, $P_0=0=Q_1$. Following the recipe in Lemma~\ref{admissible-Teichmuller-slices}, we see that $B_0=0$, $A_1= 0$ and $B_2 = 0$, which ensures that the vector field $v$ is indeed $\sigma$-invariant. The conclusion is that $\coker(\delb|_\V) = \coker(\delb^{\sigma})$ and so $\ind (\delb|_\V) = 3 \chi(\overline{\Sigma})$ as required. Moreover, given a Teichmüller slice $\mathscr{S}\subset \W$, by definition it projects isomorphically onto $\coker(\delb^{\sigma})$ and hence also onto $\coker(\delb|\V)$. This proves the Theorem in the case $B= \emptyset$.

Now we treat the case when $B$ is non-empty. Let $\hat{B} = B + \sigma(B)$, a divisor in $S$. We write $L_{\hat{B}}$ for the associated holomorphic line bundle. There is a short exact sequence of sheaves:
\[
0 \to \O(TS^{1,0}) \to \O(TS^{1,0} \otimes L_{\hat{B}}) \to \CC \to 0
\]
where $\CC$ is a skyscraper sheaf supported at the points of $\hat{B}$. The stalk $\CC_p$ of $\CC$ at the point $p \in \hat{B}$ is the space of $(r-1)$-jets of holomorphic sections of $TS^{1,0} \otimes L_{\hat{B}}$ taken at $p$, where $r$ is the multiplicity of $p \in \hat{B}$. From the short exact sequence of sheaves, we have the following long exact sequence
\[
0 \to \ker (\delb) \to \ker (\delb_{\hat{B}}) \to \bigoplus_{p \in \hat{B}} \CC_p \to \coker(\delb) \to \coker(\delb_{\hat{B}}) \to 0
\]
where $\delb$ is the holomorphic structure on $TS^{1,0}$ and $\delb_{\hat{B}}$ that on $TS^{1,0} \otimes L_{\hat{B}}$. 

Next we bring in the action of the involution $\sigma$. This lifts to the line bundle $L_{\hat{B}}$. It anti-commutes with $\delb_{\hat{B}}$ and so, as before, we can restrict to get a map 
\[
\delb^{\sigma}_{\hat{B}} 
	\colon 
		C^{2,\alpha}(TS^{1,0}\otimes L_{\hat{B}})^+ 
		\to
		C^{1,\alpha}(T^*S^{0,1} \otimes TS^{1,0} \otimes L_{\hat{B}})^-
\]
where we use the superscript $\sigma$ to indicate the change in domain and range.  The involution $\sigma$ acts on each term in the long exact sequence, leading to a new sequence
\[
0 \to \ker (\delb^\sigma) \to \ker (\delb^\sigma_{\hat{B}}) \to \bigoplus_{p \in B} \CC_p \to \coker(\delb^\sigma) \to \coker(\delb^\sigma_{\hat{B}}) \to 0
\] 
(Note that we now only sum the stalks of $\CC$ over the points of $B$ rather than $\hat{B}$.) It follows that 
\[
\ind(\delb^\sigma_{\hat{B}}) = \ind (\delb^\sigma) + 2k = 3\chi(\overline{\Sigma}) + 2k.
\]
Moreover, the fact that $\coker(\delb^\sigma) \to \coker(\delb^\sigma_{\hat{B}})$ is surjective means that a Teichmüller slice surjects onto $\coker(\delb^\sigma_{\hat{B}})$. To go from here to the statement of the Theorem, concerning admissible vector fields, we argue exactly as we did when $B= \emptyset$. 
\end{proof}

\subsection{Smoothness of the moduli space}
\label{smooth-moduli-space-section}

Let $\overline{\Sigma}$ be a connected compact surface with genus $g$ and boundary $\del \Sigma$ having $c>0$ components. We fix a reference complex structure $j_0$ on $\overline{\Sigma}$. This determines a space $\A$ of admissible maps (see Definition~\ref{admissible-maps}) and a space $\J$ of admissible complex structures (see Definition~\ref{def-admissible-j}).

\begin{definition}\label{definition-moduli-space}
Write $\X_{g,c}^\infty = \{(u,j)\}$ for the set of pairs $(u,j)$ where
\begin{enumerate}
\item
$ u \in \A$ is $C^{\infty}$.
\item
$j \in \J$ is $C^{\infty}$.
\item
$u$ is $(J,j)$-holomorphic on the interior $\Sigma$. 
\end{enumerate}
We write $\X_{g,c} \subset \A \times \J$ for the closure of $\X_{g,c}^\infty$; i.e.\ with respect to the $C^{2,\alpha}$ topology on the coordinate functions of the admissible map $u = (x,y_i,tz_i)$ and the $C^{1,\alpha}$-topology on the complex structure $j$. 

The \emph{moduli space of $J$-holomorphic curves in $\overline{Z}$} with genus $g$ and $c$ boundary components~is
\[
\M_{g,c} = \X_{g,c}/\sim
\]
where $(u,j) \sim (u',j')$ if they are related by a $C^{2,\alpha}$ admissible diffeomorphism of $\overline{\Sigma}$ (see Definition~\ref{def-admissible-diffeo} and Lemma~\ref{admissible-diffeos-act}).
\end{definition}

\begin{remark}
An a~priori larger collection of $J$-holomorphic curves comes from considering \emph{all} pairs $(u,j) \in \A \times \J$ for which the admissible map $u$ is $J$-holomorphic on the interior. Such a pair $(u,j)$ will fail to lie in $\X_{g,c}$ if it cannot be approximated by $J$-holomorphic curves which are smooth up to the boundary.  The reason we pass via the closure of $\X^{\infty}_{g,c}$ is to avoid this pathology, and work in a setting in which smooth curves are dense. It seems unlikely that such a pathological map $u$ exists, although we have been unable to rule it out. 
\end{remark}

We are now finally in position to prove the first major result of this article. 

\begin{theorem}\label{smooth-moduli-space}
$\M_{g,c}$ is a smooth Banach manifold.
\end{theorem}

\begin{proof}
We give the proof first in the case that $\chi(\overline{\Sigma}) <0$. We will treat the other two cases, when $\overline{\Sigma}$ is the disc or an annulus, afterwards.

It suffices to prove that $\X_{g,c} \subset \A\times \J$ is a submanifold. This is because the action of admissible diffeomorphisms on $\X_{g,c}$ is free and  proper. To see this, note that it is almost true for $\J$ alone. It is a standard result of Teichmüller theory that action of diffeomorphisms on $\J$ is proper and the stabiliser of a point $j\in \J$ is the finite group of automorphisms of $(\overline{\Sigma},j)$. But these automorphisms act freely on $\A$, because the maps in $\A$ are \emph{embeddings} on $\del \Sigma$. So once we establish that $\X_{g,c}$ is a submanifold, the manifold structure will descend to the quotient $\M_{g,c}$. 

$\X_{g,c} \subset \A\times \J$ is cut-out by the $J$-holomoprhic curve equation $\delb_ju = 0$. 
By Lemma~\ref{admissible-j-give-same-maps}, for $(u,j) \in \A \times \J$, 
\[
\delb_j u \in C^{1,\alpha}_1(\z{T^*\Sigma}^{0,1}_j \otimes u^*\e{TZ})
\]
To prove that $\X_{g,c}$ is a submanifold, it suffices to show that if $\delb_ju=0$ then the linearised Cauchy--Riemann operator is surjective as a map
\[
 T_u \A \oplus T_{j} \J \to C^{1,\alpha}_1(\z{T^*\Sigma}^{0,1} \otimes u^*\e{TZ})
\]
The result will then follow from the implicit function theorem. 

A first remark: since $\X^{\infty}_{g,c}$ is dense in $\X_{g,c}$, we can also assume that both $u$ and $j$ are smooth. Secondly, we might as well have used $j$ as our ``reference'' complex structure on $\overline{\Sigma}$ (denoted $j_0$ above) since, by Proposition~\ref{admissible-decay} and Lemma~\ref{admissible-j-give-same-maps} $j$ and $j_0$ determine the same admissible maps and so are interchangable. From here on, we assume then that $j$ is ``the'' reference point in $\J$.

We now pick an admissible Teichmüller slice $\mathscr{S} \subset T_{j} \J$ (see Definitions~\ref{Teichmüller-slice} and~\ref{def-admissible-Teichmuller-slice}). To show that $\X_{g,c} \subset \A \times \J$ is a smooth Banach manifold, we will prove that the linearised Cauchy--Riemann operator with the following domain is surjective:
\[
D_{(u,j)} \colon T_u\A \oplus \mathscr{S} \to C_1^{1,\alpha}(\z{T^*\Sigma}^{0,1}\otimes u^*\e{TZ})
\]
(The point of the Teichmüller slice is to work transverse to the orbits of the diffeomorphism group.)

Recall the splitting~\eqref{range-splitting}:
\[
\z{T^*\Sigma}^{0,1} \otimes u^*\e{TZ} 
	\cong 
		(T^*\overline{\Sigma}^{0,1}\otimes T\overline{\Sigma}^{1,0} \otimes L_{B}) 
		\oplus 
		\z{T^*\Sigma}^{0,1} \otimes N_u
\]
With respect to this we decompose $D_{(u,j)} = D' + D''$ where $D'$ takes values in the first summand above, and $D''$ takes values in the second. We know from~\eqref{intertwines} that when $v \in \V_B$ is an admissible tangent vector field, 
\begin{equation}
D_{(u,j)}(\diff u(v), 0) = \delb_{B}( v) 
\label{intertwines-2}
\end{equation}
(Note that here is it essential that we are working with $u$ which is $C^{\infty}$, or at least $C^{3,\alpha}$. If $u$ were merely $C^{2,\alpha}$, then $\diff u(v)$ would be $C^{1,\alpha}$ which is not of sufficient regularity to lie in $T_u \A$.) Meanwhile a direct calculation shows that if we vary $j$ in the direction of $s \in \mathscr{S}$, then 
\begin{equation}
D_{(u,j)}(0,s) = s \otimes \chi_B
\label{vary-j}
\end{equation}
where we tensor with $\chi_B$ to take care of the fact that we are identifying $s$ with a section of $T^*\overline{\Sigma}^{0,1} \otimes T\overline{\Sigma}^{1,0}\otimes L_B$. Now Theorem~\ref{tangential-surjective} tells us that 
\[
D' = D_{(u,j)} \colon \V_B \oplus \mathscr{S} \to C^{1,\alpha}_1(T^*\overline{\Sigma}^{0,1}\otimes T\overline{\Sigma}^{1,0} \otimes L_{B})
\]
is surjective. Note that neither~\eqref{intertwines-2} nor~\eqref{vary-j} have a normal component and so on deformations of the form $(\diff u(v),s)$, we have that $D' = D_{(u,j)}$, with no need to project tangentially.

The crux of the proof is to show that $D'' \colon T_u \A \to C^{1,\alpha}_1(\z{T^*\Sigma}^{0,1} \otimes N_u)$ is  surjective. Assuming this momentarily, the proof is finished as follows. Let 
\[
\phi = \phi' + \phi'' \in C^{1,\alpha}_1(\z{T^*\Sigma}^{0,1} \otimes u^*\e{TZ})
\] 
where $\phi'$ is a section of $T^*\overline{\Sigma}^{0,1} \otimes T\overline{\Sigma}^{1,0} \otimes L_{B}$ and $\phi''$ is a section of $\z{T^*\Sigma}^{0,1}\otimes N_u$. We first solve $D''(\xi) = \phi''$ (assuming surjectivity of $D''$). This means that $D_{(u,j)}(\xi) = A(\xi) + \phi''$ where $A(\xi)$ is a section of $T^*\overline{\Sigma}^{0,1} \otimes T\overline{\Sigma}^{1,0} \otimes L_{B}$. Now we solve $D'(\diff u(v),s) = \phi' - A(\xi)$. Then
\[
D_{(u,j)}(\diff u(v) + \xi, s) = \phi' - A(\xi) + \phi'' + A(\xi) = \phi
\]
which proves surjectivity of $D_{(u,j)}$.

We use the $0$-calculus to prove that $D'' \colon T_u \A \to C^{1,\alpha}_1(\z{T^*\Sigma}^{0,1} \otimes N_u)$ is surjective. First, note that (by Remark~\ref{tangent-space-A-gamma})
\[
t^{1+\alpha}\Lambda^{2,\alpha}(u^*N_u) \subset T_u\A
\]
On this subspace, $D''=D_{N_u}$ and we have already shown in Proposition~\ref{Fredholm-normal-directions} that $D_{N_u}$ is Fredholm between the following spaces:
\begin{equation}
D_{N_u} 
	\colon 
		t^{1+\alpha}\Lambda^{2,\alpha}(N_u) 
		\to 
		t^{1+\alpha}\Lambda^{1,\alpha}(\z{T^*\Sigma}^{0,1} \otimes N_u)
\label{DNu}
\end{equation}
The domain of~\eqref{DNu} is not the whole of $T_u\A$; the perturbations do not move the boundary $\gamma$ of $u$. We now use the additional freedom to move the boundary to eliminate the cokernel of~\eqref{DNu}. 

To do this, we will use an integration-by-parts formula. To set the notation up, let $\psi \in t^2 \Lambda^{1}(\z{T^*\Sigma}^{0,1} \otimes N_u)$. Near the boundary $\del \Sigma$ we write
\[
\psi = t^2 \left( \hat\psi \otimes \frac{\diff s^{0,1}}{t}\right)
\]
where $\hat\psi$ is a $\Lambda^{1}$ section of $N_u$. (The point here is that $t^{-1}\diff s$ has bounded norm with respect to the $0$-metric on $\overline{\Sigma}$.) We write $\psi_0 = \hat\psi|_{\del \Sigma}$ for the boundary value of $\psi$. Similarly, given $\phi \in t^{-1}\Lambda^{1}(N_u)$ we write $\phi = t^{-1}\hat{\phi}$ near $\del \Sigma$ for $\phi$ a $\Lambda^{1}$ section of $N_u$ and call $\phi_0 = \hat{\phi}|_{\del \Sigma}$ the boundary value of $\phi$. 

With this notation in hand, we claim the following is true. Given
\begin{itemize}
\item
$\phi \in t^{-1} \Lambda^{1}(N_u) \cap t^{-3/2}W^{2,1}(N_u)$
\item
$\psi \in t^2\Lambda^{1}(\z{T^*\Sigma}^{0,1} \otimes N_u) \cap t^{3/2}W^{2,1}(\z{T^*\Sigma}^{0,1} \otimes N_u)$
\end{itemize}
(where the Sobolev norms $W^{2,1}$ are defined using the $0$-metric $h$ on $\overline{\Sigma}$) we have:
\begin{equation}
\int_\Sigma (D_{N_u} \phi ,\psi) \dvol_h
	=
		\int_\Sigma (\phi, D_{N_u}^*\psi) \dvol_h
		+
		\int_{\del \Sigma} (\phi_0,\psi_0) \diff s
\label{integrate-by-parts}
\end{equation}
To see that this is true, first note that the choice of the weighted Hölder spaces for $\phi,\psi$ ensures that their boundary values make sense as continuous sections over $\del \Sigma$, whilst the choice of weighted Sobolev spaces ensures that the $L^2$-innerproducts over $\Sigma$ are defined. We now apply Stokes's theorem over the region $t \geq \epsilon$ for $\epsilon >0$. This gives
\[
\int_{t \geq \epsilon} (D_{N_u} \phi, \psi) \dvol_h
	=
		\int_{t\geq \epsilon} (\phi, D_{N_u}^* \psi) \dvol_h
		+
		\int_{t=\epsilon} \left\langle \phi,\psi \right\rangle
\]
Here $\left\langle \phi,\psi \right\rangle$ denotes the 1-form on the curve $t=\epsilon$ obtained from contracting $\phi$ and $\psi$ via the Hermitian metric on $N_u$ and then restricting to $t=\epsilon$. Given the expansions of $\phi$ and $\psi$, this is equal to $(\hat{\phi},\hat{\psi})\diff s$. Equation~\ref{integrate-by-parts} then follows by taking $\epsilon$ to zero. 

We now apply~\eqref{integrate-by-parts} to prove the surjectivity of $D''$. Let $\psi \in t^{1+\alpha}\Lambda^{1,\alpha}(\z{T^*\Sigma}^{0,1} \otimes N_u)$, and assume that $\psi$ is $L^2$-orthogonal to the image of $D''$. We will show that $\psi=0$. To begin note that $\psi$ is orthogonal to the image of $D_{N_u}$ acting as in~\eqref{DNu}. The $0$-calculus assures us that $\psi \in \ker D_{N_u}^*$ and, since the indicial roots of $D_{N_u}^*$ are $-1,2$, that $\psi \in t^2\Lambda^{1,\alpha} \cap t^{3/2}W^{2,1}$. Moreover, by Theorem~\ref{0-Fredholm}, the boundary value $\psi_0$ is a section of the sub-bundle of $\z{T^*\Sigma}^{0,1}\otimes N_u \to \del \Sigma$ which corresponds to the indicial root $2$. By Proposition~\ref{normal-invertible} this is precisely the normal bundle to boundary link $K\defeq \pi( u (\del \Sigma))$

We will now vary $K \subset S^3$ in the space of $C^{2,\alpha}$ submanifolds via a $C^{2,\alpha}$ section $\xi \in C^{2,\alpha}(E)$ of its normal bundle to $E \to K$. By definition of $\A$, we can construct a tangent vector $\phi \in T_u\A$ of the form $\phi \in t^{-1}\xi +C^{\infty}(N_u)$ in a collar neighbourhood of $\del \Sigma$. (See Remark~\ref{not-really-blowing-up} for an explanation of how $\phi$ appears to blow up with respect to the natural metric on $N_u$, and yet corresponds to a genuine deformation of $u$.) In particular, $\phi_0 =\xi$. Now for \emph{any} choice of $\phi \in t^{-1}\Lambda^{1,\alpha}(N_u)$, we have $D_{N_u}(\phi) \in t^{-1} \Lambda^{1,\alpha}(N_u)$, but since $\xi$ corresponds to this indicial root of $D_{N_u}$ it follows that for our particular choice of $\phi$ we actually have $D_{N_u}(\phi) \in \Lambda^{1,\alpha}(N_u)$ and so, in particular, $\phi \in t^{-1} L^2_1$. 

We can now apply equation~\eqref{integrate-by-parts} to obtain
\[
\int_\Sigma (D_{N_u}\phi, \psi) \dvol_h = \int_{\del \Sigma} (\xi, \psi_0) \diff s
\]
Since we assume that $\psi$ is orthogonal to the image of $D_{N_u}$ we can conclude that $\psi_0$ is orthogonal to all sections $\xi$ of $E \to K$. It follows that $\psi_0 = 0$. This means that in fact $\psi$ has a higher order of vanishing that the indicial roots alone impose. It follows that $\psi$ vanishes to infinite order at the boundary (Theorem~\ref{0-Fredholm}). Next note that $\psi$ solves the Laplace-type equation $D_{N_u}D_{N_u}^*\psi = 0$. This means that Mazzeo's unique continuation theorem applies to $\psi$, showing that in fact $\psi=0$ (Corollary~11 of \cite{Mazzeo2}). It follows that $D''$ is surjective and hence that $D_{(u,j)}$ is surjective. This proves the theorem when $\chi(\overline{\Sigma}) <0$. 

When $\overline{\Sigma}$ is the disc or an annulus, we need to be a little careful to deal with automorphisms of the domain. Consider the disc first. In this case there is a unique complex structure modulo diffeomorphisms; at the same time the $\delb$-operator on tangent vectors is surjective and so there is no need to involve a Teichmüller slice. Let $j_0$ denote the standard complex structure on the disc $\overline{D}$ and let
\begin{equation}
\hat{\M}^{\infty}_{0,1} = \{ u \colon (\overline{D},j_0) \to \overline{Z}: u\in \A \text{ is } C^{\infty} \text{ and }J\text{-holomorphic}\}
\label{hat-moduli}
\end{equation}
We then let $\hat{\M}_{0,1}$ be the closure of $\hat{\M}^{\infty}_{0,1}$ in the space of admissible maps. The group $\PSL(2,\R)$ of biholomorphisms of $\overline{D}$ acts on $\hat{\M}_{0,1}$ and the quotient is the moduli space we're really interested in:
\[
\M_{0,1} = \frac{\hat{\M}_{0,1}}{\PSL(2,\R)}
\] 
The above proof, in the case $\chi(\overline{\Sigma})<0$, goes through verbatim to show that $\hat{\M}_{0,1}$ is a smooth Banach manifold. Now the action of $\PSL(2,\R)$ is free, because all maps are embeddings on the boundary. To prove the quotient $\M_{0,1}$ is smooth we must show the action is proper. In other words, given a sequence $(u_i) \subset \hat{\M}_{0,1}$ that converges and a second sequence $(g_i) \subset \PSL(2,\R)$ for which $(u_i \circ g_i)$ converges, we must extract a convergent subsequence of the $(g_i)$. To show this, write $u_i \to u$ and $u_i \circ g_i \to v$. We consider the sequence $(g_i(0)) \subset \overline{D}$ where $0 \in D$ is the origin. By compactness of $\overline{D}$ a subsequence of $g_i(0)$ converges to a point $p \in \overline{D}$. If we can show that $p \in D$ lies in the interior then it follows that a further subsequence of the $g_i$ converges in $\PSL(2,\R)$. Now convergence of admissible maps implies $C^0$ convergence and so, along this same subsequence, $u_i(g_i(0)) \to u(p)$. At the same time, $u_i(g_i(0))\to v(0)$ and so $u(p) = v(0)$. Since $v$ is admissible, $v(0)$ is in the interior $Z$ of the twistor space. Since $u(p)$ is in the interior, $p$ cannot lie on the boundary of the disc and so must lie in the interior, $p \in D$ as required. 

Finally we deal with the case $\chi(\overline{\Sigma}) = 0$, i.e.\ $\overline{\Sigma}$ is an annulus. Here a given complex structure $j$ has both a 1-dimensional moduli space \emph{and} a 1-dimensional space of biholomoprhisms. We fix a path $j_t$ of complex structures on $\overline{\Sigma}$, with $t \in (0,1)$ for which $t \mapsto [j_t]$ gives a diffeomorphism to the moduli space of complex annuli. We choose the $j_t$ so that for all $t$ the group of $j_t$-biholomorphisms is the same circle subgroup $S^1 \hookrightarrow \Diff(\overline{\Sigma})$.  We moreover require that the $j_t$ all have the same 1-jet at the boundary (so that they all determine the same space of admissible maps). Next let 
\[
\hat{\M}^{\infty}_{0,2}
=
\{ (u,t) : u \in \A \text{ is } C^{\infty} \text{ and }(j_t,J)\text{-holomorphic} \}
\]
and let $\hat{\M}_{0,2}$ be the closure of $\hat{\M}^\infty_{0,2}$. The distinguished circle $S^1 \subset \Diff (\overline{\Sigma})$ acts on $\hat{\M}_{0,2}$ and the moduli space we're interested in is the quotient:
\[
\M_{0,2} = \frac{\hat{M}_{0,2}}{S^1}
\]
The above proof goes through verbatim to show that $\hat{\M}_{0,2}$ is a smooth Banach manifold. The $\delb$ operator on $T\overline{\Sigma}$ has a 1-dimensional cokernel, which is accounted for by varying $t$, which corresponds to the Teichmüller slice. Finally, the circle action on $\hat{\M}_{0,2}$ is free (since the maps are embeddings on the boundary of $\overline{\Sigma}$ and since the circle is compact the quotient is again a manifold. 
\end{proof}

\subsection{Transversality of evaluation maps}\label{ev-is-transverse-section}

The moduli space of admissible $J$-holomorphic curves with $n$ marked points is defined as
\[
\M_{g,c}^n = \frac{\{ (u,j,p_1, \ldots, p_n) \}}{\sim}
\]
where $(u,j)$ corresponds to a point in the moduli space $\M_{g,c}$ of admissible $J$-holo\-morphic curves, the points $p_1, \ldots, p_n \in \Sigma$  are distinct points of the interior and $(u,j,p_i) \sim (u',j',p_i')$ if they are related by a diffeomorphism of $\Sigma$ in the obvious way. 

Just as $\M_{g,c}$ is a smooth Banach manifold, so is $\M_{g,c}^n$. In fact, this follows directly from what we have done in~\S\ref{smooth-moduli-space-section}. Assume first that $\chi(\overline{\Sigma})<0$. In this case, we showed that $\X_{g,c}$ was a smooth Banach manifold, with a free and proper action of the group $\G$ of admissible diffeomorphisms of $\overline{\Sigma}$. So $\X_{g,c} \to \M_{g,c}$ is a principal bundle. Now let $U \subset \Sigma^n$ denote the open set where components are pairwise distinct. $\G$ acts on $U$ and the associated bundle is the moduli space of pointed maps:
\[
\M_{g,c}^n = \X_{g,c} \times_{\G} U
\]
In particular it is the quotient of a smooth Banach manifold by a free and proper action, so is itself a smooth Banach manifold. 

When $\overline{\Sigma}$ is a disc or an annulus, we can again use the associated bundle construction. This time, we use $\PSL(2,\R)$ or $S^1$ in place of $\G$ and the principal bundle $\PSL(2,\R) \hookrightarrow \hat{\M}_{0,1} \to \M_{0,1}$ or $S^1 \hookrightarrow \hat{\M}_{0,2} \to \M_{0,2}$ in the role of $\X_{g,c} \to \M_{g,c}$.  

Evaluation of $u$ at the points $p_i$ defines a map $\ev \colon \M_{g,c}^n \to Z^n$: 
\[
\ev[u,j,p_i] = (u(p_1),\ldots, u(p_n)).
\]

\begin{theorem}\label{ev-is-transverse}
The evalutaion map $\ev  \colon \M_{g,c}^n \to Z^n$ is a submersion.
\end{theorem}
\begin{proof}
We give the proof first in the case $n=1$, which contains the main ideas but for which the notation is easier. 

It suffices to prove that $\ev$ is a submersion at the points of $\M_{g,c}^1$ for which $u$ is $C^{\infty}$ up to the boundary (since they are dense). Let $[u,j,p] \in \M_{g,c}^1$ be such a point and let $\xi \in T_{u(p)}Z$. We will show that $\xi \in \im \diff (\mathrm{ev})$. There are 3 cases to consider.
\vspace{0.4\baselineskip}

\noindent{\bf Case 1.} 
The point $p$ is not a branch point of $u$ and $\xi \in \im \diff u_p$. In this case $\xi$ is in the image of $\diff(\mathrm{ev})$ simply by moving the point $p$ on $\Sigma$ in the direction corresponding to $\xi$.
\vspace{0.4\baselineskip}

\noindent{\bf Case 2.} 
The point $p$ is a branch point of $u$ and $\xi$ lies in the image of $T_p\Sigma \otimes L_B \hookrightarrow T_{u(p)}Z$ under $\diff u$. We take a section $\zeta$ of $T\Sigma\otimes L_B$ with the property that $\diff u_p(\zeta \otimes \chi_B^{-1})= \xi$.  We choose $\zeta$ so that it is holomorphic near $p$ and compactly supported away from the other branch points and the boundary. We will now adjust $\zeta$ and move $j$ to give a holomorphic reparametrisation of $u$, but without affecting the value of $\diff u(\zeta \otimes \chi^{-1}_B)$ at $p$. 

By construction, $\delb_{T\overline{\Sigma}\otimes L_B}( \zeta)$ is a section of $T^*\overline{\Sigma}^{0,1} \otimes T\overline\Sigma\otimes L_B$ which is zero on a neighbourhood of $B$. Set $\eta = \delb_{T\overline{\Sigma}\otimes L_B}( \zeta) \otimes \chi_B^{-1}$. This section of $T^*\overline{\Sigma}^{0,1} \otimes T\overline{\Sigma}$ is only defined a~priori away from $B$ but since $\zeta$ is holomorphic on a neighbourhood of $B$, $\eta$ extends to the whole of $\Sigma$. By Theorem~\ref{tangential-surjective} we can find an admissible vector field $v$ such that $\delb_{T\overline{\Sigma}}(v)+\eta =s$ lies in a Teichmüller slice $\mathscr{S}$. It follows that 
\[
D_{u,j}\left(\diff u\left(v  + \zeta\otimes \chi_B^{-1}\right),-s\right) = 0
\] 
and so $(\diff u\left(v  + \zeta\otimes \chi_B^{-1}\right),-s)$ corresponds to a tangent vector to $\M_{g,c}$ at $[u,j]$. Notice that 
\[
\diff u_p\left(v  + \zeta\otimes \chi_B^{-1}\right) = \diff u_p(\zeta \otimes \chi_B^{-1}) = \xi
\]
since $\diff u_p(v)=0$  because $p$ is a branch point.  
\vspace{0.4\baselineskip}

\noindent{\bf Case 3.} It remains to treat the case when $\xi \in (N_u)_p$ is normal to $u(\Sigma)$. Again we extend $\xi$ to a section $\zeta$ of $N_u$ which is compactly supported away from the boundary. Write $\eta = D_{N_u}(\zeta)$. We will find a section $v$ of $N_u$ with $v(p) = 0$ and $D_{N_u}(v)=-\eta$. Assume for a moment that we have such a $v$. Then $D_{N_u}(\zeta +v)=0$ and so $\zeta + v$ satisfies the normal projection of the linearised Cauchy--Riemann equations. By Theorem~\ref{tangential-surjective}, there is a tangential vector-field $\sigma \in \V_B$ and adjustment $s \in \mathscr{S}$ of $j$, such that the whole triple $(\sigma + \zeta + v,s)$ is tangent to $[u,j] \in \M_{g,c}$. Moreover, the normal component of $(\sigma +\zeta + v)(p)$ is precisely $\xi$. The tangential component need not be zero, but this doesn't matter since we correct for this via the method of one of the two previous cases (depending on whether $p$ is a branch point or not).

The proof of the existence of the section $v$ of $N_u$ that we seek is essentially identical to that of the surjectivity of $D_{N_u}$ in the proof of Theorem~\ref{smooth-moduli-space}. Write $\mathscr{N}_p \leq T_u\A$ for the subspace of admissible perturbations of $u$ which vanish at $p$ and lie in $N_u$. This contains the subspace of sections of $N_u$ in $t^{1+\alpha}\Lambda^{2,\alpha}$ which also vanish at $p$. Since this is of finite codimension in the space of all sections $t^{1+\alpha}\Lambda^{2,\alpha}(N_u)$ and since $D_{N_u}$ is Fredholm on $t^{1+\alpha} \Lambda^{2,\alpha}(N_u)$ it follows that
\[
D_{N_u} \colon \mathscr{N}_p \to t^{1+\alpha} \Lambda^{1,\alpha}(\z{T^*\Sigma}^{0,1} \otimes N_u) 
\] 
has closed image with finite dimensional cokernel. Suppose now that $\psi$ is $L^2$-orthogonal to the image. Then $\psi$ is smooth on $\Sigma \setminus p$ and  $D_{N_u}^*(\psi)=0$ there. Now the same integration by parts argument, via equation~\ref{integrate-by-parts} shows that in fact $\psi$ vanishes to infinite order at the boundary and once again Mazzeo's Theorem on unique continuation at infinity forces $\psi=0$ on $\Sigma\setminus p$ and hence all of $\Sigma$ by continuity. This means that $D_{N_u}$ is surjective and so we can find the required solution $v$ to $D_{N_u}(v) =-\eta$ with $v(p)=0$.
%\vspace{0.4\baselineskip}

We now do the same thing with multiple marked points. Given $[u,j,p_i] \in \M_{g,c}^n$ with $u$ a $C^\infty$ map, we pick $\xi = (\xi_1, \ldots, \xi_n) \in T_{u(p_1)}Z \oplus \cdots \oplus T_{u(p_n)}Z$. Write $\xi_i = \xi_i^T + \xi_i^N$ for the tangential-normal decomposition of each $\xi_i$. We first use the method of Case~3 above to eliminate the $\xi^N_i$. We take a section $\eta$ of $N_u$ which equals $\xi^N_i$ at $p_i$ and is compactly supported away from the boundary. Then we use the same argument as before to find a section $v$ of $N_u$ with $v(p_i)=0$ for all $i$, such that $D_{N_u}(\zeta+ v)=0$. (The space of sections of $N_u$ in $t^{1+\alpha}\Lambda^{2,\alpha}$ which vanish at all of the $p_i$ has finite codimension. From here the argument proceeds verbatim.) 

Now Theorem~\ref{tangential-surjective} gives $\sigma \in \V_B$ and $s \in \mathcal{S}$ such that the triple $(\sigma + \zeta + v, s)$ is tangent to $[u,j] \in \M_{g,c}$ and $\sigma + \zeta + v$ has the right normal component at each point $p_i$. It remains to deal with the tangential parts which is done just as before, via Case~1 for each $p_i$ which is not a branch point, and a single application of Case~2 to deal simultaneously with the $p_i$ which are branch points.
\end{proof}

\begin{remark}
When $n=1$, i.e.\ a single marked point, there is a very simple proof that the evaluation map is a submersion, exploiting the symmetries of $Z$. The isometry group $\SO(4,1)$ of $\H^4$ acts transitively on $Z$ preserving $J$. This means that it also acts on $\M^1_{g,c}$, by composition. The infinitesimal action at a point $z \in Z$ is a surjection $ \so(4,1) \to T_z Z$ and so the infinitesimal action of $\SO(4,1)$ on $\M^1_{g,c}$ can be used to hit any $\xi \in T_{u(p)}Z$.
\end{remark}

\subsection{Maps with branch points}\label{branch-points-codim2-section}

In this section we show that the $J$-holomorphic curves with branch points make up a codimension~4 subset of $\M_{g,c}$ and that those whose projections to $\H^4$ have branch points are a codimension~2 subset. See Theorem~\ref{branch-points-codim2} below for a precise statement.

The $J$-holomorphic curve $(u,j)$ has a branch point at $p \in \Sigma$ if $ \diff u_p = 0$. This differential lives in the 3-dimensional complex vector space $T_p^*\Sigma^{1,0}_j \otimes T_{u(p)}Z$. These spaces fit together to give a rank-3 complex vector bundle $\E \to \M^1_{g,c}$ over the moduli space of pointed curves. The differential $\diff u_p$ determines a section of $\E$ and the branch points are precisely zeros of this section:
\[
\B_{g,c} = \{ [u,j,p] \in \M^{1}_{g,c} : \diff u_p = 0\}
\]

Similarly, we can consider those $J$-holomorphic curves in $Z$ whose projections to $\H^4$ have branch points. The differential $\diff (\pi \circ u)_p$ lies in $T_p^*\Sigma^{1,0}_j \otimes H_{u(p)}$ (where $H$ denotes the horizontal distribution). These spaces fit together to give a rank-2 complex vector bundle $\F \to \M^1_{g,c}$. The differential $\diff (\pi \circ u)_p$ determines a section of $\F$ and the projected branch points are the zeros of this section:
\[
\B^{\pi}_{g,c} = \{ [u,j,p] \in \M^1_{g,c} : \diff (\pi \circ u)_p =0 \}
\]

\begin{theorem}\label{branch-points-codim2}~
\begin{enumerate}
\item
$\B_{g,c}$ is a codimension~6 submanifold of $\M^1_{g,c}$. The natural projection $\B_{g,c} \to \M_{g,c}$, whose image is the set of all $J$-holomorphic curves with branch points, is Fredholm with index $-4$.
\item
$\B^{\pi}_{g,c}$ is a codimension~4 submanifold of $\M^1_{g,c}$. The natural projection $\B_{g,c} \to \M_{g,c}$, whose image is the set of all $J$-holomorphic curves which project to minimal surfaces with branch points, is Fredholm with index $-2$.
\end{enumerate}
\end{theorem}

\begin{proof}
We prove the statement for $\B_{g,c}$; the proof for $\B^\pi_{g,c}$ is similar. The vector bundle $\E \to \M^1_{g,c}$ has complex rank~3. We will show that the section of $\E$ determined by $\diff u_p$ is transverse to the zero section; it then follows that $\B_{g,c}$ is a submanifold of codimension~6. 

As we now explain, the proof of transversality is very similar to the proof of Theorem~\ref{ev-is-transverse} above, that the evaluation map is a submersion. Let $u \colon (\overline{\Sigma},j)\to \overline{Z}$ be a $J$-holomorphic curve. In twistor half-space coordinates  $x,y_i,z_j$ on $Z$, for any vector $t$ tangent to $\overline{\Sigma}$, 
\begin{equation}
\diff u (t) = (t \cdot x) \del_x + (t \cdot y_i) \del _{y_i} + (t \cdot z_j) \del_{z_j}
\label{diff-u-t}
\end{equation}
Now consider a section $w = a\del_x + b_i \del_{y_i} + c_j \del_{z_j}$ of $u^*TZ$ where $a,b_i,c_j$ are functions on $\Sigma$. Moving $u$ infinitesimally in the direction $w$ we see from~\eqref{diff-u-t} that the corresponding infinitesimal change in $\diff u(t)$ is 
\[
(t\cdot a)\del_x + (t \cdot b_i)\del_{y_i}+ (t \cdot c_j)\del_{z_j}
\]
(where we treat $x,y_i,z_j$ as functions on $\overline{\Sigma}$ by pulling back with $u$). So proving that the section of $\E$ vanishes transversely amounts to showing that for any $\xi = \alpha \del_x + \beta_i \del_{y_i} + \gamma_j \del_{z_j}$ we can find $w = a \del_x + b_i \del_{y_i} + c_j \del_{z_j}$ and an adjustment $s \in \mathcal{S}$ of $j$ such that $(w,s)$ is tangent to $\M_{g,c}$ at $[u,j]$ with $(t \cdot a)(p) = \alpha$,  $(t\cdot b_i) = \beta_i$ and $(t\cdot c_j) = \gamma_j$. In other words we need to have enough freedom in choosing $(w,s)$ so as to prescribe $w$ to first order at $p$. This should be compared with the proof that $\ev \colon \M^1_{g,c} \to Z$ is a submersion, which required us to have enough freedom to prescribe $w$ to zeroth order at $p$. 

The proof is now almost identical: we write down a compactly supported $w$ with the correct $1$-jet at $p$ and then use the Fredholm theory developed previously to correct it so that it yields a tangent vector to $\M^1_{g,c}$, whilst not affecting the 1-jet at $p$. We omit the details, which are the same as those in the proof of Theorem~\ref{ev-is-transverse}.

Finally, the submersion $\M^{1}_{g,c} \to \M_{g,c}$ has 2-dimensional fibres. The submanifold $\B_{g,c}$ meets the fibre over $[u,j]$ in the branch points of $u$, a finite set. It follows that the restriction of the forgetful map to $\B_{g,c} \to \M_{g,c}$ is Fredholm of index $2-6=-4$ as claimed.
\end{proof}

\subsection{The index of the boundary map}\label{boundary-map-Fredholm-section}

Let $\L_c$ denote the Banach manifold of all oriented $C^{2,\alpha}$ links in $S^3$ with $c$ components. %When $L \in \L_c$ is actually smooth, with normal bundle $\nu$, then $T_L\L_c = C^{2,\alpha}(L,\nu)$. When $L$ is merely $C^{2,\alpha}$, we take a smooth link $L'$ which is $C^{2,\alpha}$-close to $L$; then $L$ itself is identified with a $C^{2,\alpha}$ section $s_0$ of the normal bundle $\nu' \to L'$ and, by adding $s_0$, we get a natural identification $T_L\L_c = C^{2,\alpha}(L',\nu')$. 
Given an admissible map $u \in \A$, the boundary curve $\pi \circ u \colon \del \Sigma \to S^3$ has image which is a oriented $C^{2,\alpha}$ link. We define a map $\beta \colon \M_{g,c} \to \L_c$ which sends $[u,j]$ to the oriented link $\beta[u,j] \defeq \pi( u (\del \Sigma) )$. The goal of this section is to prove the following.

\begin{theorem}\label{boundary-map-Fredholm}
The boundary value map $\beta \colon \M_{g,c} \to \L_c$ is Fredholm of index zero. \end{theorem}
One could anticipate this via the Eells--Salamon correspondence: $J$-holomorphic curves  in $Z$ correspond to minimal surfaces in $\H^4$. At least when a minimal surface is embedded, deformations correspond to sections of the normal bundle which solve a Laplace-type equation and so one expects the index to be zero (given appropriate boundary conditions, Dirichlet in this case). This is exactly what is proved by Alexakis--Mazzeo for embedded minimal surfaces in $\H^3$ \cite{Alexakis-Mazzeo}. %When a minimal surface has branch points however, we must deform the parametrising map, rather than go directly to the image and so we have to work in the context of $J$-holomorphic curves. 

We start with the following Proposition, which relates the index of $\diff \beta$ to that of  the normal part of the Cauchy--Riemann operator:
\[
D_{N_u}  \colon t^{1+\alpha}\Lambda^{2,\alpha}(N_u) \to t^{1+\alpha}\Lambda^{1,\alpha}(\z{T^*\Sigma}^{0,1}\otimes N_u)
\]
\begin{proposition}\label{index-dbeta}~
\begin{itemize}
\item
The map $\diff \beta \colon T_{[u,j]}\M_{g,c} \to T_{\beta[u,j]}\L_c$ is Fredholm. 
\item
Write the branch locus of $u$ as $B = \sum k_i p_i$ and let $k = \sum k_i$ be the total degree of $B$. Then $ \ind(\diff \beta) = \ind(D_{N_u}) + 2k$.
\end{itemize}
\end{proposition}
\begin{proof}
We give the proof first in the case $\chi(\overline{\Sigma})<0$. Write a tangent vector to $\M_{g,c}$ at $[u,j]$ as triple: $(\diff u(\xi) + \eta, s)$ where $\xi \in \V_B$ is an admissible section of $T\overline{\Sigma}^{1,0}\otimes L_B$ (see Definition~\ref{admissible-vector-fields}), $\eta$ is a section of $N_u$ and $s \in \mathscr{S}$ is an element of an admissible Teichmüller slice. The fact that $(\diff u(\xi)+\eta, s)$ preserves the $J$-holomorphic equation means that
\begin{align*}
\delb_{B}(\xi) + s\otimes \chi_B + A(\eta) & =0\\
D_{N_u}(\eta)	&= 0
\end{align*}
Here $A(\eta)$ is, by definition, the component of $D_{(u,j)}(\eta)$ which lies in the summand $T^*\overline{\Sigma}^{0,1} \otimes T\overline{\Sigma}^{1,0}\otimes L_B$. Given a solution to this system, $\diff \beta(\diff u(\xi)) + \eta, s)$ is the component of $\pi_*(\eta|_{\del \Sigma})$ which is normal to the boundary link $\beta(u)$ of the initial map~$u$. 

Suppose that $\diff \beta (\diff u(\xi) + \eta, s)=0$. This tells us that $\pi_*\eta$ has no component normal to $\beta(u)$. This component of the boundary value of $\eta$ is precisely  that which corresponds to the indicial root $-1$. The only other indicial root is $2$. This, together with the fact that $D_{N_u}(\eta)=0$, means that $\eta$ must vanish on the boundary. Lemma~\ref{normal-and-zero-holder} then gives $\eta \in  t^{\alpha}\Lambda^{1,\alpha}(N_u)$. We have already shown that $D_{N_u}$ is Fredholm at this weight (Theorem~\ref{Fredholm-normal-directions}) and so $ \eta \in \ker D_{N_u}$ lies in a finite dimensional space. Meanwhile, by Theorem~\ref{tangential-surjective}, the map 
\begin{equation}
\begin{gathered}
\delta_B \colon \V_B \oplus \mathscr{S} \to C^{1,\alpha}_1(T^*\overline{\Sigma}^{0,1} \otimes T\overline{\Sigma}^{1,0} \otimes L_B)\\
\delta_B \colon (\xi,s)\mapsto \delb_{B}(\xi)+s\otimes \chi_B
\end{gathered}
\label{surjective-tangent-plus-slice}
\end{equation}
is surjective and $\dim \ker(\delta_B) = 2k$. So for any $\eta \in \ker (D_{N_u})$ there is a $\dim \ker \delta_B$-dimensional space of solutions to 
$\delta_B(\xi,s) + A(\eta)=0$. It follows that $\dim \ker (\diff \beta) = \dim \ker (D_{N_u})+ 2k$.

To complete the proof we show that $\dim\coker \diff \beta = \dim \coker D_{N_u}$. To do this, let $v \in T_{\beta(u)}\L_c$. We write $K=\beta(u)$ for the boundary link in $S^3$ and interpret $v$ as a section of the normal bundle  $E \to K$. In a collar neighbourhood we can extend $v$ to a section $\eta_v$ of $N_u \to \Sigma$, which is of weight $t^{-1}$ (because of the scaling of the metric on $Z$ at the boundary, see Remark~\ref{not-really-blowing-up}).  Since the boundary value $v$ lies in the bundle corresponding to the indicial root $-1$ and since the only other indicial root is $2$, we can choose this extension to have $D_{N_u}(\eta_v) \in t^{1+\alpha}\Lambda^{1,\alpha}(\z{T^*\Sigma}^{0,1}\otimes N_u)$. 

Suppose momentarily that $D_{N_u}$ is surjective between spaces of weight $t^{1+\alpha}$. We will show that $v \in  \im \diff \beta$. By our assumption, we can find $\psi \in t^{1+\alpha}\Lambda^{2,\alpha}(N_u)$ such that $D_{N_u}( \psi) = - D_{N_u}(\eta_v)$ and so $\eta = \eta_v + \psi$ is in $\ker D_{N_u}$. Moreover, since $\psi$ is weight $t^{1+\alpha}$, the boundary value of $\eta$ is equal to $v$ as well. We then use the surjectivity of~\eqref{surjective-tangent-plus-slice} to find $(\xi,s)$ so that $(\xi + \eta, s)$ solves the linearised Cauchy--Riemann equations. With this done, we have a tangent vector to $\M_{g,c}$ on which $\diff \beta(\xi + \eta, s) = v$. 

In general, $D_{N_u}$ need not be surjective. In this case, we have a map
\begin{equation}\label{coker-dbeta-is-coker-DN}
\Phi \colon T_{\beta(u) }\L_c \to \coker D_{N_u}
\end{equation}
given by sending $v$ to the class $[D_{N_u}(\eta_v)]$ where $\eta_v$ is any extension of $v$ to a section of $N_u$ of weight $t^{-1}$, with $D_{N_u}(\eta_v)$ of weight $t^{1+\alpha}$. Note that $\eta_v$ has weight $t^{-1}$ whereas the cokernel is defined via images of sections with weight $t^{1+\alpha}$ and so the class $[D_{N_u}(\eta_v)]$ is not automatically zero. Moreover, any other extension $\eta'_v$ differs by an element of $t^{1+\alpha}\Lambda^{2,\alpha}$ and so defines the same class in the cokernel, meaning $\Phi$ is well defined.  The argument in the previous paragraph shows that $\ker \Phi \subset \im \diff \beta$. Equally if $v \in \im \diff \beta$ then there is a tangent vector $(\xi+\eta, s)$ to $\M_{g,c}$ where $\eta$ has boundary value $v$. The fact that $D_{N_u}(\eta) = 0$ means that we can take $\eta = \eta_v$ and this shows that $v \in \ker \Phi$. The upshot is $\im \diff \beta = \ker \Phi$ and  $\dim \coker \diff \beta \leq \dim \coker D_{N_u}$. It follows that $\diff \beta$ is Fredholm. 

To complete the proof of the formula for the index we must show that $\Phi$ is surjective, and so gives an isomorphism $\coker \diff \beta \cong \coker D_{N_u}$. This is done in an identical fashion to the proof of Theorem~\ref{smooth-moduli-space}. The cokernel $\coker D_{N_u}$ is spanned by the classes of elements of $\ker D_{N_u}^*$. If $\Phi$ were not surjective, there would be an element $\psi \in \ker D_{N_u}^*$ which is orthogonal to all sections of the form $D_{N_u}(\eta_v)$ as $v$ varies over $T_{\beta(u)}\L_c$. Now an identical argument to the proof of Theorem~\ref{smooth-moduli-space} shows that 
\[
0 = \int_\Sigma (D_{N_u}(\eta_v), \psi) = \int_{\del \Sigma}(v, \psi_0)\diff s
\]
and so $\psi_0=0$ which in turn forces $\psi = 0$.

The cases of $\chi(\overline{\Sigma}) = 0, 1$ are almost identical. The only difference is that $\dim \ker (\delta_B) = 2k + 1$ when $\chi(\overline{\Sigma})=0$ and $\dim \ker (\delta_B) = 2k + 3$ when $\chi(\overline{\Sigma})=1$. But in these cases we must also quotient by the group of biholomorphisms which exactly cancels these extra terms. For example, in the notation of~\eqref{hat-moduli} in the proof Theorem~\ref{smooth-moduli-space}, the above argument proves directly that as a map from $\hat{\M}_{0,1}$ the boundary map has index~$\ind(D_{N_u}) + 2k + 3$. When passing to the genuine moduli space $\M_{0,1}$ by dividing by the $\PSL(2,\R)$ action we recover $\ind(\diff \beta) = \ind(D_{N_u}) + 2k$ as claimed. 
\end{proof}

It remains to compute the index of $D_{N_u}$. Unfortunately there is no ``off the shelf'' index theorem we can appeal to in this setting. Instead we will compute the index by a trick which exploits the fact that $Z$ is an ``almost Calabi--Yau threefold''. 

We first explain how the trick works for a compact Kähler Ricci-flat Calabi--Yau threefold $X$. The fact that the metric is Kähler and Ricci-flat means that there is a unit-length holomorphic volume form $\theta$. Let $C \subset X$ be an embedded compact $J$-holomorphic curve with normal bundle $N \to C$.  We will show directly (without appeal to the index theorem) that the Cauchy--Riemann operator $D_N$ of $C$ in the normal directions has $\ind(D_N)=0$. At points of $C$, we have $\overline{\theta} \in \Lambda^{0,3}T^*X = T^*C^{0,1} \otimes \Lambda^{0,2}N^*$. Using the metric, we interpret $\overline{\theta}$ as an anti-linear bundle isomorphism $N \to T^*C^{0,1} \otimes N$,  that we denote $\Theta$. Now one can check that $D_{N}$ and its adjoint $D^*_N$ are related by $D_N = \Theta \circ D^*_N \circ \Theta$. Since composing with isomorphisms doesn't affect the index, $\ind(D_N) = \ind (\Theta \circ D^*_N \circ \Theta) = \ind(D_N^*)$ from which it follows that $\ind(D_N)=0$. 

There are three things which are different in our situation: firstly, whilst we can find a unit-length complex volume form, it is not holomorphic. This means that the equation $D_{N} = \Theta \circ D^*_N \circ \Theta$ no longer holds. Secondly, we are not working on a closed Riemann surface, but on one with boundary, along which the operators involved are $0$-elliptic. Thirdly, our curve may have branch points. To ease the explanation we deal first with the case when $u$ has no branch points.

\begin{proposition}\label{index-DNu-no-branch-points}
If $u$ has no branch points then $\ind(D_{N_u}) = 0$.
\end{proposition}

\begin{proof}
We begin by describing the complex volume form on $Z$. Recall the horizontal and vertical 1-forms $\alpha_i, \beta_i$ on $Z$ from \eqref{vertical-1-forms}. We write
\begin{align*}
\omega_1 &= \alpha_0\wedge \alpha_1 + \alpha_2 \wedge \alpha_3\\
\omega_2 &= \alpha_0\wedge \alpha_2 + \alpha_3 \wedge \alpha_1\\
\omega_3 &= \alpha_0\wedge \alpha_3 + \alpha_1 \wedge \alpha_2
\end{align*}
and then define a pair of 3-forms on $Z$ by
\begin{align*}
\theta_\mathrm{Re} &= \beta_1 \wedge \omega_1 + \beta_2 \wedge \omega_2 + \beta_3 \wedge \omega_3\\
\theta_\mathrm{Im} 
	& = 
		z_1(\beta_2 \wedge\omega_3 - \beta_3\wedge \omega_2)
		+
		z_2(\beta_3 \wedge\omega_1 - \beta_1 \wedge \omega_3)
		+
		z_3 (\beta_1 \wedge \omega_2 - \beta_2 \wedge \omega_1)
\end{align*}
One can check that $\theta = \theta_\mathrm{Re} + i\theta_\mathrm{Im}$ is  a unit-length complex 3-form on $Z$. It is in this sense that $Z$ is ``almost Calabi--Yau''. (The 3-form $\theta$ is not closed---that would imply $J$ were integrable, which it is not---although  $\diff \theta_{\mathrm{Re}} =0$, as one can check directly.) 

Given a $J$-holomorphic curve $u$, we use the $(0,3)$-form $\overline{\theta} = \theta_{\mathrm{Re}} - i \theta_{\mathrm{Im}}$ to define an anti-linear bundle isomorphism $\Theta \colon N_u \to \z{T^*\Sigma}^{0,1}\otimes {N}_u$. Explicitly, given sections $\xi$ of $\z{T\Sigma}$ and $\nu$ of $N_u$, we set $\Theta(\xi)(\nu)$ to be the section of $N_u$ determined by
\begin{equation}
g(\Theta(\xi)(\nu),\mu) = \overline{\theta}(\diff u(\xi),\nu,\mu)
\label{Theta-from-theta}
\end{equation}
for all $\mu \in N_u$. Here $g$ is the Hermitian metric on $N_u \subset u^*\e{TZ}$ induced from the edge metric on $\overline{Z}$. Our convention is that $g$ is complex-linear in the left-hand argument. Since $\overline{\theta}$ is anti-linear in all arguments and $\diff u$ is complex-linear, it follows that $\Theta(\xi)(\nu)$ is anti-linear in both $\xi$ and $\nu$. Moreover, because $u$ has no branch-points $\Theta$ is an isomorphism of complex vector bundles. 

$\Theta$ has unit length when measured with respect to metrics induced entirely from the edge metric on $\overline{Z}$ (so using the $0$-metric on $\overline{\Sigma}$ given by restricting the edge metric via $\diff_eu \colon \z T \Sigma \hookrightarrow u^*\e{TZ}$). Moreover, derivatives of $\Theta$ of all order are uniformly bounded, $|\nabla^r \Theta| \leq C$, when measured using the metric and Levi-Civita connection of the edge metric on $\overline{Z}$. This follows from the fact that the original complex 3-form $\theta$ is invariant under the transitive isometric action of $\Isom(\H^4)$ on $Z$. It follows that $\Theta$ gives an isomorphism between the weighted spaces $t^{\delta}\Lambda^{k,\alpha}(N_u) \to t ^\delta\Lambda^{k,\alpha}(\z{T^*\Sigma}^{0,1}\otimes N_u)$. (An aside: when we carried out the Fredholm analysis, we used a fixed $0$-metric on $\overline{\Sigma}$, not the metric induced from $\overline{Z}$, which depends on $u$. However,  $u$ is asymptotically an isometry for this induced metric, and so $\|\Theta\|_{C^r}$ is bounded when measured with respect to this fixed choice of  background metric too.) 

We now consider the following operator
\[
\begin{gathered}
\hat{D} \colon t^{1+\alpha}\Lambda^{2,\alpha}(N_u) \to t^{1+\alpha}\Lambda^{1,\alpha}(\z{T^*\Sigma}^{0,1} \otimes N_u)\\
\hat{D} = \Theta \circ D^*_{N_u} \circ \Theta
\end{gathered}
\]
Since we have composed a Fredholm operator with isomorphisms, $\hat{D}$ is also Fredholm, with $\ind (\hat{D})   = \ind (D^*_{N_u}) = - \ind(D_{N_u})$.

To prove that $\ind (D_{N_u})=0$ we will show also that $\ind(\hat{D}) = \ind (D_{N_u})$. We claim that:
\begin{enumerate}
\item\label{normal-operators-agree} 
$D_{N_u}$ and $\hat{D}$ have the same normal operator (in the sense of the $0$-calculus).
\item\label{symbols-agree} 
The $0$-symbols of $D_{N_u}$ and $\hat{D}$ agree.
\end{enumerate}
Momentarily assuming these two facts, we set $D_r = r D_{N_u} + (1-r)\hat{D}$. By \ref{symbols-agree}, $D_r$ is $0$-elliptic for all $r$; now by \ref{normal-operators-agree} and Proposition~\ref{normal-invertible} the normal operator of $D_r$ is invertible and so $D_r$ is Fredholm for all $r$. It follows by invariance of the index that $\ind(\hat{D})= \ind(D_0)=\ind(D_1) = \ind (D_{N_u})$ and so $\ind(D_{N_u})=0$ as claimed. 

We first prove point~\ref{normal-operators-agree}. To see this, we recall the calculation in Lemma~\ref{whole-normal-operator} of the normal operator $\D$ associated to $D_{N_u}$. We fix an orthonormal frame for $N$:
\[
e_2 = t\del_{y_2}-\del_{z_3}, \quad
e_3 = Je_2 = t\del_{y_3} + \del_{z_2},\quad
e_4 = -\del_{z_2}, \quad
e_5 = Je_4 = \del_{z_3}
\]
and then in this frame the normal operator of $D_{N_u}(t\del_t)$ is given by
\[
\D = 
	\begin{pmatrix}
	\tau \del_\tau & -\tau \del_\sigma & 0 & 1\\
	\tau \del_\sigma & \tau \del_\tau & 1 & 0 \\
	0 & 2 & \tau \del_\tau - 1 & -\tau \del_\sigma \\
	2 & 0 & \tau \del_\sigma & \tau \del_\tau - 1
	\end{pmatrix}
\]
It follows that if we use $(t^{-1}\diff t)^{0,1}$ to trivialise $\z{T^*\Sigma}^{0,1}$ near $\del \Sigma$ then the normal operator of $D_{N_u}^*$ is identified with
\[
\D^* =
	\begin{pmatrix}
	1 - \tau \del_\tau & -\tau \del_\sigma & 0 & 2\\
	\tau \del_\sigma & 1 - \tau\del_\tau & 2 & 0\\
	0 & 1 & -\tau \del_\tau & - \tau \del_\sigma \\
	1 & 0 & \tau \del_\sigma & -\tau \del_\tau
	\end{pmatrix}
\]
(We have used here that $(\tau\del_\sigma)^* = - \tau \del \sigma$ whilst $(\tau \del_\tau)^* = 1 - \tau \del_\tau$ because of the dependence of the half-space volume form on $\tau$, namely $\dvol = \tau^{-2}\diff \sigma \wedge \diff \tau$.)

We now need the matrix representing $\Theta$. Note that $\Theta$  corresponds to a section of $\z{T^*\Sigma}^{0,1}\otimes \Lambda^{0,2}N^*$. If we contract with $t \del t$ then $\Theta(t \del_t)$ is a section of $\Lambda^{0,2}N^*$ with $|\Theta(t\del_t)| \to 1$ at the boundary. Now a unit length element of $\Lambda^{0,2}N^*$ is of the form $\frac{1}{2}(\chi_2 - i \chi_3)$ where $\chi_2$ and $\chi_3$ are metric-dual to antilinear endomorphisms $J_2, J_3$ which give a quaternionic structure $J_1=i, J_2, J_3$ on $N$. Then $\frac{1}{2}(\chi_2 - i \chi_3)$ is metric-dual to $\frac{1}{2}(J_2 - J_1J_3) = J_2$. This means that, after making a rotation in the $y_2,y_3$ coordinates (which leaves the form of $\D$ invariant), at a point on the boundary $\Theta(t\del_t)$ is represented by
\[
M =
	\begin{pmatrix}
	0 & 0 & 0& -1\\
	0 & 0 & -1 & 0\\
	0& 1 & 0 & 0\\
	1 & 0 & 0 & 0
	\end{pmatrix}
\]
It is now a simple mater to check that $M \circ \D^* \circ M = \D$. This proves that $D_{N_u}$ and $\hat{D}$ have the same normal operators.

We must also check part~\ref{symbols-agree}, that the symbols agree. At a point $p \in \Sigma$, let $\alpha, \nu_1, \nu_2$ be an orthonormal frame for $
\e{T^*Z}^{0,1}$ with $\alpha$ spanning $\z{T_p^*\Sigma}^{0,1}$ and $\nu_1, \nu_2$ spanning $\Lambda^{0,1}N^*_u$ at $p$. Then, up to a factor (which we can assume to be 1 by rotating our frame), $\Theta$ has the following form: for $\phi = \phi_1 \nu_1 + \phi_2 \nu_2$ we have $\Theta(\phi) = \alpha\otimes (\overline{\phi}_1 \nu_2 - \overline{\phi}_2 \nu_1)$. The symbol of $D_{N_u}^*$ in the direction of $\alpha$ is minus the metric contraction against $\alpha$. We can now compute directly that
\[
\Theta \circ \sigma(D_{N_u}^*) \circ \Theta(\phi) 
= 
\alpha \otimes \phi
=
\sigma(D_{N_u}) (\phi)
\qedhere
\]
\end{proof}

\begin{proposition}\label{index-DNu}
If $u$ has branch locus $B = \sum k_i p_i$ of total degree $ k = \sum k_i$, then $\ind(D_{N_u}) = -2k$. 
\end{proposition}

\begin{proof}
We run the same argument as in the proof of Proposition~\ref{index-DNu-no-branch-points}. There is one important difference: $\Theta$ as defined in \eqref{Theta-from-theta} vanishes at the branch points of $u$. To deal with this, write $L_B \to  \overline{\Sigma}$ for the holomorphic line bundle associated to the branching divisor $B$ and $\chi_B$ for a holomorphic section of $L_B$ whose zero divisor is $B$. We choose $\chi_B$ to be smooth up to the boundary. Fix a Hermitian metric in $L_B$, for which $\chi_B$ and its derivatives are bounded. The metric gives an anti-linear isomorphism $L^*_B \to L_B$. It follows that
\[
\Lambda^{0,3}(u^*\e{TZ}) \cong \z{T^*\Sigma}^{0,1}\otimes L_B \otimes \Lambda^{0,2}N^*_u
\]
In this way, \eqref{Theta-from-theta} determines two anti-linear bundle isomorphisms
\begin{align*}
\Theta_1 &\colon N_u \to \z{T^*\Sigma}^{0,1} \otimes N_u \otimes L_B\\
\Theta_2 & \colon N_u \otimes L_B \to \z{T^*\Sigma}^{0,1} \otimes N_u
\end{align*}
 
The holomorphic structure on $L_B$, together with the Cauchy--Riemann operator $D_{N_u}$ combine to define a Cauchy--Riemann operator $D_{N_u\otimes L_B}$ on $N_u \otimes L_B$. We now consider the following operator 
\begin{equation}
\hat{D} \defeq \Theta_2 \circ D^*_{N_u\otimes L_B} \circ \Theta_1 \colon  t^{1+\alpha} \Lambda^{2,\alpha}(N_u) \to t^{1+\alpha} \Lambda^{1,\alpha}(\z{T^*\Sigma}^{0,1}\otimes N_u)
\label{twisted-DN*}
\end{equation}
There are no branch points near $\del \Sigma$ and so just as before, $\hat{D}$ is a $0$-elliptic operator whose normal operator is $\D$ in an appropriate frame. In particular, $\hat{D}$ is Fredholm. Moreover, the same computation as before shows that it has the same symbol as $D_{N_u}$. It then follows, just as when $B = \emptyset$, that 
\[
\ind(D_{N_u}) = \ind (\hat{D}) = \ind(D^*_{N_u\otimes L_B}) = - \ind (D_{N_u \otimes L_B})
\]

To complete the proof we now show that $\ind(D_{N_u\otimes L_B}) = \ind(D_{N_u}) + 4k$. The method is essentially identical to that used to compute the effect  on the Riemann--Roch formula of twisting by a holomorphic line-bundle;  accordingly we only sketch the argument. 

We begin by ``turning off'' the anti-linear part of $D_{N_u}$ on a neighbourhood of each branch point, so it determines a genuine holomorphic structure on $N_u$ near the branch points. This can be done continuously and away from $\del \Sigma$, and so it does not alter the fact that $D_{N_u}$ and $D_{N_u\otimes L_B}$ are Fredholm nor does it change their indices. Meanwhile, given the divisor $B= \sum k_ip_i$, we write $V_i$ for the space of $(k_i-1)$-order jets at $p_i$ of holomorphic sections of $N_u \otimes L_B$ and $V = \bigoplus V_i$. Since $D_{N_u\otimes L_B}$ is now a genuine holomorphic structure near $p_i$, each $V_i$ is a complex vector space; $\dim_\C V_i = 2k_i$ and hence $\dim_\R V = 4k$. 

Just as in the standard situation (holomorphic bundles over a compact Riemann surface) there is an exact sequence
\begin{equation}
0 
	\to \ker D_{N_u} 
	\stackrel{T_1}{\longrightarrow} \ker D_{N_u\otimes L_B} 
	\stackrel{T_2}{\longrightarrow} V 
	\stackrel{T_3}{\longrightarrow} \coker D_{N_u}
	\stackrel{T_4}{\longrightarrow} \coker D_{N_u\otimes L_B}
	\to
0
\label{exact-sequence-cohomology}
\end{equation}
The map $T_1$ is given by tensoring with the holomorphic section $\chi_B$ of $L_B$ whose zero divisor is $B$. The map $T_2$ sends a section to its jets of appropriate order at the points $p_i$. The map $T_3$ is the obstruction to extending a given jet to a global solution of $D_{N_u\otimes L_B}(\phi)=0$ with weight $t^{1+\alpha}$. Finally the map $T_4$ is induced by the map $t^{1+\alpha}\Lambda^{1,\alpha}(\z{T^*\Sigma}^{0,1}\otimes N_u)\to t^{1+\alpha} \Lambda^{1,\alpha}(\z{T^*\Sigma}^{0,1} \otimes N_u \otimes L_B)$ given by tensoring with~$\chi_B$. One can check that the fact we are working with $0$-elliptic Cauchy--Riemann operators is not a problem, the usual proofs go through verbatim since all that really matters is that the Fredholm theory applies.

Now taking the alternating sum of the dimensions in~\eqref{exact-sequence-cohomology} we get
\[
\ind(D_{N_u}) + 4k = \ind(D_{N_u\otimes L_B})
\]
Since we also have $\ind (D_{N_u\otimes L_B}) = - \ind(D_{N_u})$ it follows that $\ind(D_{N_u}) = -2k$ as claimed.
%DETAILS:
%
%Define $T_3 \colon V \to \coker D_{N_u}$ as follows. Given $v = (v_i) \in V$, we extend $v$ to a section $\psi \in t \Lambda^{1,\alpha}(N_u \otimes L)$ which is holomorphic on a neighbourhood of each $p_i$. Then $\phi = \psi \otimes s^{-1}$ is a meromorphic section of $N_u$ for which $D_{N_u}(\phi)$ actually vanishes near each $p_i$. We define $T_3(v)$ to be the class in $\coker D_{N_u}$ of $D_{N_u}(\phi)$. This class is not necessarily trivial since the poles of $\phi$ mean that it is not in the domain of $D_{N_u}$. Now $T_3(v)=0$ if and only if there is a solution $\phi'$ to $D_{N_u}(\phi') = - D_{N_u}(\phi)$. This happens if and only if there is a sectio $\phi'\otimes s + \psi \in \ker D_{N_u\otimes L}$ whose jet is $v$. So we have extended the exact sequence one step further:
%\[
%0 \to \ker D_{N_u} \stackrel{T_1}{\longrightarrow} \ker D_{N_u\otimes L} \stackrel{T_2}{\longrightarrow} V \stackrel{T_3}{\longrightarrow} \coker D_{N_u}
%\] 
%The last arrow in the exact sequence is $T_4 \colon \coker D_{N_u} \to \coker D_{N_u\otimes L}$. This is induced by the map $t\Lambda^{0,\alpha}(\overline{T^*\Sigma}\otimes N_u)\to t \Lambda^{0,\alpha}(\overline{T^*\Sigma} \otimes N_u \otimes L)$ given by tensoring with~$L$.
%SHOULD EXPLAIN WHY THIS IS EXACT WITH T_3 AND SURJECTIVE!
\end{proof}

Taken together, Propositions~\ref{index-dbeta}, \ref{index-DNu-no-branch-points} and~\ref{index-DNu}  imply Theorem~\ref{boundary-map-Fredholm}. We close this section with the following observation.

\begin{lemma}\label{corank-1-immersed}
If $\dim(\coker \diff \beta )\leq 1$ at $u \in \M_{g,c}$ then $u$ is an immersion. In particular, if the link $K \subset S^3$ is a regular value of $\beta$ then all $J$-holomorphic curves $u \in \M_{g,c}$ with $\beta(u) =K$ are immersed.
\end{lemma}
\begin{proof}
In the course of the proof of Proposition~\ref{index-dbeta} we showed that $\coker \diff \beta \cong \coker D_{N_u}$. Proposition~\ref{index-DNu} shows that $\dim (\coker D_{N_u}) \geq 2k$ where $k$ is the total branching multiplicity of $u$. So if $\dim (\coker \diff \beta) \leq 1$, then $k=0$.
\end{proof}

\begin{remark}
Even when $u$ is an immersed $J$-holomorphic curve, it doesn't necessarily follow that the minimal surface $\pi \circ u$ is immersed: branch points can still occur if the $J$-holomorphic curve is tangent at some point to a vertical twistor fibre.
\end{remark}

\subsection{Orientability of the moduli spaces}

When $K \in \L_c$ is a regular value of $\beta \colon \M_{g,c} \to \L_c$, the preimage $\beta^{-1}(K)$ is a submanifold. In this section we prove that these submanifolds come with a coherent choice of orientation. We use the standard approach, of trivialising the determinant line bundle associated to $\diff \beta$. Recall that $\det (\diff \beta) \to \M_{g,c}$ is the real line bundle given by
\[
\det(\diff \beta)=\det(\ker \diff \beta ) \otimes  \det (\coker\diff \beta)^*
\]
where $\det (V)$ denotes the top exterior power of the vector space $V$. 

\begin{proposition}\label{orientations}
Let $\M_{g,c}^{\mathrm{imm}} \subset \M_{g,c}$ denote the open subset of $J$-holomorphic immersions. The determinant bundle $\det(\diff \beta) \to \M^{\mathrm{imm}}_{g,c}$ has a distinguished choice of trivialisation. 
\end{proposition}
\begin{proof}
We use the notation as in the previous section. In Proposition~\ref{index-dbeta} we showed that $\ker \diff \beta$ is given by triples $(\xi, \eta,s)$ with 
\begin{align*}
\delb(\xi) + s + A(\eta) &= 0\\
D_{N_u}(\eta) & =0
\end{align*}
where $\eta \in t^{1+\alpha}(\Lambda^{2,\alpha}(N_u)$, $\xi \in \V$ is an admissible section of $T\overline{\Sigma}^{1,0}$ and $s \in \mathscr{S}$ lies in an admissible Teichmüller slice.  The map  $\ker \diff \beta \to \ker D_{N_{u}}$ from $(\xi,s,\eta) \mapsto \eta$  makes $\ker \diff \beta$ into an affine bundle over $\ker D_{N_u}$ with fibres modelled on the complex vector space of solutions $(\xi,s)$ to $\delb(\xi) + s=0$. Since complex vector spaces are oriented, it follows that there is a canonical isomorphism $\det(\ker \diff \beta) \cong \det (\ker D_{N_u})$ and so $\det (\diff \beta) \cong \det (D_{N_u})$.

To trivialise $\det(D_{N_u})$ recall the proof of Proposition~\ref{index-DNu-no-branch-points}. There we showed that $rD_{N_u} +(1-r) \Theta \circ D_{N_u}^*\circ \Theta$ was a path of Fredholm operators. Now, since $\Theta$ is an isometry, it follows that $\Theta^* = \Theta^{-1}$ and so 
\[
\tilde{D}_r = r\Theta^* \circ D_{N_u} + (1-r)D^*_{N_u} \circ \Theta
\] 
is also a path of Fredholm operators. (Here we define the adjoints using the induced $0$-metric on $\overline{\Sigma}$ given by pulling back the edge metric on $\overline{Z}$ via the immersion $\overline{\Sigma} \to \overline{Z}$.) Now $\tilde{D}_{1/2}$ is self-adjoint and so there is a natural isomorphism $\ker \tilde{D}_{1/2} \cong \coker \tilde{D}_{1/2}$ giving a canonical trivialisation of $\det (\tilde{D}_{1/2})$. This in turn gives a trivialisation of $\det(\tilde{D}_1)$ and so of $\det (D_{N_u})$. 
\end{proof}

Proposition~\ref{orientations} allows us to attach a sign to each point $[u,j] \in \M_{g,c}$ at which $\diff \beta$ is an isomorphism. 

\begin{definition}\label{sign-of-regular-point}
Let $[u,j] \in \M_{g,c}$. If $\diff \beta$ is an isomorphism at $[u,j]$ then $\ker \diff \beta = 0 = \coker \diff \beta$ and so there is a canonical trivialisation of the line $\det (\diff \beta_{[u,j]})$ at $[u,j]$. By Lemma~\ref{corank-1-immersed}, $[u,j] \in \M^{\mathrm{imm}}_{g,c}$. Over this subset we have a global trivialisation of the whole line bundle $\det (\diff \beta)$, coming from Proposition~\ref{orientations}. We say $[u,j]$ is \emph{positive} if these two trivialisations of $\det(\diff \beta_{[u,j]})$ agree and \emph{negative} otherwise.  
\end{definition}

%When $K \in \L_c$ is a regular value of $\beta$, we do this at every point of $\beta^{-1}(K)$ and so define an orientation on the $0$-manifold $\beta^{-1}(K)$. We would like to use this to define a link invariant of $K$, as the signed count of points in $\beta^{-1}(K)$. Suppose that $K_0,K_1 \in \L_c$ are regular values, joined by a path $K_t$ in $\L_c$ which is transverse to $\beta$. Then $\diff \beta$ has corank at most 1 at all points in $\beta^{-1}(K_t)$. It follows from Lemma~\ref{corank-1-immersed} and Proposition~\ref{orientations} that 
%\[
%X = \{ ([u,j],t) \in \M_{g,c} : \beta(u) = K_t\}
%\]
%is an oriented 1-dimensional cobordism between $\beta^{-1}(K_0)$ and $\beta^{-1}(K_1)$. Of course, for this to be useful in terms of counting $J$-holomorphic curves we also need the preimages $\beta^{-1}(K_i)$ and $X$ to be compact. The remainder of the article explores this problem.

\section{Convergence of minimal surfaces near infinity}\label{compactness-at-infinity}

Our focus now is the properness (or failure thereof) of the boundary map $\beta \colon \M_{g,c} \to \L_c$. Let $u_n \colon (\overline{\Sigma},j_n) \to \overline{Z}$ be a sequence of admissible $J$-holomorphic curves whose boundary links $K_n=\pi(u_n(\del \Sigma))$ converge in $C^{2,\alpha}$. This section is dedicated to showing that, up to diffeomorphisms, the images converge on a definite neighbourhood of the boundary of $\overline{\H}^4$.

\subsection{Statement of main result of \S\ref{compactness-at-infinity}}\label{statement-of-compactness-near-infinity}

To state our precise result on convergence near infinity, we need to introduce some notation. We write $f_n = \pi \circ u_n$ for the corresponding sequence of conformal harmonic maps. We fix half-space coordinates $(x,y_i)$ so that we can treat the boundary links $K_n \subset \R^3$ as lying in Euclidean space. Given $\epsilon>0$ we write 
\begin{equation}
\begin{aligned}
\overline{M}_n &= f_n(\overline{\Sigma}),\\
\overline{\Sigma}_{n,\epsilon} &= (x\circ f_n)^{-1}[0,\epsilon],\\ 
\overline{M}_{n,\epsilon} &= f_n(\overline{\Sigma}_{n,\epsilon})
\end{aligned}\label{n-epsilon-notation}
\end{equation}
Let $K_\infty = \lim K_n$. We will write $K_n$ as the graph over $K_\infty$ of a section $s_n$ of the normal bundle of $K_\infty$. In fact, there is a small subtlety here: $K_\infty$ is in principle only $C^{2,\alpha}$ and so its normal bundle is only $C^{1,\alpha}$, yet we need to talk about the higher regularity of $s_n$. In order to do so, we choose a smooth link $K \subset \R^3$ which is very close to $K_\infty$ in $C^{2,\alpha}$; then for all large $n$, $K_n$ can be written as a graph over $K$ of a section $s_n$ of the normal bundle of $K$. By disregarding a finite number of the $K_n$, we will assume henceforth that every $K_n$ is the graph of a $C^{2,\alpha}$-section $s_n$ of the normal bundle of $K$.

The main result of \S\ref{compactness-at-infinity} is about the extension of this graphical description into the interior. We write $[0,\infty) \times K \subset [0,\infty) \times \R^3 = \overline{\H}^4\setminus\{\infty\}$  for the ``half-space cylinder'' over $K$. We write $E \to [0,\infty) \times K$ for the normal bundle of the cylinder (which is the pull-back of the normal bundle of $K \subset \R^3$). Given a section $s$ of $E$, by ``the graph of $s$'' we mean the subset of $[0,\infty) \times \R^3 = \overline{\H}^4\setminus\{\infty\}$ given by
\[
\{ (\xi,\eta + s(\xi,\eta)): \xi \in [0,\infty),\ \eta \in K\}
\]
With all of this in hand, the main result of \S\ref{compactness-at-infinity} is the following. 

\begin{theorem}\label{convergence-near-infinity}
Let $u_n \colon (\overline{\Sigma},j_n) \to \overline{Z}$ be a sequence of admissible $J$-holomorphic curves whose boundary links $K_n=\pi(u_n(\del \Sigma))$ converge in $C^{2,\alpha}$ to an embedded link $K_\infty$ and let $K$ be a smooth link very close to $K_\infty$ in the $C^{2,\alpha}$ sense. There exists $\epsilon>0$ such that, after passing to a subsequence,
\begin{enumerate}
\item
For all $n$, the map $f_n = \pi \circ u_n$ is an embedding on $\overline{\Sigma}_{\epsilon,n}$.
\item
For all $n$, the minimal surface $\overline{M}_{n,\epsilon}$ is the graph of a section $s_n$ of the normal bundle $E \to [0,\epsilon]\times K$. 
\item
The sections $s_n$ converge in $C^{2,\alpha}$ on $[0,\epsilon]\times K$, and in $C^{\infty}$ on compact subsets of $(0,\epsilon]\times K$.
\end{enumerate}
\end{theorem}

The proof of Theorem~\ref{convergence-near-infinity} takes up the following four subsections.

\subsection{The convex hull}\label{convex-hull-section}

We begin by recalling the convex hull, which goes back to Anderson~\cite{Anderson}.  The proof of the lemma below is a simple application of barriers, but the conclusion is absolutely crucial in what will follow. 

\begin{lemma}[\cite{Anderson}]\label{barrier}
Let $H \subset \H^4$ be a totally geodesic copy of $\H^3$ which meets the boundary in a copy $S \subset S^3$ of $S^2$. Suppose that a link $K$ lies entirely in one component of $S^3 \setminus S$. Then any complete minimal surface with idea boundary equal to $K$ lies entirely in the corresponding component of $\H^4 \setminus H$. 
\end{lemma}

The convex hull $C(K)$ of $K$ is $C(K) = \H^4 \setminus F$ where $F$ is the union of all the totally geodesic copies of $\H^3$ whose boundaries do not meet $K$. The point is that, by Lemma~\ref{barrier}, any minimal surface with ideal boundary $K$ must lie entirely in $C(K)$

We will use this to get $C^0$ control near infinity of a minimal surface purely in terms of its ideal boundary $K$. The idea is to place a small Euclidean sphere of radius $r$ in $\R^3$ tangent to $K$ at $\eta \in K$, and to do this at all points $\eta \in K$ and for all 2-planes which contain $T_\eta K$. We will then apply Lemma~\ref{barrier} to trap the minimal surface in the exterior of all the corresponding copies of $\H^3 \subset \H^4$. The next definition determines the largest size of spheres we can use.

\begin{definition}
Let $K \subset \R^3$ be a $C^2$ knot or link. Let $\eta \in K$ and $v \in \R^3 \setminus\{0\}$ be normal to $K$ at $\eta$. Write $P(v) = \left\langle v \right\rangle \oplus T_\eta K$, an oriented 2-plane. Given $\rho \in (0,\infty)$, there is a unique closed Euclidean 3-ball of radius $\rho$, lying on the positive side of $P(v)$ and whose boundary sphere passes through $\eta$, where it is tangent to $P(v)$. We denote this 3-ball by $\overline{B}(\eta,v,\rho)$. We now define
\begin{align*}
r_K(\eta,v) &=\sup \{ \rho : \overline{B}(\eta,v,\rho) \cap K = \{\eta\}\} \\
r_K &= \inf \{ r_K(\eta,v) : \eta \in K, v\in T_\eta K^{\perp}\setminus\{0\} \}
\end{align*}
The fact that $K$ is $C^2$ ensures that $r_K>0$. 
\end{definition}

\begin{remark}
The defintion of $r_K$ is similar to the radius of curvature of $K$ (defined via radii of spheres which are tangent to $K$ to second order) but the definition is \emph{global} in the sense that when considering the point $\eta \in K$, we take into account intersections with points of $K$ which may lie a long way from $\eta$.  
\end{remark}

\begin{lemma}\label{convex-hull}
Let $\overline{M} \subset \overline{\H}^4$ be a complete minimal surface with ideal boundary a $C^2$ knot or link $K$. Use half-space coordinates for which $\overline{M} \subset [0,\infty) \times \R^3$. Let $(\xi,p) \in \overline{M}$, where $\xi \in [0,\infty)$ and $p \in \R^3$. If $\xi < r_K$, 
\[
d(p,K)   \leq r_K - \sqrt{r_K^2 - \xi^2}
\] 
where $d(p,K) \defeq \inf_{\eta \in K} d(p,\eta)$ denotes the Euclidean distance from $p$ to $K$ in $\R^3$. 
\end{lemma}
\begin{proof}
Pick a small $\epsilon>0$. We claim that if $\xi < r_K-\epsilon$ then for all $(\xi,p) \in M$, 
\begin{equation}
d(p,K) < r_K - \sqrt{(r_K-\epsilon)^2 -\xi^2}
\label{distance-claim}
\end{equation}
From here the result follows by sending $\epsilon$ to zero. To prove~\eqref{distance-claim}, suppose it is false. Since it certainly holds for $\xi=0$ (when $p \in K$), if it fails to hold there will be a value of $\xi < r_K-\epsilon$ and a point $(\xi,p) \in  M$ for which we have equality in~\eqref{distance-claim}. Write $\eta \in K$ for a point with $d(p,\eta) = r_K - \sqrt{(r_K-\epsilon)^2 -\xi^2}$. Consider the point $q$ on the line in $\R^3$ from $\eta$ to $p$ which lies past $p$ at a distance $r_K$ from $\eta$. The Euclidean 4-dimensional distance from $(\xi,p)$ to $(0,q)$ is $r_K-\epsilon$. On the other hand, the 3-ball $\overline{B}$ centred at $q$ of radius $r_K-\epsilon/2$ is contained inside those balls which are used in the definition of $r_K$. Hence $\overline{B}$ contains no points of $K$ and so, by Lemma~\ref{barrier}, $f(\overline{\Sigma})$ must lie outside the totally geodesic copy of $\H^3$ given by those points whose Euclidean 4-dimensional distance to $q$ is $r_K-\epsilon/2$. Since $(\xi,p)$ is nearer to $(0,q)$ than that we have a contradiction.
\end{proof}

We close this subsection with a lemma on the robustness of this $C^0$ control under $C^2$ convergence of the boundaries.

\begin{lemma}\label{rK-convergence}
Let $K_n \subset \R^3$ be a sequence of links which converge in $C^2$ to a link $K \subset \R^3$. Then $r_{K_n} \to r_K$.
\end{lemma}
\begin{proof}
Let $\overline{B}$ be one of the closed 3-balls which appears in the definition of $r_K$, with boundary tangent to $\eta$. We move it so that it becomes tangent to $K_n$ at the point of $K_n$ nearest to $\eta$. We shrink its radius slightly so that it is marginally smaller than the radius of curvature of $K_n$ at this point, and moreover so that it doesn't meet any other point of $K_n$. Thanks to the $C^2$ convergence, the larger $n$ is the less we need to shrink the radius to ensure these two things happen. The end result is a ball which appears in the definition of $r_{K_n}$. In this way, for any $\epsilon>0$ we have $r_{K_n} \geq r_K - \epsilon$ for all sufficiently large $n$. Reversing the roles of $K_n$ and $K$ in the discussion shows that $r_K \geq r_{K_n}-\epsilon$ also which completes the proof. 
\end{proof}

\subsection{A uniform graphical region}\label{C1-control-section}

We now prove a substantial part of Theorem~\ref{convergence-near-infinity}. A very important tool is the following result of Brian White. To state it, we first need a definition.

\begin{definition}\label{converge-as-sets}
Let $N$ be a Riemannian manifold. We say that \emph{a sequence $M_n \subset N$ of subsets of $N$ converge as sets to $M \subset N$} if
\[
M = \Big\{ p \in N : \limsup_{n} d(p,M_n) = 0 \Big\} = \Big\{ p \in N : \liminf_n d(p,M_n) = 0\Big\}
\]
where $d(p,A)=\inf\{d(p,q) : q \in A\}$. 
\end{definition}

\begin{theorem}[White, \cite{White}]\label{white}
Suppose $X_n$ is a sequence of proper $m$-dimensional stationary integral varifolds without boundary in a Riemannian manifold $N$. Suppose that the $X_n$ converge as sets to a subset of an $m$-dimensional connected $C^1$ properly embedded submanifold $X$ of $N$. If the $X_n$ converge weakly to $X$ with multiplicity~1 anywhere, then they converge smoothly to $X$ everywhere. 
\end{theorem} 

As White points out, the important point here is the word ``anywhere''. He shows this local statement implies that the $X_n$ actually converge with multiplicity~1 to $X$ \emph{everywhere}. From here the smooth convergence follows by Allard regularity. 

We now turn to the main result of this section. See~\eqref{n-epsilon-notation} above  for a reminder of the notation. At this stage it is convenient to try to write the minimal surface $\overline{M}_n$ as a graph over the half-space cylinder $[0,\infty) \times K_n$ generated by its own boundary (rather than passing to a common base link $K$). In this section we are only concerned with $C^1$ control of the graphing functions and so the fact that normal bundle of $K_n$ is only $C^{1,\alpha}$ is not a problem in this regard. We write $E_n \to [0,\infty) \times K_n$ for the normal bundle. In the following statement, given a section $\sigma_n$ of $E_n$ we define $\nabla \sigma_n$ by first parametrising $K_n$ by arc length and then treating $\sigma_n$ as an $\R^3$-valued function of two real variables.  (We write these sections as $\sigma_n$ to distinguish them from the $s_n$ of Theorem~\ref{convergence-near-infinity}, which take values in the \emph{fixed} normal bundle $E \to [0,\epsilon)\times K$.)

\begin{theorem}\label{C1-bound-at-infinity}
Let $u_n \colon (\overline{\Sigma},j_n) \to \overline{Z}$ be a sequence of admissible $J$-holomorphic curves whose boundary links $K_n=\pi(u_n(\del \Sigma))$ converge in $C^{2}$ to an embedded link. There exists $\epsilon>0$ such that for all $n$:
\begin{enumerate}
\item
The conformal harmonic map $f_n = \pi \circ u_n$ is an embedding on  $\overline{\Sigma}_{n,\epsilon}$.
\item
The minimal surface $\overline{M}_{n,\epsilon}$ is the graph of a section $\sigma_n$ of $E_n \to [0,\epsilon] \times K_n$.
\item
$\|\nabla \sigma_n\|_{C^0} < 1$.
\end{enumerate}

\end{theorem}

\begin{proof}
The proof is by contradiction. Suppose the result is not true. Since $f_n$ is an embedding on $\del \Sigma$ and since $\overline{M}_n$ meets infinity at right angles, there is certainly an $\epsilon_n>0$ that works for $f_n$. Assuming the result is false (and after passing to a subsequence which we don't bother to notate) we have a sequence $\epsilon_n>0$ with $\epsilon_n \to 0$ such that everything is fine for $f_n$ in the region $x < \epsilon_n$ but that something goes wrong at $x=\epsilon_n$. The two most important potential problems are as follows:
\begin{enumerate}
\item\label{derivative-bound-breaks}
There is a point $P_n = (\epsilon_n, p_n) \in \overline{M}_n$ at which $|\nabla \sigma_n|=1$. (Here $p_n \in \R^3$ and we're using our fixed choice of half space coordinates.)
\item\label{branch-point-happens}
There is a branch point or self-intersection point $P_n = (\epsilon_n,p_n) \in \overline{M}_{n}$. 
\end{enumerate}
We will show how these problems lead to a contradiction via a rescaling argument, but first there are four other things which might in principle stop the graphical region extending, which we can rule out directly. 

Firstly, one might worry that the surface is about to turn back on itself at $ \epsilon_n$. This corresponds to $|\nabla\sigma_n| \to \infty$ and, since we have $|\nabla \sigma_n|< 1$ for $x< \epsilon_n$, we do not need to worry about this. 

The second thing which could go wrong is that $f_n$ has a critical point which is \emph{not} a branch point. Since $f$ is conformal, the only other option for a critical point is a ramification point, at which the local degree of $f$ is strictly greater than 1. This means that $f_n$ has degree greater than one onto its image, but this contradicts the fact that $f_n$ restricts to an embedding $\del \Sigma \to K_n$. 

The third potential problem is the appearance of a new connected component of $\overline{M}_n$ at $x=\epsilon_n$, which is disjoint from the surface in the region $x<\epsilon_n$. This is also easy to rule out. It corresponds to $x\circ f_n$ having a local minimum. Geometrically this is a tangency of $\overline{M}_n$ with the horosphere $x=\epsilon_n$. The horosphere has positive mean cuvature and the new component of $\overline{M}_n$ lies entirely on the ``inside'' $x\geq \epsilon_n$, in contradiction of the maximum principle. Alternatively, one can use the harmonic map equation to show that $\Delta(x \circ f) \leq 0$ at any critical point of $x$, so again local minima are ruled out by the maximum principle. 

The fourth possible problem is that when we project the curve $M_n \cap \{x = \epsilon_n\}$ to $\R^3$ (by ignoring the $x$ coordinate) it reaches the normal injectivity radius of $K_n$ and so there is no way to write it as a graph of a section of the normal bundle (at least not in a consistent way). This is ruled out by Lemmas~\ref{convex-hull} and~\ref{rK-convergence}. Since $K_n \to K_\infty$ in $C^2$ we have that both $r_{K_n}$ and the normal injectivity radius of $K_n$ are uniformly bounded away from zero. It follows that there exists $c>0$ such that for all $\xi\leq c$ and all large $n$ the curve $M_n\cap \{x = \xi\}$ lies strictly closer to $K_n$ than the normal injectivity radius. 

Having limited the potential problems to the two cases listed above (i.e.\ $|\nabla s_n|=1$, or a branch or self-intersection point) we will now carry out the rescaling argument to derive a contradiction. We first line up the centres of the rescaling. Let $\eta_n \in K_n$ be the point of $K_n$ nearest to $p_n \in \R^3$ (where $(\epsilon_n, p_n)$ is the troublesome point on the minimal surface $M_n$). By passing to a subsequence we can assume that $\eta_n$ converges to a limit in $\R^3$. By choice of coordinates we can assume that $\eta_n \to 0$. We now translate our surfaces by setting 
\[
\overline{Y}_n = \overline{M}_n - (0,\eta_n)
\] 
The ideal boundary of $\overline{Y}_n$ is $\hat{K}_n = K_n - \eta_n$. Since $\eta_n \to 0$, we have $\hat{K}_n \to K_\infty$ in $C^2$. The translated surfaces are also graphs, now over the cylinders $[0,\epsilon_n) \times \hat{K}_n$: we put $\hat{\sigma}_n(\xi, \eta) = \sigma_n(\xi, \eta+ \eta_n) - \eta_n$. The point of the translation is that the boundary links now all pass through the origin, which is also the point nearest to the projection $p_n$ of the ``problem point''. 

We now rescale, with a dilation centred on the origin: $(x,y_i) \mapsto \epsilon_n^{-1}(x,y_i)$. This is a hyperbolic isometry and so the rescaled surfaces $X_n \defeq \epsilon_n^{-1}Y_n$ are still minimal. 
The dilated surfaces are again graphs, this time over cylinders of uniform length: $[0,1)\times \epsilon^{-1}_n\hat{K}_n$. The section whose graph is $X_n$ is simply $\tilde{\sigma}_n(\xi, \eta) = \epsilon^{-1}_n \hat{\sigma}_n(\epsilon_n\xi, \epsilon_n \eta)$.

In each case the difficulty occurred at $P_n=(\epsilon_n, p_n) \in M_{n}$ which now corresponds to the point $Q_n = (1, \epsilon^{-1}_n(p_n-\eta_n)) \in X_n$. Since $\sigma_{n}(0,\eta_n)=\eta_n$ and $\sigma_n(\epsilon_n, \eta_n) = p_n$, and since $|\nabla \sigma_n|\leq 1$ in the region $0\leq x \leq \epsilon_n$, we have
\[
|p_n - \eta_n| \leq \int_{0}^{\epsilon_n} \left|\del_\xi\sigma_n(t,\eta_n) \right|\diff t \leq \epsilon_n
\]
It follows that the troublesome points $Q_n = (1,q_n)$ in the rescaled surfaces have $|q_n|\leq 1$ and so lie in a compact region. We pass to a subsequence so that the $Q_n$ converge to a limit $Q$. 

We now explain how to deduce a contradiction via White's Theorem~\ref{white}. Denote the ideal boundary of $X_n$ by $\tilde{K}_n = \epsilon^{-1}\hat{K}_n$. Each $\hat{K}_n$ passes through the origin and, since they converge in $C^2$ to $K_\infty$, it follows that $\tilde{K}_n$ converges in compact sets to the straight line $L$ which is tangent to $K_\infty$ at the origin. Moreover, by Lemma~\ref{rK-convergence}, $r_{\hat{K}_n} \to r_{K_\infty}$. Now the behaviour of the quantity $r_K$ under rescaling means that $r_{\tilde{K}_n} = \epsilon^{-1}_nr_{\hat{K}_n} \to \infty$. This means we can fit larger and larger barriers around $\tilde{K}_n$ at the origin. Lemma~\ref{convex-hull} now implies that $X_n$ converges in $C^0$ on sets of the form $\{x, |y| \leq R\}$ (in the rescaled coordinates), to the totally geodesic $H \cong \H^2 \subset \H^4$ which meets infinity in the line $L$. It follows that $X_n$ converges to $H$ as sets (in the sense of Definition~\ref{converge-as-sets}) on each of these regions.  Moreover, we know that for $x<1$ (in the rescaled coordinates) this convergence is actually multiplicity~1 because here the $X_n$ are graphs over a sequence of cylinders which converge smoothly to $H$. We are now in a position to apply White's Theorem to conclude that $X_n$ converges in $C^{\infty}$ to $H$, on any region of the form $\{x,|y| \leq R\}$, with multiplicity one everywhere. 

We are now able to reach a contradiction. By passing to a subsequence we can assume that either all the bad points were in case~\ref{derivative-bound-breaks} or all were in case~\ref{branch-point-happens}. Suppose first that at all the points $P_n$ we had $|\nabla \sigma_n(\epsilon_n, \eta_n)| = 1$. Then the same is true for the dilated surfaces at $Q_n$: we have dilated the domain and range of the graphs equally and so  $|\nabla \tilde{\sigma}_n(1,0)| = |\nabla \sigma_n(\epsilon_n,\eta_n))| =1$. Now let $H_n$ denote the totally geodesic copy of $\H^2$ which passes through $Q_n$ and meets infinity in the straight line which is parallel to $T_{0} \hat{K}_n$. We compare the tangent plane $T_{Q_n}X_n$ to $T_{Q_n}H_n$. The fact that $|\nabla \tilde{\sigma}_n|=1$ at this point means that these two 2-planes are different by a definite amount, independent of $n$. For example, the distance between them in the natural metric on the Grassmannian is bounded uniformly away from zero. But both $X_n$ and $H_n$ converge smoothly to $H$ and the points $Q_n$ converge to $Q \in H$, so the tangent planes at $T_{Q_n}X_n$ and $T_{Q_n}H_n$ must become arbitrarily close, which is a contradiction. 

The other possibility is that the points $(\epsilon_n,p_n)$ were all branch points or points of self-intersection. Then the same is true for the points $Q_n$ in the rescaled surfaces. This means that the local denisty of $X_n$ at $Q_n$ is bounded strictly away from 1. But density is lower semicontinuous and so the density of $H$ at the limit $Q$ is also strictly greater than $1$, contradicting the fact that $H$ is smooth. 
\end{proof}

\subsection{\texorpdfstring{$C^2$}{C2} control up to the boundary}\label{C2-control-section}

The next step is to obtain $C^2$ control of the surfaces $\overline{M}_n$, up to the boundary and uniformly in $n$. We do this in Theorem~\ref{C2-bound-at-infinity} below, which gives uniform control over the \emph{Euclidean} second fundamental forms of the surfaces. 

To begin we need some information on the hyperbolic second fundamental forms. Write $A_n$ for the second fundamental form of ${M}_n \subset \H^4$ and $|A_n|$ for its norm (all quantities at this point are defined via the hyperbolic metric). Any minimal surface in $\H^4$ with $C^2$ ideal boundary is asymptotically totally geodesic at infinity. (This follows from the expansion~\eqref{asymptotic-expansion} and also from Lemma~\ref{hyperbolic-to-euclidean} below.) The next lemma shows that this decay of $A_n$ is uniform in $n$. 

\begin{lemma}\label{uniform-decay-hyperbolic-SSF}
 Under the hypotheses of Theorem~\ref{C1-bound-at-infinity}, $A_n \to 0$ as $x \to 0$ uniformly in $n$. I.e.\ for all $e>0$ there exists $\delta>0$ such that if $(x,p) \in {M}_n$ with $x < \delta$ then $|A_n(x,p)| < e$. 
\end{lemma}

\begin{proof}
Assume for a contradiction that the result is false. Then (at least after passing to a subsequence of the $M_n$) there is a value of $e>0$ and a sequence of points $(\delta_n,p_n) \in M_n$ with $\delta_n \to 0$ and $|A_n(\delta_n,p_n)| \geq e$. We now carry out a rescaling exactly as in the proof of Theorem~\ref{C1-bound-at-infinity}. Let $\eta_n \in K_n$ be the point nearest to $p_n$. Setting $X_n = \delta_n^{-1}(M_n-(0,\eta_n))$ we obtain a sequence of minimal surfaces which converges, by White's Theorem,  in $C^{\infty}$ on compact sets to a totally geodesic copy $H \subset \H^4$ of $\H^2$. At the same time a subsequence of the points $Q_n=(1, \delta_n^{-1}(p_n - \eta_n)) \in X_n$ converge to a point $Q$ on $H$. Write $B_n$ for the second fundamental form of $X_n$, and $B$ for the second fundamental form of $H$. On the one hand $B=0$, on the other hand $|B(Q)| = \lim |B_n(Q_n)| = \lim |A_n(\delta_n,p_n)| \geq e>0$. (We have used here that $|B_n(Q_n)| = |A_n(\delta_n,p_n)|$, since the translations and rescalings are hyperbolic isometries).  
\end{proof}

Let $v_n$ be the orthogonal projection of $x\del_x$ to the normal bundle of $M_n$. The quantity $|v_n(x_0,p_0)|$ measures the extent to which $M_n$ fails to be perpendicular to the horosphere $x= x_0$. Any minimal surface in $\H^4$ meets infinity at right angles and so $|v_n| \to 0$ as $x \to 0$. We now show that this happens uniformly in $n$ for our surfaces $M_n$. 

\begin{lemma}\label{uniform-decay-vn}
 Under the hypotheses of Theorem~\ref{C1-bound-at-infinity}, $|v_n| \to 0$ as $x \to 0$ uniformly in $n$. I.e.\ for all $e>0$ there exists $\delta>0$ such that if $(x,p) \in {M}_n$ with $x < \delta$ then $|v_n(x,p)| < e$. 
\end{lemma}

\begin{proof}
The proof is the same as that of Lemma~\ref{uniform-decay-hyperbolic-SSF}. Working by contradiction, the rescaling argument gives a sequence of minimal surfaces which converge in $C^{\infty}$ on compact sets  to a totally geodesic half-plane which is tangent to $x\del_x$. On the other hand the limit has a point at which $x\del_x$ has non-zero normal component. (We use here that the vector field $x\del_x$ is invariant under dilations centred on the origin.)
\end{proof}

We will now work with the \emph{Euclidean} metric $\tilde{g} = x^2 g = \diff x^2 + \diff y^2$. We write $\tilde{A}_n$ for the second fundamental form of $\overline{M}_n$ with respect to $\tilde{g}$. If $v$ is a tangent vector to $\H^4$ we write $\tilde{v} = x^{-1}v$, the point being that $|\tilde{v}|_{\tilde{g}} = |v|_{g}$. %In general, for a tensor $T$ on $\overline{M}$ we write $|T|_g$ for the norm taken using metrics induced from the hyperbolic metric $g$, and $|T|_{\tilde{g}}$ for the norm taken using metrics induced from the Euclidean metric $\tilde{g}$. 

\begin{lemma}\label{hyperbolic-to-euclidean}
Let $\overline{M} \subset \overline{\H}^4$ be an embedded surface. Let $\mu \in T_p\H^4$ be normal to $M$ at $p$ and $u,v \in T_pM$ be tangent to $M$ there. The hyperbolic and Euclidean second fundamental forms in the direction $\mu$ and $\tilde{\mu}$ are related by:
\[
\tilde{A}(\tilde{u},\tilde{w};\tilde{\mu}) = \frac{1}{x}A(u,w;\mu) - \frac{1}{x}\tilde{g}(\tilde{u},\tilde{w}) g (x\del_x, \mu)
\]
It follows that when $M$ is minimal with respect to the hyperbolic metric then its Euclidean mean curvature in the direction $\tilde{\mu}$ is
\[
\tilde{H}(\tilde{\mu}) = -\frac{2}{x} g(x\del_x ,\mu)
\]
and the norms of the second fundamental forms are related by
\[
x^2|\tilde{A}|^2_{\tilde{g}} = |A|^2_{g} + 4|v|^2_g
\]
where $v$ is the orthogonal projection of $x\del_x$ to the normal bundle of $M$. 
\end{lemma}
\begin{proof}
This follows from a direct calculation which we omit, involving the change in the Levi-Civita connection under a conformal change of metric. 
\end{proof}

Putting Lemmas~\ref{uniform-decay-hyperbolic-SSF},~\ref{uniform-decay-vn} and~\ref{hyperbolic-to-euclidean} together we obtain:

\begin{corollary}\label{running-to-the-origin}
Under the hypotheses of Theorem~\ref{C1-bound-at-infinity}, the Euclidean second fundamental forms $\tilde{A}_n$ of the surfaces $\overline{M}_n$ satisfy $x |\tilde{A}_n|_{\tilde{g}} \to 0$ as $x \to 0$, uniformly in~$n$. 
\end{corollary}

We will also need to know the boundary value of the Euclidean mean curvature in terms of the boundary curve.

\begin{lemma}\label{boundary-value-Euclidean-H}
Let $\overline{M} \subset \H^4$ be an embedded surface, which is minimal in $\H^4$ and which meets the boundary at infinity in a $C^2$-curve $K$. Given a normal $u$ to $K \subset \R^3$, the Euclidean mean curvature of $\overline{M}$ in the direction $u$, $\tilde{H}(u)$, defined with respect to the rescaled metric $\tilde{g} = x^2g$ is equal on the boundary $x=0$ to $2\kappa(u)$, where $\kappa(u)$ is the curvature of $K$ in the direction ${u}$. 
\end{lemma} 

\begin{proof}
The surface meets the boundary at right angles. This means that at any point $p\in K$ (i.e. with $x=0$) and normal direction $u$ at $p$, the  principal curvatures of $\overline{M}$ at $p$ corresponding to $u$ occur in the directions tangent to $K$ and orthogonal to it. The $K$ direction gives a principal curvature of $\kappa(u)$. To compute the curvature in the orthogonal direction, we use the asymptotic expansion of Proposition~\ref{asymptotic-expansion}.  Using isothermal coordinates $(s,t)$ on $\overline{M}$ with $t=0$ the boundary, the curve given by setting $s$ to be constant is of the form $t \mapsto (x(t),y(t))$ where
\[
x(t) = |\dot\gamma| t + O(t^3),\quad y(t)= \gamma + \left(\frac{1}{2}\ddot{\gamma} -\frac{\dot{a}}{a}
\dot{\gamma}\right)t^2 + O(t^3)
\]
Here $\gamma$ is the corresponding parametrisation of the boundary $K$. The curvature of this curve in the direction $u$ is readily checked to be $\kappa(u)$ also.
\end{proof}

Each surface $\overline{M}_n$ in Theorem~\ref{C1-bound-at-infinity} is $C^{2,\alpha}$ up to the boundary and so $|\tilde{A}_n|_{\tilde{g}}$ is bounded as $x\to 0$. Our main result in this section is that this bound can be made uniformly in $n$. Before we give the proof, we state the analytic estimate from Calderon--Zygmund $L^p$-theory which will be an important tool. 

\begin{theorem}[The $L^p$ gradient estimate]\label{CZ}
Let $\Omega \subset \R^n$ be a bounded domain with smooth boundary. Let $a^{ij}_{rs} \in C^0(\overline{\Omega})$ for $i,j =1,\ldots ,n$ and $r,s = 1, \ldots , k$ satisfy the uniform ellipticity condition:
\[
a^{ij}_{rs} P^r_iP^s_j \geq c |P|^2,\quad c>0,\quad \forall P \in \R^{nk}
\]
Let $F \in L^p(\Omega, \R^k)$ and $G \in L^p(\Omega, \R^{nk})$ where $p \in (n,\infty)$.

If $f \in W^{1,p}_0(\Omega)$ is a weak solution of the equation
\[
\del_i \left( a_{rs}^{ij} \del_j f^r\right) = F_s + \del_i G^i_s
\]
then 
\[
\| f \|_{W^{1,p}} \leq C \left( \| F\|_{L^p} + \| G\|_{L^p}\right)
\]
where $C$ depends only on $\Omega$, $\|a^{ij}_{rs}\|_{C^0}$ and $c$. 
\end{theorem}
This is due to Calderon-Zygmund \cite{Calderon-Zygmund} in the case of constant coefficients and Campanato \cite{Campanato}, Morrey \cite{Morrey} and Simader \cite{Simader} for continuously varying coefficients. The version stated here is essentially that given in \cite{Dolzmann-Muller} (which itself refers back to Morrey, Theorems~6.4.8 and 6.5.5 for the proof).

We are now ready to prove the main result of this section. Note that we need to use convergence of the boundaries in $C^{2,\alpha}$; for Theorem~\ref{C1-bound-at-infinity} and the other results of \S\ref{compactness-at-infinity} to this point $C^2$ convergence was enough. 

\begin{theorem}\label{C2-bound-at-infinity}
Let $u_n \colon (\overline{\Sigma},j_n) \to \overline{Z}$ be a sequence of admissible $J$-holomorphic curves whose boundary links $K_n=\pi(u_n(\del \Sigma))$ converge in $C^{2,\alpha}$ to an embedded link, for some $0<\alpha <1$. Let $\epsilon>0$ satisfy the conclusions of Theorem~\ref{C1-bound-at-infinity}. There are constants $0<\delta< \epsilon$ and $C$ such that for all large $n$ and all $(x,p) \in \overline{M}_{n}$ with $x \leq \delta$, we have $|\tilde{A}_n(x,p)|_{\tilde{g}} \leq C$.
\end{theorem}

\begin{proof}
We argue by contradiction. If the result is false then (after passing to a subsequence which we don't bother to notate) there is a sequence $(\delta_n,p_n)\in \overline{M}_n$ of points with $\delta_n \to 0$ for which $|\tilde{A}_n(\delta_n,p_n)|_{\tilde{g}} \to \infty$. We will dilate the surfaces to produce a sequence of surfaces with bounded Euclidean second fundamental forms. To do this we need to pick our points carefully. We start with the above sequence $(\delta_n,p_n)$ and now for each $n$ we claim there is a better point $(\delta^*_n, p^*_n) \in \overline{M}_n$ which has the following two properties: 
\begin{enumerate}
\item $|\tilde{A}_n(\delta_n^*,p_n^*)|_{\tilde{g}} \geq |\tilde{A}_n(\delta_n,p_n)|_{\tilde{g}}$.
\item\label{doubling-bound} For all $(x,p) \in \overline{M}_n$, with $x  \leq \frac{1}{|\tilde{A}_n(\delta_n^*,p_n^*)|_{\tilde{g}}}$, we have $|\tilde{A}_n(x,p)|_{\tilde{g}} < 2 |\tilde{A}_n(\delta_n^*,p_n^*)|_{\tilde{g}}$.
\end{enumerate} 
To prove the claim, consider first the point $(\delta_n,p_n)$. If this satisfies the doubling bound~\ref{doubling-bound} then we're done. If not, there is a point $(\delta_n^1,p_n^1) \in \overline{M}_n$ with $\delta_n^1 \leq |\tilde{A}_n(\delta_n,p_n)|_{\tilde{g}}^{-1}$ and with $|\tilde{A}_n(\delta_n^1,p_n^1)|_{\tilde{g}} = 2 |\tilde{A}_n(\delta_n,p_n)|_{\tilde{g}}$. Now if this point satisfies the doubling bound we're finished. If not, there is a point $(\delta_n^2,p_n^2)$ with $\delta_n^2 \leq (2|\tilde{A}_n(\delta_n,p_n)|_{\tilde{g}})^{-1}$ and $|\tilde{A}_n(\delta^2_n,p^2_n)|_{\tilde{g}} = 4|\tilde{A}_n(\delta_n,p_n)|_{\tilde{g}}$. We continue in this way and this process must terminate after a finite number of steps. If not, we would have produced a sequence $(\delta_n^j,p_n^j)$ for $j = 1,2, \ldots $ with $|\tilde{A}_n(\delta_n^j,p_n^j)|_{\tilde{g}} = 2^j |\tilde{A}_n(\delta_n,p_n)|_{\tilde{g}} \to \infty$ as $j \to \infty$ and at the same time $\delta^j_n \leq 2^{1-j}|\tilde{A}_n(\delta_n,p_n)|_{\tilde{g}}^{-1} \to 0$ as $j \to \infty$. This can't happen however: we can pass to a subsequence in $j$ for which the $p_n^j$ converge to $p$ say. Then $|\tilde{A}_n(\delta_n^j,p_n^j)|_{\tilde{g}} \to |\tilde{A}_n(0,p)|_{\tilde{g}}$ which is finite: each surface $\overline{M}_n$ is $C^2$ and so has continuous second fundamental form.

Write $R_n \defeq |\tilde{A}_n(\delta^*_n,p^*_n)|_{\tilde{g}}$. We now carry out the rescaling argument as in the proof of Theorem~\ref{C1-bound-at-infinity}, this time with scale factor $R_n$. In detail, we let $\eta_n \in K_n$ be the point on the boundary link which is closest to $p^*_n$. We then translate and rescale, setting $\overline{X}_n = R_n(\overline{M}_n - (0,\eta_n))$. Just as in the proof of Theorem~\ref{C1-bound-at-infinity}, the surfaces $\overline{X}_n$ converge in $C^\infty$ on compact subsets of the region $x>0$ to a flat half-plane, by White's Theorem. Let $\tilde{B}_n$ be the Euclidean second fundamental form of $\overline{X}_n$. By our choice of the points $(\delta_n^*,p_n^*)$, after rescaling the Euclidean second fundamental form $\tilde{B}_n$ of $\overline{X}_n$ satisfies $|\tilde{B}_n|\leq 2$  in the region $0\leq x\leq 1$, for the rescaled coordinates (since the norm of the second fundamental form scales inversely to length). So we can pass to a subsequence which converges \emph{up to the boundary} in $C^{1,\alpha}$ to a flat half-plane. 

The point $(\delta_n^*,p_n^*) \in \overline{M}_n$ corresponds to the point $Q_n = (R_n\delta_n^*,R_n(p^*_n- \eta_n)) \in \overline{X}_n$. By Corollary~\ref{running-to-the-origin}, $R_n\delta_n^* \to 0$. The bound $|\nabla \sigma_n|\leq 1$ implies that $|p^*_n - \eta_n| \leq \delta_n^*$ and so $R_n(p^*_n - \eta_n) \to 0$ also. So $Q_n \to 0$ converges to the origin. By our choice of scale factor $|\tilde{B}_n(Q_n)|=1$. We want to deduce a contradiction with the fact that the limit is flat. To do this we must upgrade the convergence, which is at this point only in $C^{1,\alpha}$, to ensure that it is actually $C^2$ up to the boundary. 

Write $\tilde{H}_n$ for the Euclidean mean curvature of $\overline{X}_n$. Since the original surfaces $\overline{M}_n$ were $C^{2,\alpha}$ up to the boundary, the same is true of the $\overline{X}_n$ and so $\tilde{H}_n \in C^{0,\alpha}$. The next step is to  prove a uniform $C^{0,\alpha}$ bound on $\tilde{H}_n$ near the origin. 

First we claim that the boundary values of $\tilde{H}_n$ converge to zero in $C^{0,\alpha}$. By Lemma~\ref{boundary-value-Euclidean-H} the mean-curvature on the boundary in the normal direction $u$ to the surface $\overline{M_n}$ at $x=0$ is simply $2\kappa_n(u)$, twice the curvature of the boundary curve $K_n$ in the direction $u$. Since we assume that $K_n \to K_\infty$ in $C^{2,\alpha}$, the boundary values of the mean curvatures \emph{before} rescaling converge in $C^{0,\alpha}$ to to a definite limit. This means that after rescaling, $\tilde{H}_n|_{x=0} \to 0$ in $C^{0,\alpha}$ on compact sets. 

We now extend this control of $\tilde{H}_n$ to the interior. For this, we use the fact that, since the surfaces $\overline{X}_n$ are hyperbolic-minimal, they also solve the \emph{conformally invariant} Willmore equation in Euclidean space. The Willmore equation tell us that on the interior of $\overline{X}_n$ we have
\begin{equation}
\Delta_n (\tilde{H}_n) = F_n
\label{Willmore-n}
\end{equation}
Here $\Delta_n$ is the covariant Laplacian on the normal bundle of $\overline{X}_n$ and the right-hand side $F_n$ is built out algebraically of the second fundamental form $\tilde{B}_n$. (In arbitrary dimension this goes back at least to Weiner~\cite{Weiner}.) The precise form of $F_n$ is not important here, but what does matter is that we have a uniform bound $\|F_n\|_{C^0} \leq C$ because $|\tilde{B}_n| \leq 2$. 

The idea behind the next step is now simple to state: the fact that the rescaled surfaces have $|\tilde{B}_n|\leq 2$ and converge in $C^{1,\alpha}$ means that the coefficients of $\Delta_n$ have appropriate behaviour for us to be able to ultimately apply the $L^p$ estimate of Theorem~\ref{CZ} with a constant which is independent of $n$. This, together with Sobolev embedding will give a uniform Hölder bound on $\tilde{H}_n$ on a neighbourhood of the boundary. 

We now give the details. The surfaces $\overline{X}_n$ converge to a flat half-plane. We choose our coordinates so that this half-plane is given by $y_2=0=y_3$. We write the piece of surface $\overline{X}_n \cap \{ x\leq 1,\ |y_1|\leq 1\}$ as graph over $[0,1]\times[-1,1]$, by forgetting the $y_2,y_3$ coordinates. In this way we use $x,y_1$ as coordinates on each of the $\overline{X}_n$. In fact it will be more convenient to use a coordinate chart with a smooth boundary. Write $\overline{\Omega}$ for a domain in $[0,1]\times [-1,1]$ which contains $\{ (0,y_1) : |y_1| \leq 3/4\}$ and which has a smooth boundary. We use $\overline{\Omega}$ as our coordinate patch on each $\overline{X}_n$.  

Next we trivialise the normal bundle of $\overline{X}_n$ over $\overline{\Omega}$, by picking an orthonormal frame $v_1,v_2$ for the normal bundle of $\overline{X}_n$. Let $a_n$ be the 1-form defined by $a_n(u)= \left\langle u\cdot v_1, v_2 \right\rangle$. (The subscript $n$ is to remind us we are working on the surface $\overline{X}_n$.) Then with respect to this frame, a section $\phi = \phi_1v_1 + \phi_2v_2$ has covariant derivative equal to 
\[
\nabla \begin{pmatrix} \phi_1 \\ \phi_2 \end{pmatrix} 
=
\begin{pmatrix}
\diff \phi_1 - a_n \otimes \phi_2\\
 \diff \phi_2 + a_n\otimes \phi_1
\end{pmatrix}
\]
It follows that
\[
\Delta_n \phi= -\diff^*_n\diff \phi -
\begin{pmatrix}
\diff^*_na_n\, \phi_2 + 2 \left\langle a_n, \diff \phi_2 \right\rangle - |a_n|^2 \phi_1\\
\diff^*_na_n\, \phi_1 + 2 \left\langle a_n, \diff \phi_1 \right\rangle - |a_n|^2 \phi_2\\
\end{pmatrix}
\]
Here the operator $\diff^*_n$ is the adjoint of the exterior derivative with respect to the metric on~$\overline{X}_n$. 

Up to here the frame was arbitrary and so we have no control over the resulting coefficients of $\Delta_n$. We now explain how to choose the frame in order to better control the terms involving~$a_n$. The vector $\del_{y_2}$ is normal to the limiting flat plane. For each $n$, we project it to the normal bundle of $\overline{X}_n$ and normalise it to have unit length. Denote the result by $v_1$. Since the surfaces $\overline{X}_n$ converge to the plane in $C^{1,\alpha}$,  $v_1$ converges back to $\del_{y_2}$ in $C^{0,\alpha}$. Moreover, the uniform bound $|\tilde{B}_n| \leq 2$ implies a uniform $C^1$-bound on $v_1$. The second element in the frame $v_2$ is determined by $v_1$ via the orientation. Now the $C^1$ bound on $v_1$ implies $\|a_n\|_{C^0} \leq C$. 

Next we rotate our frame to eliminate the $\diff^*_na_n$ term (using ``Coulomb gauge'' in the language of gauge theory). For each $n$, we let $\theta_n$ solve $-\diff^*_n\diff\, \theta_n = \diff^* a_n$. Then the new frame $v'_1 = (\cos \theta_n) v_1 + (\sin \theta_n) v_2$, and $v_2' = - (\sin \theta_n) v_1 + (\cos \theta_n) v_2$ has $a_n' = a_n + \diff \theta_n$ and so $\diff^*_na_n'=0$. We want to do this whilst controlling $\diff \theta_n$ and for this we impose the condition that $\theta_n$ vanishes on the boundary $\del \Omega$. We write out explicitly the equation $\theta_n$ solves in our coordinates:
\begin{equation}
\del_i (\sqrt{|g|} g^{ij} \del_j \theta) = - \del_i (\sqrt{|g|}a_n)^i
\label{Laplacian-in-coordinates}
\end{equation}
Here $g_{ij}$ are the metric coefficients of $\overline{X}_n$ in our coordinate patch. The terms $\sqrt{|g|}$ and $g^{ij}$ converge in $C^{0,\alpha}$, since the surfaces themselves converge in $C^{1,\alpha}$. This means that for any choice of $p$, we can apply the $L^p$ gradient estimate of Theorem~\ref{CZ}:
\[
\| \diff \theta_n \|_{L^p} \leq C \| a_n\|_{L^p} \leq C'
\]
where $C,C'$ do not depend on $n$ (but will depend on $p$). We make two side remarks: firstly the fact that we only have a $C^0$ bound on $a_n$, as opposed to a $C^{0,\alpha}$ bound, is what means we have to use $L^p$ estimates here rather than Schauder estimates. Secondly the estimate taken straight from the books uses $L^p$-norms defined by the flat metric $\delta_{ij}$ on $\overline{\Omega}$; this is uniformly equivalent to the $L^p$-norms defined by $g_{ij}$ and so it does not matter which we use here; we will not mention this sort of point from here on. From this estimate it follows that 
\[
\|a_n'\|_{L^p} \leq \| a_n\|_{L^p} + \| \diff \theta_n\|_{L^p} \leq C. 
\]
From now on we use this frame, and drop the prime notation. So we have $\diff^*_na_n=0$ and $\|a_n\|_{L^p} \leq C_p$, where $p$ is at our disposal. 

Notice that because $\diff^*_na_n=0$, we have $\diff^*_n (a_n \phi_i) = \left\langle a, \diff \phi_i \right\rangle$. This means that for our special choice of frame the covariant Laplacian on the normal bundle is given by 
\[
\Delta_n \phi
	= 
		- \diff^*_n \diff \phi 
		- \begin{pmatrix}
		2 \diff^*_n (a_n\phi_2) - |a_n|^2 \phi_1\\
		2 \diff^*_n (a_n \phi_1) - |a_n|^2 \phi_2
		\end{pmatrix}
\]
Write $\tilde{H}_n = \phi_{n,1}v_1 + \phi_{n,2}v_2$ in this frame. Then the equation $\Delta_n\tilde{H}_n = F_n$ implies that $\phi_n=(\phi_{n,1},\phi_{n,2})$ satisfies
\begin{equation}
-\diff^*_n \diff \phi_n
=
\Phi_n + \diff^*_n G_n
\label{Willmore-divergence}
\end{equation}
where the right-hand side here is given by
\begin{align*}
\Phi_n &= F_n + \begin{pmatrix} |a_n|^2 \phi_{n,1}\\ |a_n|^2 \phi_{n,2}\end{pmatrix}\\
G_n & = \begin{pmatrix} 2 a_n \phi_{n,2}\\ 2a_n \phi_{n,1}\end{pmatrix}
\end{align*}
The point of writing the equation this way is that, because $\tilde{H}_n$ and $F_n$ are bounded in $C^0$ and $a_n$ is bounded in $L^p$, the source terms on the right-hand side are controlled: $\Phi_n$ is bounded in $L^{p/2}$ and $G_n$ is bounded in $L^p$.

We now want to apply our $L^p$ gradient estimate to~\eqref{Willmore-divergence} to obtain an $L^p$ bound on $\diff \phi$. To do so we need to to work with a function which vanishes on the boundary. To arrange this, let $\chi \colon \overline{\Omega} \to [0,1]$ be a smooth cut-off function, equal to 1 on a neighbourhood of the points of the form $(0,y_1)$ for $|y_1|<1/2$ and equal to zero on the part of the boundary of $\overline{\Omega}$ which doesn't lie on the $y_1$-axis. Now put $\psi_n = \chi \phi_n|_{\del \Omega}$. Note that $\psi_n \to 0$ in $C^{0,\alpha}(\del\Omega)$, because the boundary value of $\tilde{H}_n$ itself converges to zero on the part of $\overline{X}_n$ corresponding to the $y_1$-axis. We write $\overline{\psi}_n$ for the harmonic extension of $\psi_n$ to the whole of $\overline{\Omega}$. I.e. $\diff^*_n\diff \overline{\psi}_n=0$ and $\overline{\psi}_n|_{\del \Omega} = \psi_n$. We then consider
\[
\hat{\phi}_n \defeq \chi \phi_n - \overline{\psi}_n
\]
which vanishes on the boundary.

The harmonic extension $\overline{\psi}_n$ solves the equation
\[
\del_i (\sqrt{|g|}g^{ij}\del_j)\overline{\psi}_n = 0
\]
Elliptic theory gives a constant $C$ such that
\[
\| \overline{\psi}_n\|_{C^{0,\alpha}(\overline{\Omega})} 
\leq
C \| \psi_n \|_{C^{0,\alpha}(\del \Omega)}
\]
Moreover, since the $g_{ij}$ converge in $C^{0,\alpha}$, the constant in this estimate can be taken independent of $n$. This means that $\overline{\psi}_n \to 0$ in $C^{0,\alpha}(\overline{\Omega})$ and so to bound $\chi \phi_n$ in $C^{0,\alpha}$ it suffices to bound $\hat{\phi}_n$. To this end we compute, via~\eqref{Willmore-divergence}
\begin{align*}
-\diff^*_n\diff \hat{\phi}_n
	&=
		-\diff^*_n\diff (\chi \phi_n)\\
	&=
		-\chi\, \diff^*_n\diff \phi_n
		-2\diff^*_n(\phi_n\diff \chi)
		+ 
		\phi_n \diff^*_n\diff \chi\\
	&=
		\hat{\Phi}_n + \diff_n^*\hat{G}_n
\end{align*}
from~\eqref{Willmore-divergence}, where
\begin{align*}
\hat{\Phi}_n
	&=
		\chi \Phi_n + \phi_n\diff^*_n\diff \chi - \left\langle \diff \chi, G_n \right\rangle,\\
\hat{G}_n
	&=
		\chi G_n - 2\phi_n\diff \chi.
\end{align*}
Importantly, we have uniform bounds: $\|\hat{\Phi}_n\|_{L^{p/2}} \leq C$ and $\|\hat{G}_n\|_{L^p} \leq  C$, where we are free to take $p$ as we please.

Now we can apply the Calderon--Zygmund estimate. Explicitly, in coordinates, $\hat{\phi}_n$ solves the PDE
\begin{equation}
\del_i (\sqrt{|g|}g^{ij}\del_j \hat{\phi}_n)
=
\sqrt{|g|}\hat{\Phi}_n + \del_i \left(\sqrt{|g|} \hat{G}_n\right)^i
\label{final-PDE-for_CZ}
\end{equation}
The terms $g^{ij}$ and $\sqrt{|g|}$ converge in $C^{0,\alpha}$, $\hat{\phi}_n$ vanishes on the boundary and the source terms on the right-hand side are both bounded in $L^{p/2}$. So we have a uniform constant $C$ such that
\[
\| \hat{\phi}_n\|_{W^{1,p/2}} \leq C
\]
Taking $p$ sufficiently large, Sobolev embedding then implies a uniform bound $\|\hat{\phi}_n\|_{C^{0,\alpha}} \leq C$. It follows that there is a subdomain $\overline{\Omega}' \subset \overline{\Omega}$  (namely the region where $\phi=1$) containing the points $(0,y_1)$ with $|y_1| \leq 1/2$ for which $\|\tilde{H}_n\|_{C^{0,\alpha}(\overline{\Omega}')} \leq C$. 

The contradiction is now nearly at hand. Let $w_n \colon \overline{\Omega} \to \R^2$ be the function whose graph is $\overline{X}_n$. We know that $w_n \to 0$ in $C^{1,\alpha}(\overline{\Omega})$. Write $T = \{(0,y_1): |y_1|\leq 1/4\}$. We have  that $w_n|_{T} \to 0$ in $C^{2,\alpha}(T)$: the boundaries of the original surfaces converge in $C^{2,\alpha}$, so the rescaled boundaries converge to the boundray of the flat half-plane in $C^{2,\alpha}$. Now the mean curvature $\tilde{H}_n$ is a quasi-linear elliptic combination of the graphing functions $w_n$. Schematically,
\[
\tilde{H}_n = L(\nabla w_n)(w_n)
\]
where $L(\nabla w_n) = L_n$ is a linear elliptic operator whose coefficients depend on $\nabla w_n$. Let $\overline{\Omega}'' \subset \Omega'$ be a smaller domain, with $\{(0,y_1): |y_1|<1/5\}$ in its boundary. Schauder boundary estimates imply that 
\[
\| w_n\|_{C^{2,\alpha}(\overline{\Omega}'')}
\leq
C \left( \| L_n w_n\|_{C^{0,\alpha}(\Omega')} + \| w_n|_T\|_{C^{2,\alpha}(T)}\right)
\]
Since the $w_n$ converge in $C^{1,\alpha}$, the coefficients of $L_n$ converge in $C^{0,\alpha}$ and the constant $C$ here can be taken independent of $n$. Now the fact that 
\[
\|L_n w_n\|_{C^{0,\alpha}(\Omega')} = \| \tilde{H}_n\|_{C^{0,\alpha}(\Omega')} \leq C
\]
and $\| w_n|_T\|_{C^{2,\alpha}(T)} \to 0$ imply that $w_n$ is bounded in $C^{2,\alpha}(\overline\Omega'')$. We can now pass to a subsequence for which the $w_n \to 0$  in $C^2$. This means that, near the origin at least, the Euclidean second fundamental form $\tilde{B}_n$ of $\overline{X}_n$ converges to zero. At the same time, by choice of our rescaling, $|\tilde{B}_n(Q_n)|=1$ and yet $Q_n \to (0,0)$, so $|\tilde{B}_n(Q_n)| \to 0$. This contradiction completes the proof.
\end{proof}

\subsection{The proof of Theorem~\ref{convergence-near-infinity}}\label{C2alpha-control-section}

We now prove the main result of \S\ref{compactness-at-infinity}, Theorem~\ref{convergence-near-infinity}. We switch to the set-up described in \S\ref{statement-of-compactness-near-infinity}, so that $K$ is a smooth link very close in $C^{2,\alpha}$ to the limit $K_\infty$ of the $K_n$. We know from Theorem~\ref{C1-bound-at-infinity} that the surfaces $\overline{M}_{n,\epsilon}$ are graphs over $[0,\epsilon]\times K_n$ of sections $\sigma_n$ of $E_n$. Since these cylinders converge to $[0,\epsilon] \times K_\infty$, as long as we take $K$ sufficiently close to $K_\infty$, we will be able to write $\overline{M}_{n,\epsilon}$ as a graph over the \emph{fixed} cylinder $[0,\epsilon]\times K$ of a section $s_n$ of the normal bundle $E$. 

By Theorem~\ref{C2-bound-at-infinity}, the surfaces $\overline{M}_{n,\epsilon}$ have uniformly bounded Euclidean second fundamental forms, and so the sections $s_n$ are bounded in $C^2$. We pass to a subsequence so that they converge in $C^{1,\alpha}$. Since the surfaces are minimal in the region $x>0$,  standard elliptic regularity for minimal surfaces means that we can take this subsequence to converge in $C^{\infty}$ on compact subsets of $(0,\epsilon)\times K$. The main point is to extend this to $C^{2,\alpha}$ convergence up to the boundary at $x=0$. We do this following the same argument as in the proof of Theorem~\ref{C2-bound-at-infinity}, via the Willmore equation. 

The Euclidean mean curvature $\tilde{H}_n$ of $\overline{M}_{n,\epsilon}$ solves the Willmore equation on the interior. We pull everything back to $[0,\epsilon]\times K$ via the sections $s_n$ and use a Coulomb gauge trivialisation of the normal bundles, to obtain a sequence of elliptic systems on a fixed domain. The point is that the $C^{1,\alpha}$ convergence of the sections $s_n$ means that the various divergence-form differential operators which appear have coefficients which converge in $C^{0,\alpha}$, just as in the proof of Theorem~\ref{C2-bound-at-infinity}, and so we can argue identically. This gives a uniform $C^{0,\alpha}$ bound on the Euclidean mean curvature $\tilde{H}_n$ of $\overline{M}_{n}$, valid for $x \leq \epsilon/2$, say, and hence a uniform $C^{2,\alpha}$ bound on the sections $s_n$ in the same region. 

We now explain how the same argument actually gives $C^{2,\alpha}$ convergence of the $s_n$. We use the same notation as before. We work on the domain $\overline{\Omega} = [0,\epsilon] \times K$; $\tilde{H}_n$ is represented in our trivialisation by the $\R^2$-valued function $\phi_n$; $\chi$ is a cut-off function equal to $1$ on a neighbourhood of the boundary $x=0$ and supported in, say, $0 \leq x \leq \epsilon/2$; $\psi_n = \chi \phi_n|_{\del \Omega}$ and $\overline{\psi}_n$ is the harmonic extension of $\psi_n$ to $\overline{\Omega}$ with respect to the metric on $\overline{M}_{n,\epsilon}$ (pulled back to $\overline{\Omega}$ via $s_n$). 

The first thing to say is that $\psi_n$ converges in $C^{0,\alpha}(\del \Omega)$. On the boundary $x=0$ this follows from Lemma~\ref{boundary-value-Euclidean-H} and the fact that the boundary links $K_n$ converge in $C^{2,\alpha}$. On the boundary $x=\epsilon$, $\psi_n$ is identically zero. Over $\Omega$, $\overline{\psi}_n$ is harmonic, $\Delta_n\overline{\psi}_n=0$ for a sequence of operators $\Delta_n$ whose coefficients converge in $C^{0,\alpha}$. This is because the sequence of metrics on $[0,\epsilon]\times K$, obtained by pull-back from $\overline{M}_{n,\epsilon}$ via $s_n$, converge in $C^{1,\alpha}$. Since the boundary values also converge in $C^{0,\alpha}$ it follows that $\psi_n$ converges in $C^{0,\alpha}(\overline{\Omega})$. 

Next we consider $\hat{\phi}_n = \chi \phi_n - \overline{\psi}_n$ which solves the equation~\eqref{final-PDE-for_CZ}. The source terms on the right-hand side converge in $L^{p/2},$ the coefficients in $C^{0,\alpha}$ and so the solutions converge in $W^{1,p/2}$ and hence, by Sobolev embedding once $p$ is high enough, in $C^{0,\alpha}$. This means that the $\tilde{H}_n$ converge in $C^{0,\alpha}$, at least for $0 \leq x \leq \epsilon/2$. 

Now the sections $s_n$ solve a quasilinear PDE of the form $L_n(s_n) = \tilde{H}_n$, where the coefficients of $L_n$ are determined by $\nabla s_n$. Since the coefficients converge in $C^{0,\alpha}$, the right-hand side converges in $C^{0,\alpha}$ and the boundary values of the $s_n$ at $x=0$ converge in $C^{2,\alpha}$ we can conclude that the $s_n$ converge near $x=0$ in $C^{2,\alpha}$. 

\section{Convergence of \texorpdfstring{$J$}{J}-holomorphic maps}\label{convergence-interior-section}

As before, let $u_n \colon (\overline{\Sigma},j_n) \to \overline{Z}$ be a sequence of admissible $J$-holomorphic curves and $f_n = \pi \circ u_n$ the associated sequence of conformal harmonic maps to $\H^4$. We continue to assume that the boundary links $K_n =\pi(f_n(\del \Sigma))$ converge in $C^{2,\alpha}$. In this section we address the convergence of the maps $u_n$ themselves, over the whole of $\overline{\Sigma}$.

\subsection{Statement of main results of \S\ref{convergence-interior-section}}

Our first main result concerns holomorphic discs.

\begin{theorem}\label{properness-for-discs}
Let $u_n \colon (\overline{D}, j_0) \to \overline{Z}$ be a sequence of admissible $J$-holomorphic maps from the closed unit disc with its standard complex structure. Assume that the boundary knots $K_n = \pi(u_n(S^1) )\subset S^3$ converge in $C^{2,\alpha}$ to an embedded limit $K_\infty$. Write $f_n = \pi \circ u_n \colon \overline{D} \to \overline{\H}^4$ for the corresponding sequence of conformal harmonic maps.
\begin{enumerate}
\item
There exists a sequence of biholomorphisms $\phi_n \in \PSL(2,\R)$ of $\overline{D}$ such that, after passing to a subsequence, $f_n \circ \phi_n$ converges, in $C^{\infty}$ on the interior $D$ and in $C^{2,\alpha}$ up to the boundary, to a conformal harmonic map $f _\infty \colon \overline{D} \to \overline{\H}^4$ with $f_{\infty}(\del \Sigma) = K_{\infty}$
\item
If the limit $f_\infty$ has no branch points, then the same subsequence of $u_n \circ \phi_n$ converges in the space $\A$ of admissible maps to a $J$-holomorphic map $u_\infty$ which is the twistor lift of~$f_\infty$.  
\end{enumerate}
\end{theorem}

Notice that to ``upgrade'' the convergence of the $f_n$ to convergence of the $u_n$ we assume that the limit $f_\infty$ has no branch points. The precise reason for this will be made clear in~\S\ref{interior-estimates-section}. Put briefly, a branch point in $f_\infty$ may indicate a vertical bubble forming in the $u_n$. 

When considering maps with more complicated domains, we must allow for the fact that the domain may develop nodes. To this end we first recall the definition of nodal Riemann surfaces.

\begin{definition}
A \emph{nodal Riemann surface with boundary} is a pair $(\overline{S},N)$ where $\overline{S}$ is a compact Riemann surface with boundary and finitely many connected components, and $N \subset \overline{S}$ is a finite set of points of the form $N = \{p_1,q_1, \ldots , p_m,q_m\}$. We write $\overline{S}/N$ for the space obtained by identifying these points in pairs, $p_i \sim q_i$. A pair of identified points are either both on the boundary or both in the interior. When identified, these points are the \emph{nodes} of $\overline{S}/N$. We denote the node corresponding to $p_i\sim q_i$ by $N_i$. Note that components are allowed to have empty boundary. (Since we do not consider nodal surfaces with varying complex structure we do not bother to notate the complex structure on $\overline{S}$.)
\end{definition}

We briefly recall how nodal Riemann surfaces arise as limits of varying complex structures. We assume in this paragraph that $\chi(\overline{\Sigma})<0$. Let $j_n$ be a sequence of complex structures on $\overline{\Sigma}$ and write $[j_n]$ for the corresponding sequence in the moduli space. Let $h_n$ denote the complete hyperbolic metric on the interior $\Sigma$ corresponding to $j_n$. For each $n$ we write $\gamma_n$ for the $h_n$-geodesic which is homotopic to $\del \Sigma$. If $[j_n]$ limits to a nodal surface with boundary nodes, then the length $l(\gamma_n)$ tends to infinity. We will show this does not happen in the cases that interest us. Accordingly we only consider the case when $[j_n]$ approaches the boundary of the moduli space as $n \to \infty$, but with no boundary nodes. In this case there exists a sequence  $\gamma_1,\ldots, \gamma_m \subset \Sigma$ of homotopically non-trivial closed geodesics for $h_n$, whose lengths go to zero as $n \to \infty$. Acting by a diffeomorphism for each $n$, we can assume that these curves are independent of $n$ as subsets of $\Sigma$. The limiting nodal surface $\overline{S}/N$ is produced, topologically at least, by collapsing these geodesics to points.

The punctured surface $S\setminus N$ carries a natural complete hyperbolic metric $h_\infty$ which is a the limit of the $h_n$ modulo diffeomorphisms in the  following sense. There is a sequence of diffeomorphisms $\phi_n \colon \overline{\Sigma} \to \overline{\Sigma}$ and a continuous map $\psi \colon \overline{\Sigma} \to \overline{S}/N$ such that:
\begin{enumerate}
\item\label{convergence1}
 The restriction of $\psi$ gives a diffeomoprhism $\overline{\Sigma} \setminus (\gamma_1 \cup \cdots \cup \gamma_m) \to \overline{S}\setminus N$.
\item\label{convergence2}
The pre-image of each node is the corresponding closed geodesic, $\psi^{-1}(N_i) = \gamma_i$, for the metrics $h_n$.
\item\label{convergence3}
Given any neighbourhood $U\subset S$ of $N$, we have $\psi_*(\phi_n^*h_n) \to h_\infty$ in $C^{\infty}(S \setminus U)$.
\end{enumerate}

We also need the definition of nodal maps in our setting, and how they can arise as limits.

\begin{definition}\label{nodal-map-definition}
Let $(\overline{S},N)$ be a compact nodal Riemann surface with boundary, and with no boundary nodes. 
\begin{itemize}
\item
An \emph{admissible nodal $J$-holomorphic map}  $u \colon (\overline{S},N) \to \overline{Z}$ is $J$-holomorphic map on the interior $S$ and which is compatible with the identifications, i.e. $u(p_i) = u(q_i)$ for all~$i$. When a component of $S$ has non-empty boundary, we require that $u$ is admissible on that component, in the sense of Definition~\ref{admissible-maps}. 
\item
An \emph{admissible nodal conformal harmonic map} $f \colon (\overline{S},N) \to \overline{\H^4}$ is a conformal map which is harmonic on the interior, compatible with the identifications, and whose twistor lift is admissible on each component with non-empty boundary. 
\end{itemize}
\end{definition}

\begin{definition}
Let $f_n \colon (\overline{\Sigma}, j_n) \to \overline{\H}^4$ be a sequence of admissible conformal holomorphic maps. Suppose that $[j_n]$ converges to a nodal limit $(\overline{S},N)$ in the boundary of the moduli space (in the Deligne--Mumford compactifictaion). We assume that there are no boundary nodes. We say that \emph{the maps $f_n$ converge modulo diffeomorphisms to an admissible nodal conformal harmonic map $f$ with domain $(\overline{S},N)$}, if there are maps $\phi_n \colon \overline{\Sigma}\to \overline{\Sigma}$ and $\psi \colon \overline{\Sigma} \to \overline{S}/N$ satisfying conditions~\ref{convergence1},~\ref{convergence2} and~\ref{convergence3} from above (so that $\psi_*(\phi_n^*h_n) \to h_\infty$ away from the nodes) and, moreover,
\begin{enumerate}
\item
$f_n \circ \phi_n \circ \psi^{-1} \to f$ in $C^{\infty}_{\mathrm{loc}}(S\setminus N)$.
\item
$f_n \circ \phi_n \circ \psi^{-1}  \to f$ in $C^{2,\alpha}$ on a neighbourhood of the boundary  of $\overline{S} \setminus N$.
\end{enumerate}

\begin{remark}
This definition is deliberately less refined than convergence in the sense of stable maps which is typically used in Gromov--Witten theory. We make no attempt to define convergence when the limiting domain has boundary nodes and we do not introduce extra spherical components in the domain, only using those which come from the degeneration of the $j_n$. This is because we will rule out the formation of boundary nodes and avoid dealing with bubbles.
\end{remark}

\end{definition}

With these definitions in hand, we can state the second main result of this section.

\begin{theorem}\label{convergence-interior}
Let $\overline{\Sigma}$ be a compact connected Riemann surface with non-empty boundary, and $\chi(\Sigma) \leq 0$. Let $u_n \colon (\overline{\Sigma},j_n) \to \overline{Z}$ be a sequence of admissible $J$-holomorphic curves. Assume the boundary links $K_n \defeq \pi(u_n(\del \Sigma) )\subset S^3$ converge in $C^{2,\alpha}$ to an embedded limit $K_\infty$. Write $f_n = \pi \circ u_n$ for the corresponding sequence of admissible conformal harmoinc maps.
\begin{enumerate}
\item\label{smooth-limit}
If the complex structures $[j_n]$ converge to a smooth limit $[j_\infty]$ then 
	\begin{enumerate}
	\item
	A subsequence of the $f_n$ converges modulo diffeomorphisms to a conformal harmonic map $f_\infty \colon  (\overline{\Sigma},j_\infty) \to \overline{\H}^4$, with $f_\infty(\del \Sigma) = K_\infty$. 
	\item
	If, moreover, $f_\infty$ has no branch points, then the corresponding $J$-holomorphic maps $u_n$ also converge modulo diffeomorphisms, in the space $\A$ of admissible maps, to a $J$-holomorphic map $u_\infty \colon (\overline{\Sigma},j_\infty) \to \overline{Z}$, which is the twistor lift of $f_\infty$. 
	\end{enumerate}
\item\label{nodal-limit}
If the complex structures $[j_n]$ converge to a nodal limit $(\overline{S},N)$ in the Deligne--Mumford compactification of the moduli space, then 
	\begin{enumerate}
	\item\label{no-boundary-nodes}
	There are no nodes on the boundary of $\overline{S}$.
	\item\label{nodal-limit-of-maps}
	A subsequence of the $f_n$ converges modulo diffeomorphisms to an admissible conformal harmonic map $f_\infty$ with domain ${(\overline{S},N)}$ and boundary link $f_\infty(\del S) = K_\infty$. 
	\item\label{no-closed-components}
	Every component of $\overline{S}$ on which $f_\infty$ is non-constant has non-empty boundary. 
	\end{enumerate}
\end{enumerate}
\end{theorem}

Just as for Theorem~\ref{properness-for-discs}, in part~\ref{smooth-limit} here, we can only convert the convergence of the $f_n$ to convergence of the $u_n$ under the assumption that the limit $f_\infty$ has no branch points. In part~\ref{nodal-limit}, even when the limit $f_\infty$ has no branch points, we cannot conclude that the $u_n$ converge to a $J$-holomorphic curve with the same domain. As we will see in Remark~\ref{nodes-means-df-tends-to-zero}, this is because $|\diff f_n|$ will necessarily tend to zero where the nodes are forming and so vertical bubbles may potentially appear in the $u_n$ at exactly the same places the nodes are forming. 

The proofs of Theorems~\ref{properness-for-discs} and~\ref{convergence-interior} take up the remainder of this section.

\subsection{Interior estimates}\label{interior-estimates-section}

The first step is to use the theory of harmonic maps to control the derivatives of $f = \pi \circ u$ and in turn the $J$-holomorphic map $u$.

\begin{proposition}\label{interior-estimate}
For any admissible $J$-holomoprhic map $u \colon (\overline{\Sigma},j) \to \overline{Z}$ with projection $f = \pi \circ u$, we have 
\[
| \diff f|^2 \leq 2
\]
where we use the hyperbolic metric on $\H^4$ and the canonical complete hyperbolic metric on the open Riemann surface $(\Sigma,j)$ to measure the norm of $\diff f$. 
 
It follows that for each $k$, there exists a constant $C=C(k)$, independent of $u$, such that 
\[
| \nabla^k f| \leq C
\] 
where $\nabla^k f \in T^*{\Sigma}^{\otimes k} \otimes f^*T\H^4$ is the $k^{\text{th}}$ derivative of $f$ taken with respect to the Levi-Civita connections of the hyperbolic metric on $\H^4$ and the canonical complete hyperbolic metric on the open Riemann surface $(\Sigma,j)$; these same metrics are used to measure the norm.
\end{proposition}

\begin{proof}
Since $u$ is $J$-holomorphic, $f$ is both conformal and harmonic. We will first show that $|\diff f|^2 \to 2$ at $\del \Sigma$. 

Let $s+it$ be a holomorphic coordinate centred at a point $p \in \del \Sigma$, with $t>0$ on the interior. We also choose $(s,t)$ so that the hyperbolic metric on $\Sigma$ is standard here, given by $t^{-2}(\diff s^2 +\diff t^2)$. Since $u$ is $J$-holomorphic, by Proposition~\ref{asymptotic-expansion} we have
\begin{equation}
\begin{aligned}
x(s,t) &= |\dot \gamma(s)|t + O(t^2)\\
y(s,t) & = \gamma(s) + O(t^2)
\end{aligned}
\label{xy-to-t-squared}
\end{equation}
where $(x,y_i)$ are half-space coordinates on $\H^4$, pulled back to $\Sigma$ via $f$, $\gamma$ is the boundary value of $f$, with values in $\R^3$, and $|\dot\gamma|$ is the Euclidean norm of $\dot \gamma$. By definition,
\[
|\diff f|^2
	= 
		\frac{t^2}{x^2}\left[ 
		\left(\frac{\del x}{\del t}\right)^2
		+
		\left(\frac{\del x}{\del s}\right)^2
		+
		\left|\frac{\del y}{\del t}\right|^2
		+
		\left|\frac{\del y}{\del s}\right|^2
		\right].
\]
It follows from this and~\eqref{xy-to-t-squared} that $|\diff f|^2 = 2 +O(t)$.

Now the Bochner formula for a harmonic map $f \colon (M,h) \to (N,g)$ between Riemannian manifolds states that
\begin{align*}
\frac{1}{2}\Delta |\diff f|^2 
	&= 
		|\nabla^2 f|^2 
		+ 
		\left\langle \Ric_M \diff f, \diff f \right\rangle 
		-
		\left\langle \Rm_N (\diff f, \diff f) \diff f, \diff f \right\rangle\\
	&\geq
		 g_{\alpha \beta} (\Ric_M)_{ij} \diff f_i^\alpha \diff f_j ^\beta
		-
		 h_{ik} h_{jl}(\Rm_N)_{\alpha \beta\gamma\delta} 
		\diff f_i^\alpha \diff f_j^\beta \diff f_k^\gamma \diff f_l^\delta
\end{align*}
where we have used Greek indices to denote coordinate directions in $N$ and Latin indices to denote coordinate directions in $M$.

In our situation, both domain and target have sectional curvatures equal to $-1$ and so 
\[
(\Ric_\Sigma)_{ij} = - h_{ij}, \qquad (\Rm_{\H^4})_{\alpha \beta \gamma \delta} = - g_{\alpha \gamma} g_{\beta\delta} + g_{\alpha \delta} g_{\beta \gamma}.
\]
From here we see that
\begin{equation}
\frac{1}{2}\Delta |\diff f|^2 \geq - |\diff f|^2 + |\diff f|^4 - h_{ik}h_{jl}g_{\alpha\delta} g_{\beta\gamma} \diff f^\alpha_i \diff f^\beta_j \diff f^\gamma_k \diff f^\delta_l
\label{Laplacian-e-estimate}
\end{equation}
To evaluate this last term we use the fact that $f$ is conformal. Choose a holomorphic coordinate $z=p+iq$ on $\Sigma$ in which $h = F(p,q)^2(\diff p^2 + \diff q^2)$. Then, by conformality of $f$, the $p$ and $q$ derivatives of $f$ are orthogonal and have the same length:
\begin{gather*}
g_{\alpha \beta} \del_p f^\alpha \del_p f^\beta = g_{\alpha \beta} \del_q f^\alpha \del_q f^\beta = \frac{1}{2} F^{-2}|\diff f|^2\\
g_{\alpha\beta} \del_pf^\alpha\del_q f^\beta = 0.
\end{gather*}
It follows that the third term on the right-hand side of~\eqref{Laplacian-e-estimate} is
\[
- F^4 g_{\alpha\delta} g_{\beta\gamma} \left(
	\del_p f^{\alpha} \del_p f^\beta \del_p f^\gamma\del_p f^\delta
	+
	\del_q f^{\alpha} \del_q f^{\beta} \del_q f^\gamma \del_q f^\delta
	\right)
		=
			- \frac{1}{2}|\diff f|^4
\]
So
\[
\Delta |\diff f|^2 \geq |\diff f|^2\left(|\diff f|^2 - 2\right)
\]
This means that if there is an internal maximum of $|\diff f|^2$ then $|\diff f|^2 \leq 2$ at that point. Together with the fact that $|\diff f|^2 \to 2$ at $\del \Sigma$, we conclude that $|\diff f|^2 \leq 2$ on the whole of $\Sigma$. 

Given this bound on $|\diff f|^2$, the theory of harmonic maps gives uniform bounds on all higher derivatives of $f$. It's important to note that these bounds depend only on the local geometry. Since we use a hyperbolic metric on the domain, these higher order bounds do not depend on $\Sigma$, merely on the order $k$ of derivatives that we are controlling. (For example, we can pass to the universal cover $\H^2$ and use bootstrapping there.)
\end{proof}

This result shows in particular that the conformal harmonic maps cannot form bubbles. Of course, this is immediately clear on geometric grounds: a bubble would give a minimal sphere in $\H^4$  which is impossible. However it is still possible for \emph{vertical} bubbles to form in the twistor lifts. This can only happen when a branch point occurs in the map to $\H^4$, as the following Lemma shows. 

\begin{lemma}\label{no-bps-twistor-lifts-converge}
Let $f_n \colon \overline{\Sigma} \to \overline{\H}^4$ be a sequence of conformal maps which converge in $C^{2,\alpha}$ to an immersion $f$. Then the twistor lifts $u_n$ of $f_n$ converge in $C^{1,\alpha}$ to the twistor lift of $f$. 
\end{lemma}
\begin{proof}
We identify $\overline{\H}^4$ with the closed unit ball and use the flat metric there. We use any metric on $\overline{\Sigma}$ which is smooth up to the boundary. We identify the fibres of $\overline{Z}$ with unit-length elements of $\Lambda^+\R^4$.  If $v_1,v_2$ are a basis for $T_p\Sigma$, then 
\begin{equation}\label{explicit-gauss-lift}
u_n(p) = \frac{\left(\diff f_n(v_1) \wedge \diff f_n(v_2)\right)^+}{|\left( \diff f_n(v_1) \wedge \diff f_n(v_2)\right)^+|}
\end{equation}
Since $f$ is an immersion, $|\diff f| \geq \epsilon$ for some $\epsilon >0$. It follows that for all sufficiently large $n$, $|\diff f_n| \geq \frac{1}{2}\epsilon$. From this, the denominator of~\eqref{explicit-gauss-lift} is bounded away from zero, hence we can take a limit of the right-hand side in $C^{1,\alpha}$ as $n \to \infty$. 
\end{proof}

Suppose that $f$ does indeed have a branch point at $p$, then it is possible that $\diff u_n(p)$ could blow up. Looking at~\eqref{explicit-gauss-lift}, one sees that this could happen if the derivative of the numerator is bounded away from zero as $n\to \infty$. Geometrically, this means that the second fundamental form of the images of the $f_n$ is blowing up. Away from the branch point the $u_n$ will still converge, but at $p$ a vertical bubble can potentially form, giving the whole twistor fibre over $f(p)$ in the limit (or even potentially a multiple cover of this fibre).

\subsection{No branch points form near infinity}

We now show that near infinity, the conformal harmonic maps $f_n$ have energy density uniformly bounded away from zero. We use the notation of \S\ref{compactness-at-infinity}, with $\epsilon>0$ satisfying the conclusions of Theorem~\ref{convergence-near-infinity}. We pass to a subsequence, which we continue to denote $u_n$, for which the maps $f_n = \pi\circ u_n \colon \overline{\Sigma}_{n,\epsilon} \to \overline{\H}^4$ are embeddings with converging images $\overline{M}_{n,\epsilon}$. 

\begin{proposition}\label{no-bps-near-infinity}Under the hypotheses of Theorem~\ref{convergence-near-infinity}, there exists a constant $c>0$ such that for all $n$, 
\[
\inf_{\Sigma_{n,\epsilon/2}} |\diff f_n| \geq c
\]
(where we measure the norm with the hyperbolic metric on $(\Sigma, j_n)$).
\end{proposition}

\begin{proof}
For a contradiction assume the result is false. Then there exist points $z_n \in \Sigma_{n,\epsilon/2}$ with $|\diff f_n(z_n)| \to 0$. Write $x_n = x(f_n(z_n))$. By definition, $x_n \leq \epsilon/2$. There are now two possibilities: either $x_n \to 0$, or there is a subsequence for which $x_n \to \delta >0$. 

If $x_n \to 0$, then we argue in the same style as in \S\ref{compactness-at-infinity}. We use translation and homothetical rescaling in $\H^4$ to move $f_n(z_n)$ to the point $x=1, y_i=0$. The translated and rescaled minimal surfaces $\overline{M}_n$ now converge to a totally geodesic copy of $\H^2$. Meanwhile, we consider the maps $f_n$ on the geodesic ball $B(z_n;1)$ of radius 1 centred at $z_n$. After the translation and rescaling in $\H^4$ we obtain a sequence of conformal harmonic embeddings $\tilde{f}_n \colon B \to \H^4$ from a fixed hyperbolic ball, with centre $0$ say. We have $\tilde{f}_n(0) = (1,0,0,0)$ and, by Proposition~\ref{interior-estimate} uniform bounds on $\nabla^k \tilde{f}_n$ for any $k$. So, passing to a subsequence, $\tilde{f}_n$ converges to a conformal harmonic map $\tilde{f} \colon B \to \H^4$, with $\diff \tilde{f}(0)=0$, and image contained in a totally geodesic copy of $\H^2$. Since $\H^2$ is smooth, this critical point must be a ramification point of $\tilde{f}$ and so $\tilde{f} \colon B \to \tilde{f}(B)$ is a multiple cover. However, for each $n$, $\tilde{f}_n \colon B \to \tilde{f}_n(B)$ is a diffeomorphism and the images $\tilde{f}_n(B)$ converge smoothly with multiplicity~1 to $\tilde{f}(B)$ by Theorem~\ref{convergence-near-infinity}. This contradiction rules out the case $x_n \to 0$. 

The case $x_n \to \delta >0$ (after passing to a subsequence) is more involved. Let $g_n \colon M_{n,\epsilon} \to \Sigma$ denote the inverse over the interior of the embeddings $f_n \colon \overline{\Sigma}_{n,\epsilon} \to \overline{M}_{n,\epsilon}$. We equip $M_{n,\epsilon}$ with the Riemannian metric $h_n$, and hence complex structure, induced from $\H^4$. With this understood, $g_n \colon M_{n,\epsilon} \to (\Sigma,j_n)$ is a holomorphic map. Write $q_n = f_n(z_n)$. We have
\[
|\diff g_n(q_n)| = |\diff f_n(z_n)|^{-1} \to \infty
\]
The points $q_n$ lie in the compact region $\{\delta/2 \leq x \leq \epsilon/2\}$. We pass to a subsequence (still denoted by $q_n$) which converges, $q_n \to q \in M_\infty$. 

We will carry out a rescaling argument at the points $q_n$ to produce a bubble. To this end, we first claim that for each $n$ there is a neighbourhood $U_n \subset M_{n,\epsilon}$ of $q_n$, and a biholomorphism $\phi_n \colon D \to U_n$ with domain the unit disc $D \subset \C$, $\phi_n(0)=q_n$, and a metric $h_0$ on $D$  such that, in~$C^{\infty}$, 
\begin{equation}
\phi_n^*h_n \to h_0
\label{local-metrics-have-limit}
\end{equation}

To show this, we first note that there is a canonical isometry $T_qM_\infty \cong \R^2$. This is because the curve $C = \{x=\delta\} \cap M_\infty$ is oriented (it is isotopic to the boundary of $M_{\infty}$, which is an oriented link) and so $T_q C \leq T_q M_{\infty}$ gives an oriented copy of $\R$. We now orient $T_qC^\perp$ by declaring the positive side to be that on which $x$ is increasing. This completely determines the isometry $T_qM_{\infty} = T_qC \oplus T_qC^\perp \cong \R^2$. In identical fashion we have canonical isometries $T_{q_n}M_n \cong \R^2$ for each $n$. 

We now compose these linear isometries with the exponential maps to produce for each $n$ a diffeomorphism $F_n \colon \overline{B}(q,r) \to \overline{B}(q_n,r)$ where $\overline{B}(q,r)$ denotes the closed geodesic ball in $M_\infty$ centred at $q$ of radius $r$ and $\overline{B}(q_n,r)$ denotes the closed geodesic ball in $M_n$ centred at $q_n$ of radius $r$. Explicitly,
\[
F_n \colon \overline{B}(q,r) \stackrel{\exp^{-1}}{\longrightarrow} 
	T_q M_\infty \cong \R^2 \cong T_{q_n}M_n
		\stackrel{\exp_n}{\longrightarrow} \overline{B}(q_n,r)
\]
The map $F_n$ makes sense provided $r$ is smaller than the injectivity radius $i(q)$ of $M_{\infty}$ at $q$ and it is a diffeomorphism provided $r$ is also smaller than the injectivity radius $i(q_n)$ of $M_{n,\epsilon}$ at~$q_n$. Since $M_{n,\epsilon}$ converges to $M_{\infty}$ in $C^{\infty}$ near $q$, for sufficiently large $n$ we have $i(q_n) \geq i(q)/2$ and so we can choose such an $r$. (Note that we require here that the balls do not reach the interior boundaries of $M_{n,\epsilon}$ or $M_{\infty}$ at $x =\epsilon$. Since $q_n,q$ all lie in the region $x \leq \epsilon/2$, they are at least some fixed distance from $x=\epsilon$.)

Write $h$ for the Riemannian metric on $M_{\infty}$ induced from $\H^4$, $h_n$ for the induced metric on $M_n$ and $\tilde{h}_n = F_n^*h_n$. Since $M_{n,\epsilon}$ converges to $M_{\infty}$ in $C^\infty$, it follows that $\tilde{h}_n \to h$ in $C^{\infty}$. Next pick a holomorphic coordinate $\zeta \colon \overline{B}(q,r) \to \C$ with $\zeta(q)=0$, for the Riemann surface structure determined by $h$. From the Riemannian perspective, ``holomorphic'' means that $\zeta = u+ i v$ where $u, v$ are $h$-harmonic, and $\diff v = *(\diff u)$. 

We now adjust $u,v$ to produce a complex coordinate $\zeta_n$ which is $\tilde{h}_n$-holomorphic.  Let $u_n$ be $\tilde{h}_n$-harmonic, with the same boundary values as $u$. So $u_n = u + \chi_n$ where $\chi_n$ solves Poisson's equation with Dirichlet boundary condition: $\Delta_n \chi_n = - \Delta_n u$, and $\chi_n|_{\del} = 0$, where $\Delta_n$ denotes the $\tilde{h}_n$-Laplacian. There is a global Schauder estimate for $\Delta_n$:
\[
\| \chi \|_{C^{k+2, \alpha}} \leq C \left( \| \Delta_n \chi \|_{C^{k,\alpha}} + \| \chi|_{\del}\|_{C^{k+2,\alpha}}\right)
\]
Here we use the metric $h$ to define the Hölder norms. The point is that the constant $C$ can be take independent of $n$. This is because the coefficients of $\Delta_n$ converge to those of $\Delta$ because $\tilde{h}_n \to h$. From this estimate applied to $\chi_n$ we see that
\[
\| \chi_n \|_{C^{k+2,\alpha}} \leq C \| \Delta_n u \|_{C^{k,\alpha}} \to 0
\]
again since $\Delta_n$ converges to $\Delta$. So $u_n \to u$ in $C^{\infty}$. We subtract a constant to ensure also that $u_n(q) = 0$ for all $n$. (This doesn't affect the fact that $u_n \to u$ since $u(q)=0$.)

We now define $v_n$ by taking the closed 1-form $*_n (\diff u_n)$ and integrating it along the radii emanating from $q$. Here $*_n$ denotes the Hodge star of $\tilde{h}_n$. This is the unique solution to $\diff v_n = *_n(\diff u_n)$ with $v_n(q)=0$. Since $\tilde{h}_n \to h$ and $u_n \to u$, we have that $v_n \to v$ in $C^{\infty}$ also. 

Next set $\zeta_n = u_n + iv_n$. By construction, $\zeta_n$ is $\tilde{h}_n$-holomorphic with $\zeta_n(q)=0$. Since $\zeta_n \to \zeta$ in $C^{\infty}$, for all large $n$, $\zeta_n$ is a genuine coordinate on a fixed neighbourhood $U$ of $q$ in $\overline{B}(q,r)$. Moreover, there is a disc $D \subset \C$ centred on the origin which lies in the image $\zeta_n(U)$ for all large $n$. By rescaling $\zeta$ from the outset if need be, we can assume that $D$ is the unit disc.

We write $\tilde{\phi}_n = \zeta_n^{-1} \colon D \to \overline{B}(q,r)$, and $\tilde{U}_n = \tilde{\phi}_n(D)$ for the image. This is a biholomoprhism $D \to \tilde{U}_n$ with respect to the $\tilde{h}_n$-structure on $\tilde{U}_n$ and $\tilde{\phi}_n^*\tilde{h}_n \to (\zeta^{-1})^*h$. To complete the proof of our claim and the convergence~\eqref{local-metrics-have-limit}, we put $U_n = F_n(\tilde{U}_n)$ and $\phi_n = F_n \circ \tilde{\phi}_n$. Since $\tilde{h}_n = F_n^*h_n$, we see that $\phi_n$ is a biholomorphism and $\phi_n*h_n = \tilde{\phi}_n \tilde{h}_n$ which converges to a limit $h_0 = (\zeta^{-1})^*h$, as claimed. 

With this in hand, we now define a sequence of holomorphic maps $\psi_n \colon D \to \H^2$ as follows. We first consider $g_n \circ \phi_n \colon D \to (\Sigma,j_n)$ then we lift to the universal cover $\H^2 \to \Sigma$. Since $|\diff g_n(q_n)| \to \infty$, and since $\phi_n^*h_n \to h_0$, it follows that $|\diff \psi_n(0)| \to \infty$, where we use the metric $h_0$ on $D$ and the hyperbolic metric on $\H^2$ to define the norm of $\diff \psi_n$. Since $h_0$ is uniformly equivalent to the Euclidean metric, $|\diff \psi_n(0)| \to \infty$ when we use the Euclidean metric on $D$ too. We will use the Euclidean metric on $D$ from now on.

We can now find a contradiction. We use a standard point-picking trick to find a sequence $0< \epsilon_n \leq 1/4$ and points $x_n \in D$ such that 
\begin{itemize}
\item $|x_n| \leq \frac{1}{2}$.
\item $|\diff \psi_n(x_n)| \epsilon_n \geq \frac{1}{4}|\diff \psi_n(0)|$.
\item $|\diff \psi_n(y)| \leq 2 |\diff \psi_n(x_n)|$ for all $y \in D$ with $|y-x_n| \leq \epsilon_n$.
\end{itemize}
The existence of $(x_n, \epsilon_n)$ follows from the same point-picking argument used to find the sequence $(\delta_n^*,p_n^*)$ in the proof of Theorem~\ref{C2-bound-at-infinity}. We now zoom in at $x_n$ to produce the bubble. To ease the notation we write $k_n =|\diff \psi_n(x_n)|$. Define a sequence of holomorphic maps on increasingly large discs, $\hat{\psi}_n \colon D(0, k_n\epsilon_n) \to \H^2$ by $\hat{\psi}_n(w) = \psi_n(k_n^{-1}w+x_n)$. By choice of $(x_n,\epsilon_n)$, the radius $k_n\epsilon_ n\to \infty$ and we have both $|\diff \hat{\psi}_n| \leq 2$ and $|\diff \hat{\psi}_n(0)| = 1$. It follows that a subsequence converges in $C^{\infty}$ on compact subsets to a holomorphic map $\hat\psi \colon \C \to \H^2$, which is not constant, since $|\diff \hat \psi(0)|=1$.  This is a contradiction however, since the only holomorphic map $\C \to \H^2$ is constant. 
\end{proof}

This result has an important corollary. Let $\Gamma_n = (x \circ f_n)^{-1}(\epsilon/2)$. Since $x\circ f_n$ has no critical values in $[0,\epsilon]$, $\Gamma_n \subset \overline{\Sigma}$ is a smooth curve which is isotopic to the boundary. We write $l(\Gamma_n)$ for the length of $\Gamma_n$ with respect to the canonical complete hyperbolic metric on the open Riemann surface $(\Sigma, j_n)$.

\begin{corollary}\label{Gamma-bounded-length}
There exists a constant $C$ such that $l(\Gamma_n) \leq C$ for all $n$. It follows that if $[j_n]$ converges to a nodal limit, it cannot have any boundary nodes.
\end{corollary}

\begin{proof}
It follows from Theorem~\ref{convergence-near-infinity} that $f_n(\Gamma_n) \subset M_{n,\epsilon}$ converges in $C^{\infty}$ to the smooth curve $\Gamma = M_\infty \cap \{x=\epsilon/2\}$. In particular, for large enough $n$, $l(f_n(\Gamma_n)) \leq 2 l(\Gamma)$. Meanwhile,
\[
l(f_n(\Gamma_n)) \geq \left(\inf_{z \in \Gamma_n}  |\diff f|\right) l(\Gamma_n)
\]
Proposition~\ref{no-bps-near-infinity} now implies that $l(\Gamma_n) \leq \frac{1}{c} l(f_n(\Gamma_n)) \leq \frac{2}{c} l(\Gamma)$. 

Suppose now that $[j_n]$ converges to a nodal limit.  Let $\gamma_n$ be the unique geodesic in $(\Sigma,j_n)$ which is isotopic to the boundary $\del \Sigma$ (with respect to the canonical complete hyperbolic metric corresponding to $j_n$). If a boundary node occurs then $l(\gamma_n) \to \infty$, but $l(\gamma_n) \leq l(\Gamma_n) \leq C$ for all $n$, since $\gamma_n$ minimises length in its homotopy class. 
\end{proof}

\subsection{Proof of Theorem~\ref{properness-for-discs}}

We are now in position to prove Theorem~\ref{properness-for-discs}. Let $u_n \colon (\overline{D}, j_0) \to \overline{Z}$ be a sequence of $J$-holomorphic maps from the unit disc $\overline{D} = \{ z \in \C : |z| \leq 1\}$. We assume the boundary knots $K_n = \pi(u_n(S^1))$ converge in $C^{2,\alpha}$ to an embedded knot $K_\infty$. 

\begin{lemma}\label{use-up-PSL(2,R)}
There exists $r \in (0,1)$ and a sequence of biholomorphisms $\phi_n \in \PSL(2,\R)$ of $\overline{D}$ such that for all $n$, 
\[
\phi^{-1}_n(\Gamma_n) \subset \{ z \in \C : |z| < r\}
\]
\end{lemma}

\begin{proof}
By Corollary~\ref{Gamma-bounded-length}, $l(\Gamma_n)<C$ where the length is measured using the hyperbolic metric on $D$. Choose $\phi_n$ so that $\phi^{-1}_n(\Gamma_n)$ passes through the origin $0\in D$. Then $\phi^{-1}_n(\Gamma_n)$ lies entirely within hyperbolic distance $C/2$ of $0$, which means entirely within a Euclidean distance of $r<1$ from $0$, where $ r = \tanh(C/4)$. 
\end{proof}

It follows that $u_n \circ \phi_n$ sends the \emph{fixed} domain $\overline{W}= \{ z \in \C : r \leq |z| \leq 1\}$ into the subset $\{x \leq \epsilon/2\}$ where we already know the minimal surfaces converge (Theorem~\ref{convergence-near-infinity}). To ease the notation, we will now assume that we had no need to carry out this reparametrisation by $\phi_n$. I.e.~we assume that the original sequence $u_n$  already had the property that it mapped $\overline{W}$ into $\{x \leq \epsilon/2\}$. We will show that the maps $u_n$ converge here, but first we need a lemma.

\begin{lemma}\label{maps-between-annuli}
Let $A = S^1 \times [0,1]$ and $i_n, j_n$ be a sequence of $C^{1,\alpha}$ holomorphic structures on $A$ which converge in $C^{1,\alpha}$ as $n\to \infty$ to limits $i_\infty$ and $j_{\infty}$. Let $v_n \colon (A, j_n) \to (A, i_n)$ be a sequence of holomorphic maps and suppose that $v_n$ converges in $C^{2,\alpha}$ on the boundary $\del A$. Then $v_n$ converges in $C^{2,\alpha}$ to a holomorphic map $v_\infty \colon (A,j_\infty) \to (A, i_\infty)$.
\end{lemma}
\begin{proof}
We begin by uniformising the range $(A,i_n)$ in the following sense. For each $n$ there is a real number $r_n \in (0,1)$ and a holomorphic embedding $\phi_n \colon (A, i_n) \to \C$ with image $\{ z : r_n \leq |z| \leq 1\}$. Since $i_n$ is $C^{1,\alpha}$, $\phi_n$ will be $C^{2,\alpha}$. Since $i_n \to i_\infty$ in $C^{1,\alpha}$, we can arrange for the $\phi_n$ to converge in $C^{2,\alpha}$. 

Meanwhile, on the domain each $j_n$ determines a Laplacian operator. To see this, fix a positive smooth 2-form $\omega$ on $A$ and set 
\begin{equation}
\Delta_n f = \frac{\diff\left( j_n(\diff f)\right)}{\omega}
\label{Laplacian-on-RS}
\end{equation}
This is the Laplace--Beltrami operator of the metric in the conformal class determined by $j_n$ and with area form $\omega$. The equation~\eqref{Laplacian-on-RS} shows that, since the $j_n \to j_\infty$ in $C^{1,\alpha}$, the coefficients of $\Delta_n$ converge in $C^{0,\alpha}$. 

Now $\phi_n \circ v_n \colon (A,h_n) \to \C$ is a sequence of $\Delta_n$-harmonic functions, whose boundary values converge in $C^{2,\alpha}$.  Since the coefficients of $\Delta_n$ converge in $C^{0,\alpha}$ it follow that $\phi_n \circ v_n$, and hence $v_n$, converges in $C^{2,\alpha}$. Passing to the limit in the equation $\diff v_n + i_n \circ \diff v_n \circ j_n = 0$ shows that the limit is holomorphic $(A,j_\infty) \to (A, i_\infty)$. 
\end{proof}

We now return to the sequence of admissible $J$-holomorphic maps $u_n \colon \overline{D} \to \overline{Z}$, for which $\pi (u_n(S^1)) \to K_{\infty}$ in $C^{2,\alpha}$.

\begin{proposition}\label{maps-converge-near-infinity-discs}
After passing to a subsequence, the maps $f_n$ converge on $\overline{W}$ in $C^{\infty}$ on compact subsets of $r \leq |z| <1$ and in $C^{2,\alpha}$ up to the boundary $|z|=1$. The limit has no branch points and so the twistor lifts $u_n$ converge on $\overline{W}$ in the topology on the space of admissible maps.
\end{proposition} 

\begin{proof}
Firstly, by Proposition~\ref{interior-estimate}, the maps $f_n$ are bounded in $C^k$ for any $k$ on a neighbourhood of $|z|=r$ (using the hyperbolic metric on $W \subset D$). By Arzela--Ascoli, we can pass to a subsequence which converges in $C^{\infty}$ on any compact subset of $r\leq |z| <1$ and, in particular, on the inner boundary $|z|=r$. The goal is to extend this to $C^{2,\alpha}$ convergence up to the outer boundary $|z|=1$.

Recall Theorem~\ref{convergence-near-infinity}: fix a smooth knot $K$ very close to the limit $K_\infty$ and, in the region $x \leq \epsilon$, write the image of $f_n $ as a graph over $K \times [0,\epsilon]$ via a section $s_n$ of the normal bundle. By Theorem~\ref{convergence-near-infinity}, $s_n$ converges in $C^{2,\alpha}$ up to the boundary. Using $s_n$ to pull back the metric from $\overline{M}_{n,\epsilon}$ to $K \times [0,\epsilon]$, we obtain a sequence of metrics and hence complex structures on $K \times [0,\epsilon]$ which we denote by $i_n$. Since the $s_n$ converge in $C^{2,\alpha}$, it follows that the $i_n$ converge in $C^{1,\alpha}$ to a complex structure $i_\infty$ on $K\times [0,\epsilon]$. 

We now consider the maps $F_n \colon \overline{W} \to K \times [0,\epsilon]$ given by $F_n = s_n^{-1} \circ f_n$. These maps are $i_n$-holomorphic and converge near the boundary $|z|=r$. We will show that they converge in $C^{2,\alpha}$ on the whole of $\overline{W}$. Given this it follows that $f_n = s_n \circ F_n$ also converges in $C^{2,\alpha}$. To prove the convergence of the $F_n$ up to $|z|=1$, we will use a trick involving reflection to be able to apply Lemma~\ref{maps-between-annuli}.

We use ``log-polar'' coordinates: $(s,t) \mapsto e^{-(t+is)}$. This holomorphically identifies $\overline{W}$ with the cylinder $(s,t) \in S^1 \times [0, -\log r]$. By~\eqref{admissible-taylor}, the Taylor expansion of $x \circ f_n$ is odd in $t$ up to $O(t^{2+\alpha})$ and that of $y \circ f_n$ is even in $t$ up to $O(t^{2+\alpha})$. So we can extend $f_n$ to take values in the whole of $\R^4$ by reflection in $x=0$, i.e. for $t \in [\log r, 0]$, we set
\begin{align*}
x \circ f_n (s,t) &= - x \circ f_n(s,-t)\\
y \circ f_n (s,t) & = y \circ f_n(s,-t)
\end{align*}
The result is a $C^{2,\alpha}$ conformal map $S^1 \times [\log r, -\log r] \to \R^4$ that we continue to denote by $f_n$. The image of $f_n$ is a graph over the extended cylinder $K \times [-\epsilon, \epsilon]$, given by the obvious extension of the section $s_n$ by reflection to the portion over $[-\epsilon,0]$. Composing we get a map $s_n^{-1} \circ f_n \colon S^1 \times [\log r, -\log r] \to K \times [-\epsilon, \epsilon]$ that we also continue to denote by $F_n$. Meanwhile, the complex structure $i_n$ on $K \times [0,-\log r]$ is induced by the conformal structure on $\H^4$ and so we can equally use the flat metric on $\R^4$. In this way we see that $i_n$ also extends to the whole cylinder $K \times [-\epsilon, \epsilon]$. (Note that this extension is only $C^{1,\alpha}$ across $K \times \{0\}$, but this doesn't matter for our purposes.) We have that $i_n$ converges in $C^{1,\alpha}$ to a complex structure $i_\infty$ on $K \times [-\epsilon, \epsilon]$. 

We now have a sequence of maps $F_n \colon S^1\times[\log r, -\log r] \to K \times [-\epsilon,\epsilon]$ which are $i_n$ holomorphic and which converge on \emph{both} boundary components of the domain. Since $i_n \to i_\infty$ in $C^{1,\alpha}$,  it follows from Lemma~\ref{maps-between-annuli} that $F_n$ converges in $C^{2,\alpha}$ as required. 

We have now shown that $f_n$ converges on $\overline{W}$ in $C^{2,\alpha}$ up to the boundary $|z|=1$. By Proposition~\ref{no-bps-near-infinity}, $|\diff f_n| \geq c > 0$ on $\Sigma_{n,\epsilon/2}$. It follows that the same inequality holds for the limit. This means that the limit has no branch points in the interior. So $f_\infty \colon \overline{W} \to M_{\infty}$ is a proper holomorphic map of degree~1. It follows that it has no branch points on the boundary either. By Lemma~\ref{no-bps-twistor-lifts-converge}, the twistor lifts $u_n$ also converge on $\overline{W}$ as claimed.
\end{proof}

From this point, the remainder of the proof of Theorem~\ref{properness-for-discs} is straight-forward. 
\begin{proof}[Proof of Theorem~\ref{properness-for-discs}]
We use the freedom to reparametrise by biholomorphisms of $\overline{D}$ to ensure that the conclusion of Lemma~\ref{use-up-PSL(2,R)} hold for the sequence $u_n$. By Proposition~\ref{maps-converge-near-infinity-discs}, there exists $0<r<1$ such that on the set $\overline{W} = \{ z\in \C : r \leq |z| \leq 1\}$, the maps $f_n = \pi \circ u_n$ converge in $C^{2,\alpha}$ and in $C^\infty$ on the interior. Now the maps $f_n$ are bounded in $C^k$ for any $k$ on the remainder $\{ z : |z| \leq r\}$ of $\overline{D}$ (Proposition~\ref{interior-estimate}) and they converge on the boundary $|z|=r$. It follows from Arzela--Ascoli that we can pass to a subsequence which converges in $C^{\infty}$ on the whole of $D$, and in $C^{2,\alpha}$ up to the boundary. If the limit has no branch points, then the twistor lifts $u_n$ also converge in $C^{\infty}$ on $D$ and $C^{1,\alpha}$ up to the boundary (Lemma~\ref{no-bps-twistor-lifts-converge}).  
\end{proof}

\subsection{Proof of Theorem~\ref{convergence-interior}}\label{convergence-interior-proof-time}

We now turn to Theorem~\ref{convergence-interior}. We begin with the analogue of Lemma~\ref{use-up-PSL(2,R)}:

\begin{lemma}\label{collar-neighbourhood-general-case}
Under the hypotheses of Theorem~\ref{convergence-interior}, there is a collar neighbourhood $\overline{W} \subset \overline{\Sigma}$ of $\del \Sigma$, diffeomorphic to $\del \Sigma \times [0,1]$ such that, after acting by a sequence of diffeomorphisms, the $j_n$ converge on $\overline{W}$ and $\overline{W} \subset \overline{\Sigma}_{n,\epsilon}$ for all $n$; i.e., on $\overline{W}$, $x\circ u_n \leq \epsilon$ for all $n$.
\end{lemma} 
\begin{proof}
Recall the smooth curve  $\Gamma_n = (x \circ u_n)^{-1}(\epsilon/2)$, in $\Sigma$. By Corollary~\ref{Gamma-bounded-length}, the length of $\Gamma_n$, measured with the complete hyperbolic metric on $(\Sigma, j_n)$, is uniformly bounded, $l(\Gamma_n) \leq C$. We pick a boundary defining function $\tau \colon \overline{\Sigma} \to [0,\infty)$. We claim that, after acting by diffeomorphisms,  there exists $c >0$ such that for all large $n$, 
\begin{equation}
\Gamma_n \subset \{ \tau > c\}. 
\label{Gamma-cannot-escape}
\end{equation}
We can then take $\overline{W} = \{ \tau \leq c\}$. (If necessary we reduce $c$ so that the critical values of $\tau$ are all higher than $c$, this ensures $\overline{W} \cong \del \Sigma \times [0,1]$ is a genuine collar.)  To prove~\eqref{Gamma-cannot-escape} first note that if the $[j_n]$ converge to a smooth limit, then after pulling back by a sequence of diffeomorphisms, the associated hyperbolic metrics $h_n$ also converge on $\Sigma$. Meanwhile, if the $[j_n]$ converge to a nodal limit, then Proposition~\ref{Gamma-bounded-length} ensures there are no boundary nodes. So after pulling back by a sequence of diffeomorphisms, the hyperbolic metrics $h_n$ still converge near the boundary. We write this limiting metric as $h_\infty$. The metric $\tau^2 h_{\infty}$ extends smoothly up to the boundary. We use the gradient flow of $\tau$ with respect to $\tau^2h_{\infty}$ to define a collar neighbourhood $\del \Sigma \times [0,\delta]$ of $\del \Sigma$ for some $\delta>0$. In this collar, $h_\infty$ has the form
\begin{equation}
h_\infty = \frac{F(\sigma,\tau)\, \diff \tau^2 + G(\sigma,\tau)\, \diff\sigma^2}{\tau^2}
\label{h-infinity-collar}
\end{equation}
where $\sigma$ is arc length on $\del \Sigma$ with respect to the metric $(\tau^2 h_\infty)|_{\tau=0}$. Here $F$ and $G$ are smooth functions on $\del \Sigma \times [0,\delta]$. In particular they are bounded. 

Now given $0<c<\delta/2$, consider the regions $A_c =\{ (\sigma,\tau) : c\leq \tau \leq 2c \}$ and $B_c = \{(\sigma,\tau): \tau \leq 2c\}$. Let $\gamma$ be any path in $A_c$ which runs between the components of $\{\tau=c\}$ and $\{\tau = 2c\}$ of $\del A_c$ and let $\hat\gamma$ be any path in $B_c$ which is homotopic in $B_c$ to the boundary $\del \Sigma$. It follows from~\eqref{h-infinity-collar} (and the fact that $F$ and $G$ are bounded) that there is a constant $M$ such that $l(\gamma), l(\hat{\gamma}) > Mc^{-1}$. These lengths are measured with respect to $h_{\infty}$, but since the $h_n$ converge to $h_\infty$ in $B_c$ we can choose $M$ so that these lower bounds hold for lengths measured also with respect to $h_n$ for all large $n$. 

Now take $c< M/C$, where $C$ is the upper bound of $l(\Gamma_n)$. If $\Gamma_n$ were entirely contained in $B_c$ it would also be homotopic to $\del \Sigma$ in $B_c$ and so $l(\Gamma_n) >C$, a contradiction. So $\Gamma_n$ is not entirely contained in $B_c$. Suppose now that $\Gamma_n$ meets the curve $t=c$. Then it must also cross the curve $t=2c$ as it leaves $B_c$; so part of $\Gamma_n$ is a path joining two components of $\del A_c$ and this again contradicts the fact that $l(\Gamma_n)<C$. This proves $\Gamma_n \subset \{t>c\}$. 
\end{proof}

\begin{proposition}\label{maps-converge-near-infinity-general-case}
Let $\overline{W}  \cong \del \Sigma \times [0,1]$ denote the collar neighbourhood of Lemma~\ref{collar-neighbourhood-general-case}. The maps $f_n$ converge on $\overline{W}$ in $C^{\infty}$ on compact subsets of $\del \Sigma \times (0,1]$ and in $C^{2,\alpha}$ up to the outer boundary $\del \Sigma \times \{0\}$. The limit has no branch points and so the twistor lifts $u_n$ converge on $\overline{W}$ in the topology on the space of admissible maps.
\end{proposition}

\begin{proof}
The proof of this is almost identical to that of Proposition~\ref{maps-converge-near-infinity-discs}. We assume that we have acted by diffeomorphisms that $j_n$ converges on $\overline{W}$ to a limit $j_\infty$. Pick a comopnent of $\del \Sigma$ and use $j_\infty$-holomorphic cylindrical coordinates $(s,t)$ there with $s \in S^1$ and $t \in [0,\delta]$, where $t=0$ corresopnds to $\del \Sigma$. Write $\overline{U} \subset \overline{\Sigma}$ for the domain of these coordinates, where $t \in [0,\delta]$. By Lemma~\ref{collar-neighbourhood-general-case}, for small enough $\delta$ we have $\overline{U}\subset \overline{W} \subset \overline{\Sigma}_{n,\epsilon}$, so $\overline{U}$ lies in the region where the image is the graph over the cylinder $K \times [0,\epsilon]$ of a section $s_n$ of the normal bundle. 

The $j_n$ are admissible complex structures and so, by Lemma~\ref{admissible-j-give-same-maps}, the $j_n$-holomorphic maps $u_n$ have the same Taylor expansion with respect to the $j_\infty$-holomorphic coordinates $(s,t)$ to $O(t^2)$ as in \eqref{admissible-taylor}. So we can reflect $f_n$ to obtain a $C^{2,\alpha}$ map $f_n \colon S^1 \times [-\delta,\delta] \to \R^4$. Then, just as before, composing with projection to the cylinder gives a sequence of maps 
\[
F_n = s_n^{-1} \circ f_n \colon S^1 \times[-\delta, \delta] \to K \times [-\epsilon,\epsilon].
\] 
$F_n$ is holomorphic when on the domain we use the double of $j_n$ from $\overline{U} \cong S^1\times [0,\delta]$ and on the range we use the double of $i_n$ from $K \times [0,\epsilon]$. We now have a sequence of $(j_n,i_n)$-holomorphic maps between annuli where the complex structures converge in $C^{1,\alpha}$ and the boundary values converge in $C^{2,\alpha}$. Lemma~\ref{maps-between-annuli} shows the $F_n$ converge in $C^{2,\alpha}$. It follows that the $f_n$ converge in $C^{2,\alpha}$ on $\overline{U}$. We repeat this for each component of $\del \Sigma$, and then extend the convergence via the $C^k$-estimates of Proposition~\ref{interior-estimate} to the whole of $\overline{W}$. From here the rest of the proof is the same as that of Proposition~\ref{maps-converge-near-infinity-discs}.
\end{proof}

From here on the proofs of the convergence statements~\ref{smooth-limit} and~\ref{nodal-limit-of-maps} of Theorem~\ref{convergence-interior} are essentially standard. The main reference for this topic is the paper of Chen--Tian \cite{Chen-Tian}, but this is certainly overkill here. In their setting, the energy density isn't bounded and so bubbles can occur. Moreover, the harmonic maps are not assumed to be conformal. This means that in the limit, the various components may be connected by geodesics. Our situation is much simpler: there are no bubbles (at least for the harmonic maps) and no ``neck collapsing'' to geodesics. We sketch the details.

\begin{proof}[Proof of Theorem~\ref{convergence-interior}]
When the $[j_n]$ converge to a smooth limit, we can act by diffeomorphisms to arrange for the hyperbolic metrics $h_n$ to converge to a limiting metric $h_\infty$. The $C^{k}$ bound on $f_n$ then holds for this fixed metric too.  Arzela--Ascoli together with the fact that the $f_n$ converge in $C^{2,\alpha}$ on $\overline{W}$ now implies that, up to a subsequence, the $f_n$ converge on the whole of $\overline{\Sigma}$, in $C^{\infty}$ in the interior and in $C^{2,\alpha}$ up to the boundary. When the limiting map to $\H^4$ has no branch points, Lemma~\ref{no-bps-twistor-lifts-converge} implies that their twistor lifts $u_n$ converge to a $J$-holomorphic limit in the space of admissible maps. This completes the proof of part~\ref{smooth-limit} of Theorem~\ref{convergence-interior}.

Suppose now that the complex structures $[j_n]$ converge to a nodal limit, $(\overline{S},N)$. There is a continuous map $\psi \colon \overline{\Sigma} \to \overline{S}/N$, which is a diffeomorphism away from the nodes and with the pre-image of each node being the $h_n$ geodesic $\gamma_j$ which is collapsing. After pulling back by diffeomorphisms, we can arrange for the metrics $\psi_*(h_n)$ to converge on the punctured surface $S \setminus N$ to a complete hyperbolic metric $h_\infty$. The interior estimates of Proposition~\ref{interior-estimate} continue hold for the derivatives of $ f_n \circ \psi^{-1}$ defined using $h_\infty$. Together with the $C^{2,\alpha}$ convergence of the $f_n$ on $\overline{W}$, it now follows from Arzela--Ascoli that a subsequence of the $f_n \circ \psi^{-1}$ converges on compact subsets of $\overline{S} \setminus N$ to a conformal harmonic limit $f_\infty \colon \overline{S} \setminus N \to \overline{\H}^4$.  The convergence is $C^{\infty}_{\mathrm{loc}}$ on the interior and $C^{2,\alpha}$ up to the boundary. 

To extend $f_\infty$ over a point $p_i \in N$, take a small disc neighbourhood $D \subset S$ of $p_i$ which has finite area with respect to $h_\infty$. Since the energy density of $f_n$ with respect to $h_n$ is uniformly bounded (Proposition~\ref{interior-estimate}) the same is true for $f_\infty$ with respect to $h_\infty$. So $f_\infty$ has finite energy with respect to $h_\infty$ on the punctured disc $D\setminus \{p_i\}$. Since energy is conformally invariant, the energy is also finite with respect to the Euclidean metric on $D\setminus\{p_i\}$. Now Sacks--Uhlenbeck's Removable Singularity Theorem \cite{Sachs-Uhlenbeck} says that $f_\infty$ extends smoothly over $p_i$.

To prove that $f_\infty(p_i) = f_\infty(q_i)$ we argue as follows. Assume for a contradiction that $f_\infty(p_i) \neq f_\infty(q_i)$. We will show that for all large $n$, there is a cylindrical neighbourhood $R_n \subset \overline{\Sigma}$ of $\gamma_i$ such that $f_n$ sends one end very close to $f_{\infty}(p_i)$, the other very close to $f_{\infty}(q_i)$ and on which the energy of $f_n$ is very small. For a conformal map, the energy is equal to the area of the image and this will then contradict the monotonicity of area for minimal surfaces. 

To find the cylinder, write $2\pi w_n$ for the $h_n$-length of the geodesic $\gamma_i$. Near $\gamma_i$ the metric $h_n$ has a standard form:
\begin{equation}
h_n = \diff s^2 + w_n^2 \cosh^2(s) \diff \theta^2
\label{geodesic-collar}
\end{equation}
Here $s \in [-D_n, D_n]$ where $s=0$ corresponds to $\gamma_i$ and $D_n$ is the normal injectivity radius of $\gamma_i$ with respect to $h_n$. The collar lemma in hyperbolic geometry tells us that $D_n  \geq c w_n^{-1}$ for some $c$ independent of $n$. Now fix $\delta>0$ and consider the region $R_n(\delta)$ where $|s| \leq \cosh^{-1}(\delta/w_n)$. The area of this is
\[
\mathrm{Area}(R_n(\delta)) = 2 w_n\sinh (\cosh^{-1}(\delta/ w_n)) \leq C \delta
\]
for all large $n$. By Proposition~\ref{interior-estimate},  $|\diff f_n|^2 \leq 2$. Since $f_n$ is conformal, the energy is equal to the area and so the area of the image is bounded above:
\begin{equation}\label{upper-bound-area-neck}
\mathrm{Area}\left(f_n(R_n(\delta))\right) \leq 2 C \delta
\end{equation}
At the same time, the boundary $\del R_n(\delta) = S^{+}_n(\delta) \cup S^-_n(\delta)$ is a pair of circles of length $2 \pi \delta$. Now near $p_i$ the metric on $S$ also has a standard form. There are local polar coordinates in which
\begin{equation}
h_\infty = \frac{\diff r^2 + r^2 \diff \theta^2}{r^2 (\log r)^2}
\label{cusp}
\end{equation}
for $r \leq r_0$. The circle $S^+_\infty(\delta)$ given by $r= e^{-1/\delta}$ also has length $2 \pi \delta$. The convergence $\psi_*(h_n) \to h_\infty$ is such that $\psi(S^+_n(\delta))$ converges to $S^+_\infty(\delta)$.  Meanwhile $\psi(S^-_n(\delta))$ converges to the analogous circle $S^-_\infty(\delta)$ around the corresponding point $q_i$. This means that $f_n(S_n^+(\delta))$ converges to $f_\infty(S^+_\infty(\delta))$. Now $f_\infty$ is continuous, so by taking $\delta$  small and $n$ large we can ensure that $f_n(S^+_n(\delta))$ lies arbitrarily close to $f_{\infty}(p_i)$. Similarly, $f_n(S^-_n(\delta))$ lies arbitrarily close to $f_{\infty}(q_i)$.  

Now assume for a contradiction that $f_\infty(p_i)=:P_i$ and $f_\infty(q_i)=:Q_i$ are a distance $4\rho>0$ apart. By taking $n$ large and $\delta$ small we can ensure that $f_n(S^+_n(\delta))\subset B(P_i,\rho)$ and $f_n(S^-_n(\delta)) \subset B(Q_i,\rho)$. By the triangle inequality, there is a point $q$ on the image of $f_n$ with $d(q,P_i) \geq 2\rho$ and $d(q,Q_i) \geq 2\rho$. This means that $f_n(R_n(\delta))$ has no boundary in $B(q,\rho)$.  It now follows from the monotonicity of minimal submanifolds in hyperbolic space \cite{Anderson}, that 
\begin{equation}\label{lower-bound-area-neck}
\mathrm{A}\left( f_n(R_n(\delta)) \right) \geq 2\pi \left( \cosh(\rho) -1 \right) >0
\end{equation}
(where the right-hand side is the area of a hyperbolic disc of radius $\rho$). Taking $\delta$ sufficiently small,~\eqref{upper-bound-area-neck} and~\eqref{lower-bound-area-neck} contradict each other. 

It remains to prove statement~\ref{no-closed-components}, that every component of $\overline{S}$ on which $f_\infty$ is non-constant has non-empty boundary. But the image of $f_\infty$ is a minimal surface in $\H^4$ and so can't possibly be closed. This completes the proof of Theorem~\ref{convergence-interior}.
\end{proof}

\begin{remark}\label{nodes-means-df-tends-to-zero}
In the situation of part~\ref{nodal-limit} of Theorem~\ref{convergence-interior}, when nodes appear in the domain, we must have that $\inf |\diff f_n| \to 0$, even if the maps $f_n$ themselves have no branch points. To see this, note that the diameter of the ``neck'' region $R_n$ tends to infinity, and yet its image converges to a point. This means that, even when none of the maps $f_n$ nor the limit $f$ have branch points, we can't rule out the appearance of a vertical bubble at the same time as a node when we try to take a limit of the $u_n$. 
\end{remark}

\section{Knot invariants from counting minimal discs}\label{the-invariants}

Let $K \subset S^3$ be a knot. In this section we will define a collection of knot invariants $n_{d}(K)$ which count the number of minimal discs in $\H^4$ with ideal boundary equal to $K$ or, equivalently, the number of $J$-holomorphic admissible maps $u \colon \overline{D} \to \overline{Z}$ whose boundary projects to $K$, taken modulo the action of $\PSL(2,\R)$. 

\subsection{Self-linking number}\label{self-linking}
The invariants are indexed by an integer $d$, a topological invariant of the map $u$ which we call the self-linking number. The definition makes sense for any domain. 

\begin{definition}\label{defintion-self-linking-number}
Let $u \colon \overline{\Sigma} \to \overline{Z}$ be an admissible map with projection $f = \pi \circ u$. Suppose that $f$ is an immersion. Then we can write
\[
f^*T\overline{H}^4 \cong T\overline{\Sigma} \oplus E
\]
where $E\to \overline{\Sigma}$ is the normal bundle of $f$. As an $\SO(2)$-bundle over a surface with non-empty boundary, $E$ is trivialisable. Choose a  unit-length section $v$ of $E$. Since $f(\overline{\Sigma})$ meets infinity at right-angles, $v|_{\del \Sigma}$ is normal to the boundary link $f(\del \Sigma) \subset S^3$. The \emph{self-linking number of $u$} is the integer $d(u)$ given by the linking number of $f(\del \Sigma)$ and the push-off of $f(\del \Sigma)$ in the direction $v|_{\del \Sigma}$. 
\end{definition}

\begin{lemma}\label{self-linking-well-defined}~
\begin{enumerate}
\item
The self-linking number of $u$ does not depend on the choice of unit-length  section $v$ of~$E$. 
\item
If $u_t$ s a path of admissible maps for which $f_t = \pi \circ u_t$ is a path of immersions, then $d(u_0) = d(u_1)$. 
\end{enumerate}
\end{lemma}

\begin{proof}
For the first point, suppose that $\hat{v}$ is another unit-length section. Then there exists $\phi \colon \overline{\Sigma} \to S^1$ such that $\hat{v} = \phi v$. The corresponding change in the linking numbers is given by the degree of $\phi|_{\del \Sigma} \colon \del \Sigma \to S^1$. Since $\phi$ extends to $\overline{\Sigma}$ it is zero on $H_1(\del \Sigma, \Z)$. 

The second point follows from the isotopy invariance of the linking number of a pair of links. 
\end{proof}

To define $d(u)$ for \emph{any} $u \in \A$, we invoke Whitney's Immersion Theorem. Write $\A_{\mathrm{imm}}$ for the subset of $u \in \A$ for which $f=\pi \circ u$ is an immersion. Whitney's Theorem shows that the connected components of $\A_{\mathrm{imm}}$ are equal to those of $\A$. (This is proved either by an $h$-principle, or a transversality argument similar in outline to the proof of Theorem~\ref{branch-points-codim2}.) Now $d$ is constant on the connected components of $\A_{\mathrm{imm}}$ (by Lemma~\ref{self-linking-well-defined}) and so extends in a unique way to a continuous map $d \colon \A \to \Z$. Hence we can write $\A$ as a disjoint union $\A = \sqcup \A_d$ indexed by the self-linking number. In the same way, we write our moduli space of $J$-holomorphic curves as a disjoint union $\M_{g,c} = \sqcup \M_{g,c,d}$, where $\M_{g,c,d}$ is the moduli space of $J$-holomorphic curves from a connected Riemann surface of genus $g$, with $c$ boundary components and self-linking number $d$. 

(The reader who would prefer a self-contained treatment, being hesitant to rely on Whitney's Theorem in this context, can be reassured that for $\M_{g,c}$ at least, Theorem~\ref{branch-points-codim2} shows that the subset of $J$-holomorphic maps which fail to project to immersions is codimension~2. So deleting it leaves a space with the same connected components and hence $d$ is well-defined on all of $\M_{g,c}$, which is all we will actually use.)

\subsection{The minimal-disc knot-invariant}\label{definition-of-disc-invariants-section}

Let $K \subset S^3$ be an oriented $C^{2,\alpha}$ knot. In this section we define an invariant $n_{d}(K)$ of $K$ which counts minimal discs in $\H^4$ with ideal boundary $K$ and self-linking number $d$. 

Recall from~\S\ref{branch-points-codim2-section} the submanifold $\B^\pi_{0,1}\subset \M^1_{0,1}$ of \emph{branched} minimal discs, i.e.~the set of pairs $(u,p)$ where $u$ is an admissible $J$-holomorphic map from the disc and $\diff(\pi \circ u)_p=0$. Write $F \colon \M^1_{0,1} \to \M_{0,1}$ for the map which forgets the point $p$. We set 
\[
\G_{0,1} = \L_1 \setminus \beta(F(\B^\pi_{0,1}))
\]
where $\beta \colon \M_{0,1} \to \L_1$ is the boundary map. In other words, $K \in \G_{0,1}$ precisely when all minimal discs filling $K$ are immersed, i.e.\ free of branch-points. The point of singling this subset out is that, by Theorem~\ref{properness-for-discs}, $\beta$ is proper over $\G_{0,1}$. 

\begin{definition}\label{definition-of-disc-invariants}
Let $K \in \G_{0,1}$ be a regular value of $\beta$. By Theorem~\ref{properness-for-discs}, $\beta^{-1}(K)$ is a finite set of points, cut out transversely in $\M_{0,1}$. By Defintion~\ref{sign-of-regular-point} each point in $\beta^{-1}(K)$ comes with a sign.  We define the \emph{minimal-surface invariant $n_{d}(K)$} to be the signed count of the number of points in $\beta^{-1}(K) \cap \M_{0,1,d}$, i.e.\ $n_d(K)$ is the signed count of immersed minimal discs filling $K$ with self-linking number $d$.
\end{definition}

\begin{theorem}\label{invariance-of-disc-invariants}
The integer $n_{d}(K)$ is a knot invariant. I.e.\ if $K_0,K_1 \in \G_{0,1}$ are regular values of $\beta$ and they can be joined by a path $K_t$ in $\L_1$ then $n_{d}(K_0) = n_{d}(K_1)$.
\end{theorem}
\begin{proof}
Given all we have proved so far, this is a standard consequence in Fredholm differential topology. Accordingly we only sketch the details. An application of the Sard--Smale theorem shows that we can choose the connecting path $K_t$ to be transverse to $\beta$. Meanwhile, Theorem~\ref{branch-points-codim2} shows that $\B^\pi_{0,1} \subset \M^1_{0,1}$ is a submanifold of codimension~4. Now $F \colon \M^1_{0,1} \to \M_{0,1}$ is a submersion with 2-dimensional fibres and $\beta$ is Fredholm with index zero. It follows that $\beta \circ F \colon \B^{\pi}_{0,1} \to \L_1$ is Fredholm with index~-2. So we can also choose the path $K_t$ to avoid its image, i.e.\ $K_t \in \G_{0,1}$ for all $t$. By Theorem~\ref{properness-for-discs}, $\beta$ is proper over $K_t$. This means that 
\[
X = \{ ([u],t) \in \M_{0,1,d} \times [0,1]: \beta(u) = K_t \}
\]
is a smooth compact 1-dimensional manifold with boundary $\del X = \beta^{-1}(K_0) \sqcup \beta^{-1}(K_1)$. Since $K_t$ is transverse to $\beta$, for every $([u],t) \in X$, we have $\dim(\coker \diff \beta_{[u]}) \leq  1$. It follows from Lemma~\ref{corank-1-immersed} and Proposition~\ref{orientations} that $X$ lies in the region of $\M_{0,1,d}$ in which we have trivialised $\det(\diff \beta)$. Hence $X$ is an oriented cobordism from $\beta^{-1}(K_0)$ to $\beta^{-1}(K_1)$. From here we see that the signed count of the number of points in $\beta^{-1}(K_0)$ and $\beta^{-1}(K_1)$ agree. 
\end{proof}

We finish by computing the minimal disc invariants $n_{d}(U)$ of the unknot. An immediate application is the existence of a minimal disc with $d=0$ in $\H^4$ bounding any given unknotted circle in $S^3$

\begin{theorem}\label{unknot-always-filled}
The minimal-disc invariants of the unknot are
\[
n_{d}(U) 
	= 
	\begin{cases} 
	1 & \text{if }d = 0, \\
	0 & \text{otherwise}.
	\end{cases}
\]
It follows that any $C^{2,\alpha}$ unknotted circle in $S^3$ is the boundary of a minimal disc in~$\H^4$, with $d=0$. 
\end{theorem}
\begin{proof}
Let $U \subset S^3$ denote a great circle. There is a totally geodesic copy of $H \subset \H^4$ of $\H^2$ which fills $U$ and it is simple to check via barriers that there are no other minimal surfaces which fill $U$. We will show that $U$ is a regular value of $\beta$, with $d=0$. This implies the stated values of $n_{d}(U)$, at least up to sign. 

To show $U$ is regular, we use half-space coordinates $(x,y_i)$ in which $H$ corresponds to $y_2=0=y_3$. We then parametrise the twistor lift of $H$ by the map $u \colon \overline{\H}^2 \to \overline{\H}^4$ given by $x(s,t)=t$, $y_1(s,t) =s$, $z_1(s,t)=1$ and all other coordinates zero. Now recall, from the proof of Proposition~\ref{index-dbeta}, that $\diff \beta$ is surjective precisely when $D_{N_u}$ is surjective. But for the special choice of $u$ at hand here, $D_{N_u}$ is the model operator for the whole Fredholm theory and we know by Theorem~\ref{normal-invertible} that $D_{N_u}$ is invertible.

To check the sign, we first recall how this is defined (Defintion~\ref{sign-of-regular-point}). We have chosen, in Proposition~\ref{orientations}, a trivialisation of $\det (\diff \beta)$ over the part of the moduli space corresponding to immersions. Meanwhile, at a regular point $[u]$ of $\beta$, the determinant line $\det(\diff \beta_{[u]})$ is canoncially trivial, since both kernel and cokernel vanish. The sign of $u$ is given by comparing this canonical trivialisation at $u$ with the globally defined one coming from Proposition~\ref{orientations}. Now recall how the trivialisation in Proposition~\ref{orientations} is defined. We take the path $\tilde{D}_r = (1-r)\Theta^* \circ D_{N_u} + r D_{N_u}^* \circ \Theta$ of Fredholm operators starting at $\tilde{D}_0 = \Theta^* \circ D_{N_u}$ and ending at $\tilde{D}_{1/2}$. Since $\tilde{D}_{1/2}$ is self-adjoint, its kernel is isomorphic to its cokernel and so $\det (\tilde{D}_{1/2})$ is canonically trivial; we then transport this trivialisation back along the path to give a trivialisation of $\det (\Theta^* \circ D_{N_u}) \cong \det (D_{N_u})$. In the situation at hand, however, $\tilde{D}_r = \Theta^* \circ D_{N_u}$ is \emph{constant}, because $\Theta^* \circ D_{N_u} = D_{N_u}^* \Theta$. This is proved in the course of Proposition~\ref{index-DNu-no-branch-points}; it is equivalent to the fact that $M \circ \D^* \circ M = \D$ proved there. It follows that the trivialisation of $\det (\diff \beta)$ given in Proposition~\ref{orientations} agrees with the canonical one coming from the fact that $[u]$ is a regular point. Hence $[u]$ comes with a positive sign.

We now compute $d(u)$. It is a simple matter to move the totally geodesic $\H^2$ entirely off itself, whilst remaining inside a copy of $\H^3$. The corresponding push-off of $U$ lies in the corresponding $S^2= \del  \H^3 \subset S^3$ and so the self-linking number vanishes: $d(u)=0$.

Finally, suppose that $\hat{U}$ was any unknotted $C^{2,\alpha}$ circle in $S^3$. If $\hat{U}$ did not bound a minimal disc with $d=0$, it would vacuously be a regular value of $\beta$ and we could use it to deduce that $n_{0}(U)=0$, a contradiction. So $\hat{U}$ bounds a minimal disc with $d=0$ as claimed. 
\end{proof}

\bibliographystyle{amsplain}
\bibliography{plateau-in-H4-bibliography}

\vspace{\baselineskip}

\noindent
\textsc{Joel Fine\\
Département de mathématique\\
Université libre de Bruxelles\\}
\href{mailto:joel.fine@ulb.be}{\tt{joel.fine@ulb.be}}

\end{document}